\documentclass[11pt,oneside]{amsart}
\usepackage{geometry}
 
      \usepackage{amssymb}
      \usepackage{enumerate}
	  \usepackage{hyperref}

   \numberwithin{equation}{section}

      \theoremstyle{plain}
      \newtheorem{theorem}{Theorem}[section]
      \newtheorem{proposition}[theorem]{Proposition}
      \newtheorem{lemma}[theorem]{Lemma}
      \newtheorem{corollary}[theorem]{Corollary}
      
      \newtheorem{assumption}[theorem]{Assumption}

      \theoremstyle{definition}
      \newtheorem{definition}[theorem]{Definition}
      \newtheorem{example}[theorem]{Example}
      \newtheorem{remark}[theorem]{Remark}

      \newcommand{\E}{{\mathcal E}}

      \makeatletter
      \def\@setcopyright{}
      \def\serieslogo@{}
      \makeatother
      
      \usepackage{mathtools}
      \mathtoolsset{showonlyrefs,showmanualtags}
      
      \usepackage{float}
      \usepackage{tikz}
      \usepackage{tkz-berge}
      \usepackage[utf8]{inputenc}
\makeatletter
\def\keywords{\xdef\@thefnmark{}\@footnotetext}
\makeatother

\begin{document}

   \author{Matthias Hofmann}
   \address{Grupo de Física Matemática, 
	Faculdade de Ci\^{e}ncias da Universidade de Lisboa,
	Campo Grande, Edifício C6, 1749-016 Lisboa, Portugal}
	\email{mhofmann@fc.ul.pt}

   \title[An Existence Theory For Nonlinear Equations on Metric Graphs]{An existence theory for nonlinear equations on metric graphs via energy methods}

   \begin{abstract}
   \begin{tiny}
   The purpose of this paper is to develop a general  existence theory for constrained minimization problems for functionals defined on function spaces on metric measure spaces $(\mathcal M, d, \mu)$. We apply this theory to functionals defined on metric graphs $\mathcal G$, in particular $L^2$-constrained minimization problems for functionals of the form
   \begin{equation}
   E(u) = \frac{1}{2} a(u,u) - \frac{1}{q}\int_{\mathcal K} |u|^q \, \mathrm dx,
   \end{equation}
   where $q>2$, $a(\cdot, \cdot)$ is a suitable symmetric sesquilinear form  on some function space on $\mathcal G$ and $\mathcal K \subseteq \mathcal G$ is given.   
   We show how the existence of solutions can be obtained via decomposition methods using spectral properties of the operator $A$ associated with the form $a(\cdot, \cdot)$ and discuss the spectral quantities involved.  An example that we consider is the higher-order variant of the stationary NLS (nonlinear Schrödinger) energy functional with potential $V\in L^2+ L^\infty(\mathcal G)$
   \begin{equation}
   E^{(k)}(u)= \frac{1}{2} \int_{\mathcal G} |u^{(k)}|^2+ V(x) |u|^2 \, \mathrm dx - \frac{1}{p} \int_{\mathcal K} |u|^q \, \mathrm dx
   \end{equation} 
   defined on a class of higher-order Sobolev spaces $H^k(\mathcal G)$ that we introduce. When $\mathcal K$ is a bounded subgraph, one has localized nonlinearities, which we treat as a special case. When $k=1$ we also consider metric graphs with infinite edge set as well as magnetic potentials. Then the operator $A$ associated to the linear form is a Schrödinger operator, and in the $L^2$-subcritical case $2<q<6$, we obtain generalizations of existence results for the NLS functional as for instance obtained by Adami, Serra and Tilli [JFA 271 (2016), 201-223], and Cacciapuoti, Finco and Noja [Nonlinearity 30 (2017), 3271-3303], among others. 
   
\end{tiny}
   \end{abstract}
   
   \keywords{\textit{Keywords:} Quantum Graphs;  Nonlinear Schrödinger operators; Polylaplacians; Spectral Theory of Schrödinger Operators; Persson Theory; concentration-compactness techniques\\
   \textit{Mathematical Subject Classifications:} 35R02 (35J10, 35J20, 35J35, 35J60, 35Q55, 49J40, 81Q35)}
   \maketitle
\begin{tiny}
\tableofcontents
\end{tiny}
\newpage
\section{Introduction}

In recent years, there has been a growth of interest in functionals on metric graphs $\mathcal G$ of the stationary NLS (Nonlinear Schrödinger) energy functional
\begin{equation}\label{eq:introminprobfunc}
E_{\text{NLS}}(u, \mathcal G)= \frac{1}{2} \int_{\mathcal G} |u'|^2\, \mathrm dx - \frac{\mu}{q} \int_{\mathcal G} |u|^q \, \mathrm dx, \qquad u\in H^1(\mathcal G),\; \|u\|_{L^2}^2 =1,\; q>2, \; \mu >0
\end{equation}
and associated ground states of the stationary NLS energy functional, i.e. minimizers for the constrained minimization problem
\begin{equation}\label{eq:introminprob}
E_{\text{NLS}}(\mathcal G):= \inf_{\substack{u\in H^1(\mathcal G)\\ \|u\|_{L^2}^2=1}} E_{\text{NLS}}(u, \mathcal G), \qquad 2<q< 6.
\end{equation}

A metric graph (\cite{berkolaiko2013introduction}, \cite{Mu19}) is a network  of several, possibly unbounded intervals (the edge set $\mathcal E$), whose endpoints are glued together according to their correspondence in the network (the vertex set $\mathcal V$), i.e. two endpoints are glued together when they correspond to the same vertex in the underlying combinatorial graph $(\mathcal V, \mathcal E)$. The topological space admits a natural metric structure with metric $d$ based on the Euclidean metric on each interval. Introducing a Lebesgue measure $\mu$, a metric graph admits then the structure of a metric measure space $(\mathcal G, d, \mu)$. The spaces $L^p(\mathcal G)$ and $H^1(\mathcal G)$ etc. are defined in a natural way.

Minimizers of \eqref{eq:introminprob} are solutions to the stationary nonlinear Schrödinger equation on $\mathcal G$ given by
\begin{equation}
\begin{cases}
-u'' + \lambda u = \mu |u|^{q-2} u \qquad \text{edgewise,}\vspace{.5em}\\
\begin{gathered}u \text{ is continuous on }\mathcal G \text{ and satisfies the Kirchhoff condition}\\
\sum_{e\in \mathcal E:e\succ \mathsf v} \frac{\partial u}{\partial \nu} \Big |_e(\mathsf v)=0, \qquad \forall \mathsf v\in \mathcal V,
\end{gathered}
\end{cases}
\end{equation}
where $\lambda\in \mathbb R$ is a lagrange multiplier, $\mu>0$,  $e\succ \mathsf v$ denotes the relation that the edge $e$ is adjacent to the vertex $\mathsf v\in \mathcal V$ and $\frac{\partial u}{\partial \nu}|_e(\mathsf v)$ denotes the inward pointing derivative at $\mathsf v$ towards the interior of the edge~$e$. When $\mathcal G= \mathbb R$  existence of minimizers in \eqref{eq:introminprob} can be deduced by standard techniques\footnote{See  \cite[§13]{ambrosetti} for a proof based on rearrangement techniques.}. Nevertheless, on general noncompact graphs existence results are  harder to obtain due to the lack of a concept of translation invariance. In \cite{adami2015nls} it was shown on the one hand that under certain  topological configurations the problem does not admit a minimizer; on the other, in a later paper the same authors derive an existence principle based on a comparison inequality:

\begin{theorem}[\cite{adami2016threshold}]\label{thm:introast2016}
Let $\mathcal G$ be a noncompact metric graph with finitely many edges and $2<q<6$. Assume \begin{equation}\label{eq:introadamiestimate}
E_{\text{NLS}}(\mathcal G) < E_{\text{NLS}}(\mathbb R),
\end{equation} 
then there exists a minimizer for $E_{\text{NLS}}(\mathcal G)$.
\end{theorem}

This result can be used to obtain existence results on concrete graphs $\mathcal G$ via construction of so called competitors, i.e. test functions $u\in H^1(\mathcal G)$ for which $E_{\text{NLS}}(u, \mathcal G) < E_{\text{NLS}}(\mathbb R)$. This allows to deduce existence of minimizers in certain situations as shown in \cite{tentarelli2016nls} and \cite{adami2017negative}.

A variant of this problem with potential was considered in \cite{cacciapuoti2017ground} and \cite{cacciapuoti2018existence}, where the energy functional was given by
\begin{equation}\label{eq:introfunctionaltocons}
E_{\text{NLS}}^{V}(u) = \frac{1}{2} \int_{\mathcal G} \left | u'\right |^2+ V|u|^2 \, \mathrm dx -\frac{\mu}{q} \int_{\mathcal G} |u|^q \, \mathrm dx, \qquad \|u\|_{L^2}^2=1
\end{equation}
with $V\in L^1+ L^\infty(\mathcal G)$, i.e. there exist $V_1\in L^1(\mathcal G)$ and $V_\infty(\mathcal G)$ such that $V=V_1+V_\infty$.
In \cite{cacciapuoti2017ground} the existence of minimizers of \eqref{eq:introfunctionaltocons} was related to the existence of eigenvalues of the Schrödinger operator $-\Delta+V$ below the essential spectrum:

\begin{theorem}[\cite{cacciapuoti2017ground}]\label{thm:introcfn2017}
Let $\mathcal G$ be a noncompact metric graph with finitely many edges and $V\in L^1+ L^\infty(\mathcal G)$ with $V_- \in L^r(\mathcal G)$ for $r\in[1,1+ \frac{2}{q-2}]$ and $2<q\le 6$. Assume 
\begin{equation}\label{eq:introJameswants}
\inf \sigma(-\Delta+V) < \inf \sigma_{\text{ess}}(-\Delta+V_-)=0.
\end{equation}
Then there exists $\mu^*>0$ such that for $\mu \in (0, \mu^*)$ the functional \eqref{eq:introfunctionaltocons} is bounded below and the associated constrained minimization problem
\begin{equation}
E_{\text{NLS}}^{V}:= \inf_{\substack{u\in H^1(\mathcal G)\\ \|u\|_{L^2}^2=1}} E_{\text{NLS}}^V(u)
\end{equation} 
admits a minimizer.
\end{theorem}
In a sense the inequality \eqref{eq:introJameswants} replaces the inequality \eqref{eq:introadamiestimate} in Theorem~\ref{thm:introast2016} to achieve the existence results.
\begin{remark}\label{rmk:introcac2018}
 \cite{cacciapuoti2018existence} quantifies the result in Theorem~\ref{thm:introcfn2017}: given the stationary NLS ground state energy on the real line
\begin{equation}
    \gamma_q:= \inf_{\substack{u\in H^1(\mathcal G)\\ \|u\|_{L^2}^2=1}}  \frac{1}{2} \int_{\mathbb R} \left | u'\right |^2\, \mathrm dx -\frac{1}{q} \int_{\mathcal G} |u|^q \, \mathrm dx<0
\end{equation}
one can choose $\mu^*>0$ as
\begin{equation}
    \mu^*= (\Sigma_0/\gamma_q)^{\frac{3}{2}- \frac{q}{4}}.
\end{equation}
\end{remark}



Our goal in this paper is threefold. Firstly, we develop a general existence theory in a far more abstract setting which can be applied to a variety of problems as for example $E_{\text{NLS}}$ and $E_{\text{NLS}}^{V}$, providing in particular a unified approach to these problems. However, this theory is not limited to metric graphs, and may be also applied to functionals defined on function spaces on combinatorial graphs or general domains in $\mathbb R^n$.  
Secondly, we use the flexibility of  this existence theory to obtain generalizations of the results in \cite{adami2016threshold}, \cite{cacciapuoti2017ground} and \cite{cacciapuoti2018existence} in several directions by considering more general graphs and higher-order derivatives in the functionals. This paper also tackles different variants of the problems, including the case of decaying potentials and localized nonlinearities, i.e. we replace the set of integration in the term corresponding to the nonlinearity  by a bounded subgraph $\mathcal K \subset \mathcal G$, as well as a variant with magnetic potential and higher-order derivatives. Thirdly, we provide a spectral theoretical foundation for this type of existence theory.


Let us be now more precise about the abstract setting we will consider. Let $(\mathcal M, d, \mu)$ be a nonempty metric measure space. Assume $p\in [1,\infty]$ and let $X(\mathcal M)\subset L^p(\mathcal M)$  be a Banach space continuously and locally compactly imbedded in $L^p(\mathcal M)$, i.e. for any precompact, connected subset $K$, the restriction $X(K)$ is compactly imbedded in $L^p(K)$. In the case of a metric graph $\mathcal G$ a prototype would be $H^1(\mathcal G)$, but we will also apply this to higher-order Sobolev spaces $H^k(\mathcal G)$ with $k\in \mathbb N$ and $H^1(\Omega)$ for an open subset $\Omega \subset \mathbb R^N$ with $N\in \mathbb N$. 

 In §2 we establish a general existence theory for constrained minimization problems for functionals of the form
 \begin{equation}\label{eq:intfortheconstrained}
     E_c:= \inf_{\substack{u\in X(\mathcal M)\\\|u\|_{L^p}^p=c}} E(u)
 \end{equation}
with  $E\in C(X(\mathcal M), \mathbb R)$, $c>0$ and $E(0)=0$, such that the mapping 
$t\mapsto E_t$ 
is continuous for $t\ge 0$. To motivate our approach, let us briefly revisit a classical method in $\mathbb R^N$ that has served as inspiration to obtain results on metric graphs in previous works.  In general, one cannot expect existence of minimizers when $X(\mathcal M)$ is only locally compactly imbedded into $L^p$ due to the lack of globally strongly convergent subsequences. P.L. Lions  introduced in \cite{lions1984concentration} a very effective dichotomy principle based on concentration compactness for functionals defined on $\mathbb R^N$ to tackle this major difficulty.  
We will make some technical assumptions (see Definition~\ref{df:technicalassumptionsdich}, and Definition~\ref{df:definitionsuperadditivitypart}) that guarantee  a dichotomy result (Theorem~\ref{thm:main1}) for the constrained minimization problem \eqref{eq:intfortheconstrained} in the flavor of Lion's original result, as has also appeared in adapted form in \cite{adami2017negative} and \cite{cacciapuoti2017ground} in specific applications. Namely, due to the subadditivity  of the functional either the so called \emph{concentration function} for a minimizing sequence tends to zero or one has in fact existence of minimizers. 

In principle there are two strategies to recover strongly convergent minimizing sequences. Traditionally, one uses translation invariance to recover non-vanishing minimizing sequences from vanishing ones;  however, on general metric measure spaces this is not possible. The second possibility is to exclude the case of vanishing minimizing sequences altogether by other means.  
Under the correct assumptions, including the structural assumption that roughly speaking $E\in C(X(\mathcal M), \mathbb R)$ is of the form
\begin{equation}
    E(u)= \frac{1}{2} a(u,u) + \text{ nonlinear perturbation },
\end{equation}
where $a(\cdot, \cdot)$ is a suitable sesquilinear form associated to some self-adjoint operator $A$, we draw connections to spectral theoretical quantities of $A$ to exclude the case when minimizing sequences of \eqref{eq:intfortheconstrained} vanish.
In particular, this covers the problems considered in \cite{adami2017negative} and \cite{cacciapuoti2017ground}. 

More specifically, we show as a consequence of Theorem~\ref{thm:main2} (see Corollary \ref{cor:existence})  that existence holds  for ground state energies that satisfy the additional relation
\begin{equation}\label{eq:j}
    E_c < \widetilde{E}_c := \lim_{n\to \infty} \inf_{\substack{u\in X(\mathcal M)\\ \operatorname{supp} u\subset \mathcal M\setminus K_n, \; \|u\|_{L^p}^p=c}} E(u),
\end{equation}
where $K_n:= \{x\in \mathcal M|\, d(x,K)<n\}$ is the expanding ball around some precompact set $K$; this will turn out to be a generalization of \eqref{eq:introadamiestimate}. But \eqref{eq:j} has a natural spectral theoretical interpretation. In fact, given a Schrödinger operator $A=-\Delta+V$ on $\mathbb R^N$ there exists an analogous result for the linear ground state problem. As a consequence of Persson's Theorem (see for instance \cite[§14.4]{IMS1996}) ground states of $A$ exist if
\begin{equation}\label{eq:j2}
    \inf_{\substack{u\in D(A)\\ \|u\|_2^2=1}} \langle Au, u\rangle < \lim_{n\to \infty} \inf_{\substack{u\in D(A)\\ \operatorname{supp} u\subset \mathbb R^N\setminus K_n(0), \; \|u\|_{L^2}^2=1}} \langle Au, u\rangle,
\end{equation}
which is equivalent to \eqref{eq:introJameswants} (cf. §\ref{subsubsec:perssontypetheorypoly}). In our applications, we will use \eqref{eq:j2} and a perturbation argument to show \eqref{eq:j} for small nonlinearities, although in some cases we can remove or specify this restriction. 
In this context, the IMS localization formula (see \cite{sigal1982IMS}) and analogues  for similar problems which we will develop will be useful tools (see §4.3 and §5.3). We note that unlike \cite{adami2016threshold} the general existence principle does not rely on symmetrization techniques.
\medbreak

As alluded to, the functionals \eqref{eq:introminprobfunc} and \eqref{eq:introfunctionaltocons} as considered in \cite{adami2016threshold} and \cite{cacciapuoti2017ground} satisfy the prerequisites of this theory (see Example~\ref{ex:NLSclassic} and Example~\ref{ex:shouldbeputinintroduction}). In fact, one application of the existence theory constructed in §\ref{sec:A general Existence Principle} will be to a natural generalization of \eqref{eq:introfunctionaltocons}, namely the higher-order stationary NLS energy functional in §4. We will also generalize existence results on the stationary NLS energy functional with magnetic potential for general locally finite graph, using the abstract structural assumptions of the spaces considered. 
 
Let us now be more precise about the operators we investigate in this context. Given a metric graph $\mathcal G$ we define the higher-order stationary NLS energy functional
\begin{equation}\label{eq:introkfunc}
E^{(k)}(u) = \frac{1}{2} \int_{\mathcal G} |u^{(k)}|^2 + V|u|^2 \, \mathrm dx - \frac{\mu}{q} \int_{\mathcal G} |u|^q\, \mathrm dx, \qquad \begin{multlined}\mu>0,\quad 2<q<4k+2,\\
V\in L^2+L^\infty(\mathcal G)\end{multlined}
\end{equation}
and consider the ground state problem
\begin{equation}\label{eq:introkground}
E^{(k)} = \inf_{\substack{u\in H^k(\mathcal G)\\ \|u\|_{L^2}^2 =1}} \frac{1}{2} \int_{\mathcal G} |u^{(k)}|^2 + V|u|^2 \, \mathrm dx - \frac{\mu}{q} \int_{\mathcal G} |u|^q\, \mathrm dx,
\end{equation}
with $H^k(\mathcal G)$ being a higher-order Sobolev space as defined in §3.4. An application of the general existence principle on the existence of minimizers in \eqref{eq:introkground} is the following:
\begin{theorem}\label{thm:intromainnowcloser}
Let $\mathcal G$ be a noncompact metric graph with finitely many edges. Assume that either
\begin{enumerate}[(i)]
\item there exists $V=V_2+V_\infty$ such that $V_2\in L^2(\mathcal G)$ and $V_\infty \in L^\infty (\mathcal G)$ and
$$V_\infty(x)\to 0\qquad (x\to \infty)$$ 
on all edges of infinite length, or
\item $A=(-\Delta)^k+V$ admits a ground state, i.e. $\inf \sigma(A)$ is an eigenvalue. 
\end{enumerate}
Then $E^{(k)}$ is strictly subadditive, and if additionally
\begin{equation}\label{eq:intromainnowcloser}
E^{(k)} < \widetilde{E^{(k)}} := \lim_{n\to \infty}\inf_{\substack{u\in H^k(\mathcal G)\\\|u\|_{L^2}^2=1,\; \operatorname{supp} u \subset \mathcal G\setminus K_n}} E^{(k)}(u),
\end{equation}
then $E^{(k)}$ admits a minimizer.
\end{theorem}
When $k=1$ the energy functional \eqref{eq:introkfunc} reduces to the stationary NLS energy functional and we derive conditions for which the theory is applicable. Minimizers of \eqref{eq:introkground} satisfy the stationary higher-order nonlinear Schrödinger equation
\begin{equation}
\begin{cases}
(-1)^k u_e^{(2k)} + \left ( V+\lambda\right ) u_e = \mu |u_e|^{q-1} u_e, \qquad \forall e \in \mathcal E\vspace{1em}\\
\begin{multlined}
u^{(i)} \in C(\mathcal G)\quad \text{ for all } i\le 2k-1 \text{ even} \qquad \text{\it (Continuity)}\\
\qquad \land \quad\sum_{e:e\succ \mathsf v} u^{(k)}_e(\mathsf v) =0 \quad \forall i\le 2k-1 \text{ odd } \;\forall \mathsf v\in V\\
\text{\it (Kirchhoff condition)}.
\end{multlined}
\end{cases}
\end{equation}
for some Lagrange multiplier $\lambda\in \mathbb R$. While to the best of our knowledge this functional has not yet been considered on metric graphs, the stationary higher-order nonlinear Schrödinger equation on the real line of 4\textsuperscript{th} order is for instance related to traveling wave solutions of the nonlinear higher-order Schrödinger equation for the pulse envelope with higher-order dispersion  as shown in \cite[§II]{kruglov2019exact}. For combinatorial locally finite metric graphs a discussion on the existence of solutions of the nonlinear higher-order Schrödinger equation of 4\textsuperscript{th} order was  for instance considered very recently in \cite{hanshaozhao2019}.   

A minor difficulty in defining \eqref{eq:introkground} is that one needs to define higher-order Sobolev spaces $H^k(\mathcal G)$, as to date no standard way to define these spaces has emerged. We will define them in such a way that the formal Polylaplacian
\begin{equation}\label{eq:operators1}
\begin{gathered}
A=(-\Delta)^k+V\\ D(A) = H^{2k}(\mathcal G)
\end{gathered}
\end{equation}
is a self-adjoint operator on $L^2(\mathcal G)$. We remark that the choice is not necessarily unique. A discussion of self-adjoint realizations for the Bilaplacian on metric graphs can be for instance found in \cite{gregorio2017bi}.  

Analogously to \eqref{eq:introminprob} and \eqref{eq:introfunctionaltocons}, we will refer to the minimizers of $E^{(k)}$ as ground states.  Theorem~\ref{thm:intromainnowcloser} generalizes Theorem~\ref{thm:introast2016} since \eqref{eq:introminprobfunc} satisfies the prerequisites of Theorem~\ref{thm:intromainnowcloser}. Indeed, one can show with a test function argument (see Example~\ref{ex:NLSclassic}) that if $\mathcal G$ is a metric graph with finitely many edges then
\begin{equation}
\widetilde{E_{\text{NLS}}}(\mathcal G) := \lim_{n\to \infty}\inf_{\substack{u\in H^1(\mathcal G)\\\|u\|_{L^2}^2=1,\; \operatorname{supp} u \subset \mathcal G\setminus K_n}} E_{\text{NLS}}(u, \mathcal G)= E_{\text{NLS}}(\mathbb R)
\end{equation} 
and we recover Theorem~\ref{thm:introast2016}. 

Under the assumption that eigenvalues exist below the essential spectrum, i.e.
\begin{equation}
\inf \sigma((-\Delta)^k+ V) <\inf \sigma_{\text{ess}}((-\Delta)^k+V),
\end{equation}
by a perturbation argument one can ensure that \eqref{eq:intromainnowcloser} is satisfied for small nonlinearities and deduce a generalization of Theorem \ref{thm:introcfn2017}:
\begin{theorem}\label{thm:intromain1}
Let $\mathcal G$ be a noncompact metric graph with finite edge set. Let $V\in L^2+L^\infty(\mathcal G)$. Then $(-\Delta)^k+V: D((-\Delta)^k+V) \subset L^2(\mathcal G) \to L^2(\mathcal G)$ is a self-adjoint operator. Furthermore, if
\begin{equation}\label{eq:introimportantinequality}
\inf \sigma ( (-\Delta)^k +V) < \inf \sigma_{\text{ess}}((-\Delta)^k +V)
\end{equation}
then \eqref{eq:introkfunc} admits a ground state for sufficiently small $\mu>0$. 
\end{theorem}
Note that Theorem~\ref{thm:introcfn2017} also includes the critical case $q=6$, which is considered as a special case, since the functional is bounded below only for sufficiently small $\mu>0$. It is reasonable to expect that a similar result as in Theorem~\ref{thm:intromain1} holds in this particular case.

\begin{figure}[htb]
\centering
\includegraphics[scale=1]{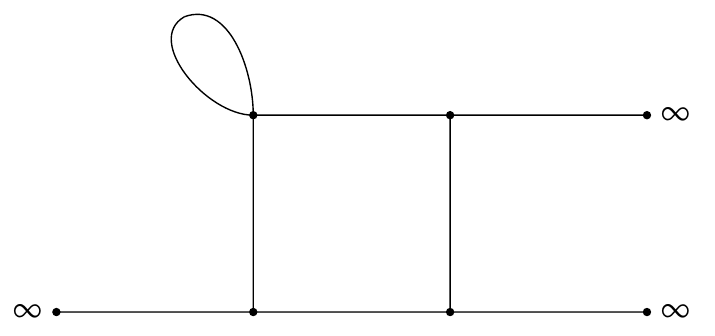}
\hspace{2em}
\includegraphics[scale=1.5]{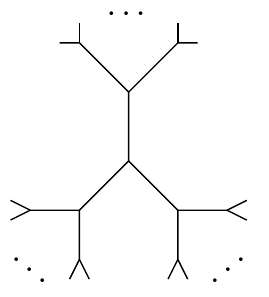}
\caption{An illustration for the classes of graphs that are considered.
To the left a finite metric graph, sometimes referred to as starlike, consisting of a core graph and attached rays 
and to the right an infinite tree graph as considered in Theorem~\ref{thm:introlastlastlast} as an example for a locally finite metric graph, i.e. finite on any precompact set.}\label{fig:introstuff}
\end{figure}

The results in Theorem \ref{thm:intromainnowcloser} and Theorem \ref{thm:intromain1} are shown for metric graphs with finitely many edges, which we refer to as finite metric graphs throughout the paper. Such graphs consist of a finite number (possibly zero) of edges of infinite length, i.e. half-lines, which we call rays, and a complement, which is compact, and which we will call the core of the graph. In \cite{cacciapuoti2017ground}, \cite{cacciapuoti2018existence} such graphs are called starlike 
(see also Figure~\ref{fig:introstuff}-left). 

Our theory also allows us to handle more general graphs, however in the case $k=1$. It remains an open question if in the general case above one can show similar existence results. If $k=1$ the minimization problem \eqref{eq:introkground} reduces to the existence of ground states of the stationary NLS energy functional. For the next result we will consider a class of graphs with countable edge set, which is finite when restricted to any precompact subset. We will refer to such graphs as locally finite metric graphs in the following. Moreover, to illustrate the scope  of our techniques, we will consider (without much extra effort) the more general situation of a magnetic Schrödinger operator.

On locally finite metric graphs we consider the following variant of the NLS energy functional
\begin{equation}\label{eq:introNLSfunc}
E_{\text{NLS}}^{(\mathcal K)}(u) = \frac{1}{2} \int_{\mathcal G} \left |\left (i \frac{\mathrm d}{\mathrm dx} + M\right )u\right |^2+ V|u|^2 \, \mathrm dx - \frac{\mu}{p} \int_{\mathcal K} |u|^q \, \mathrm dx,\qquad  \begin{multlined}\mu>0,\quad 2<q<4k+2,\\
V\in L^2+L^\infty(\mathcal G)\end{multlined}
\end{equation}
where $\mathcal K\subseteq \mathcal G$ is a subgraph of $\mathcal G$. 
In this context, we consider the magnetic Schrödinger operator with external potential
\begin{equation}\label{eq:operators2}
\begin{gathered}
A^M= \left ( i \frac{\mathrm d}{\mathrm dx} + M\right )^2+ V
\end{gathered}
\end{equation}
for $M\in H^1+ W^{1,\infty}(\mathcal G)$ and $V\in L^2+ L^\infty(\mathcal G)$ with its natural domain of definition, which we describe in detail in §5.1.

The following theorem is an analog of Theorem~\ref{thm:intromain1}. Interestingly, if one considers localized nonlinearities, i.e. $\mathcal K$ is a bounded subgraph of $\mathcal G$, then the existence result can be shown independent of the parameter $\mu>0$ in the nonlinearity:
\begin{theorem}\label{thm:intromain2}
Let $\mathcal G$ be a noncompact locally finite metric graph and $\mathcal K\subseteq \mathcal G$ a connected subgraph. Let $V\in L^2+L^\infty(\mathcal G)$ and $M\in H^1+W^{1,\infty}(\mathcal G)$. Suppose $A^M=(i \frac{\mathrm d}{\mathrm dx}+ M)^2+V$ admits a ground state that does not vanish identically on $\mathcal K$.
\begin{enumerate}[(i)]
\item If $\inf \sigma(A^M) < \inf \sigma_{\text{ess}}(A^M)$, then
\begin{equation}
E_{\text{NLS}}^{(\mathcal K)}:= \inf_{\substack{u\in H^1(\mathcal G)\\ \|u\|_{L^2}^2=1}}\frac{1}{2} \int_{\mathcal G} \left |\left (i \frac{\mathrm d}{\mathrm dx} + M\right )u\right |^2+ V|u|^2 \, \mathrm dx - \frac{\mu}{q} \int_{\mathcal K} |u|^q \, \mathrm dx
\end{equation}
admits a minimizer for sufficiently small $\mu>0$. 
\item If $\mathcal K$ is a bounded subgraph of $\mathcal G$, then minimizers exist for all $\mu >0$. 
\end{enumerate}
\end{theorem}
In §5.6 we are going to show that for a tree graph $\mathcal G$ the ground states of Schrödinger operators with magnetic potential do not vanish anywhere on $\mathcal G$. 
Then,  given a \emph{decaying potential} $V\in L^2+ L^\infty(\mathcal G)$ with $V=V_2+V_\infty$, such that $V_2\in L^2(\mathcal G)$ and $V_\infty\in L^\infty(\mathcal G)$ satisfying
\begin{equation}\label{eq:introdecaying}
\sup_{x\in \mathcal G\setminus K} |V_\infty(x)| \to 0 \qquad (n\to \infty),
\end{equation} 
we show:
\begin{theorem}\label{thm:introlastlastlast}
Let $\mathcal G$ be a noncompact locally finite tree graph with finitely many vertices of degree $1$. Suppose $M\in H^1+W^{1,\infty}(\mathcal G)$ and $V\in L^2+ L^\infty(\mathcal G)$ that satisfies \eqref{eq:introdecaying}. Then \eqref{eq:introNLSfunc} admits a minimizer if 
\begin{equation}
E_{\text{NLS}}^{(\mathcal K)} =\inf_{\substack{u\in H^1(\mathcal G)\\ \|u\|_{L^2}^2=1}}\frac{1}{2} \int_{\mathcal G} \left |\left (i \frac{\mathrm d}{\mathrm dx} + M\right )u\right |^2+ V|u|^2 \, \mathrm dx - \frac{\mu}{q} \int_{\mathcal K} |u|^q \, \mathrm dx< E_{\text{NLS}}(\mathbb R).
\end{equation}
In particular, if 
$$\inf \sigma\left ( \left ( i \frac{\mathrm d}{\mathrm dx} +M\right )^2 +V\right )<0,$$ then we have existence of minimizers of $E_{\text{NLS}}^{(\mathcal K)}$ for $0 < \mu \le (\Sigma_0/\gamma_q)^{\frac{3}{2}-\frac{p}{4}}$ as in Remark~\ref{rmk:introcac2018}.
\end{theorem}

This paper is organized as follows. In §\ref{sec:A general Existence Principle} we present the general existence theory that is the foundation for all of our results and demonstrate some basic applications to the stationary NLS energy functional on domains $\Omega\subset \mathbb R^d$, which gives a new proof of a number of results in this context.  In §\ref{sec:Sobolev spaces on graphs} we introduce higher-order Sobolev spaces and obtain inequalities on Sobolev spaces on metric graphs including variants of Sobolev and Gagliardo--Nirenberg inequalities. We also discuss basic properties of these spaces such as density results and a characterization of $W^{1,\infty}$ via uniformly bounded Lipschitz functions. In §4 we discuss the application of this existence theory to $E^{(k)}$ on finite metric graphs and prove Theorem~\ref{thm:intromainnowcloser} and Theorem~\ref{thm:intromain1}. In §4.1 we formalize the problem and show basic properties of the functional. In §4.2 we construct suitable partitions of unity and prove  a decomposition formula in §4.3, which we use to show that the existence theory is applicable. In §4.4 we prove some existence result on ground states for the higher-order stationary NLS energy functional on the real line and show in §4.5 that the existence theory is applicable for decaying potentials.  In §5 we discuss existence results for ground states of $E_{\text{NLS}}^{(\mathcal K)}$, where $\mathcal K\subseteq \mathcal G$,  on locally finite metric graphs. In §4.6 and §5.5 we discuss the energy inequality that is essential for the existence theory and relate it to spectral estimates by developing a Persson type theory for the operators in \eqref{eq:operators1} and \eqref{eq:operators2}. This will also conclude the proof of Theorem~\ref{thm:intromain2}. In particular, we discuss sufficient conditions for the potential $V$ such that \eqref{eq:introimportantinequality} is satisfied. In §5.6 we finish the paper with an application of the existence results to infinite metric trees via reduction of the problem to one without magnetic potential and prove Theorem~\ref{thm:introlastlastlast}.

Let us finish the introduction by mentioning a few other recent results on related topics. 
For a general reference on metric graphs we refer to \cite{berkolaiko2013introduction}. For a broad overview of spectral theory of operators we refer to \cite{reedmethods}. We refer to \cite{exner2018spectral} for a recent article on spectral theory for metric graphs with infinitely many edges.  The stationary energy functional
\begin{equation}
E_{\text{NLS}}^{(\mathcal K)} (u) = \frac{1}{2} \int_{\mathcal G} |u'|^2 \, \mathrm dx - \frac{\mu}{q} \int_{\mathcal K} |u|^q\, \mathrm dx, \qquad \|u\|_{L^2}^2 =1.
\end{equation}
with $\mathcal K=\mathcal G$ was considered  in \cite{adami2012stationary}, \cite{adami2015nls}, \cite{adami2016threshold}, \cite{adami2017negative} among others. A variant of the problem with localized nonlinearities in the $L^2$-subcritical case was considered in \cite{tentarelli2016nls} and for the $L^2$-critical case extended in \cite{dovetta2018ground} and \cite{dovetta20182}, where the area of integration in the nonlinearity is taken to be a bounded subgraph $\mathcal K$.  A very recent survey on results on the stationary NLS energy functional with localized nonlinearity can be found in \cite{borrelli2019overview}. Recently, classes of graphs that do not necessarily consist of finitely many edges have also been considered. For instance, \cite{dovetta2019nls} deals with a certain class of infinite tree graphs, which fall into the category of the  locally finite metric graphs that we consider here. We would also like to mention the results obtained by \cite{Akduman2019} for the NLS energy functional with growing potentials for a class of general metric graphs satisfying certain volume growth assumptions using a generalized Nehari approach.

\textbf{Acknowledgments.} I want to thank James Kennedy, who helped me shape the article and for all the helpful discussions with him. I thank Hugo Tavares, who suggested studying the NLS energy functional and provided me with references and overall knowledge on the topic and giving a few helpful suggestions. I thank Marcel Griesemer for helpful discussions as the  approach on the real line is loosely based on discussions we had in the past in the case of the NLS energy functional on the real line.  The work was supported by the Funda\c{c}\~ao para a Ci\^encia e a Tecnologia, Portugal, via the program ``Bolseiro de Investiga\c{c}\~ao'', reference PD/BD/128412/2017, and FCT project UID/MAT/00208/2019.

\section{A general Existence Principle}\label{sec:A general Existence Principle}
In this section we derive an existence theory for ground states of functionals as in \eqref{eq:introkfunc} and \eqref{eq:introNLSfunc}. To do so, we derive a more general existence principle for functionals on function spaces defined on metric measure spaces, which we will apply later to the functionals introduced before to discuss the existence of minimizers.  In §2.1 we prove a dichotomy result for minimizing sequences and discuss in §2.2 in this context the existence principle based on threshold energies under additional assumptions.

\subsection{A dichotomy result}
In the following we work with an abstract space $X(\mathcal M)$, namely a function space on a metric measure space $(\mathcal M, d, \mu)$: 
\begin{assumption}\label{as:assumption1}
Let $p\in [1,\infty)$. Let $(\mathcal M, d)$ be a metric space with a locally finite Borel measure $\mu$ on $\mathcal M$. Assume $X= X(\mathcal M)\subset L^p(\mathcal M)$ is a nontrivial Banach function space continuously and locally compactly imbedded  in $L^p(\mathcal M)$, i.e. $\mathcal M$ restricted to
$$K_R(y) := \{x\in \mathcal M|\operatorname{dist}(x,y) \le R\} $$ 
is compactly imbedded in $L^p(K_R(y))$ for all $R>0$ and $y\in \mathcal M$.
\end{assumption}

\begin{remark}
Our prototype to satisfy Assumption \ref{as:assumption1} is $X(\mathcal G)= H^1(\mathcal G)$ where $\mathcal G$ is a connected, locally finite metric graph. However, it is for instance also satisfied by $X(\Omega) = H^1(\Omega)$ for a bounded domain $\Omega \subset \mathbb R^N$ with $N\in \mathbb N$.
\end{remark}

For our results we need some further structural assumptions on the functional we consider:

\begin{definition}\label{df:technicalassumptionsdich}
Let $p\ge 2$ and let $\mathcal M$ and $X=X(\mathcal M)$ be as in Assumption \ref{as:assumption1}. Let $E \in C(X(\mathcal M), \mathbb R)$ such that $E(0)=0$ and 
\begin{equation}
E_t := \inf_{\substack{u\in X(\mathcal M)\\ \|u\|_{p}^p=t}} E(u)>-\infty
\end{equation}
for any $t\ge 0$ and $E(0)=0$. We say:
\begin{enumerate}[(1)]
\item $t\mapsto E_t$ is \emph{strictly subadditive} if
\begin{equation}
E_{t_1+t_2} < E_{t_1} + E_{t_2}, \qquad \forall t_1, t_2 >0.
\end{equation}
\item $E$ is \emph{weak limit superadditive} in $X$ if for all $c>0$ any weakly convergent minimizing sequence $u_n \rightharpoonup u$  in $X(\mathcal M)$ of $E_c$ satisfies  
\begin{equation}
\limsup_{n\to \infty} E(u_n) \ge E(u) + \limsup_{n\to \infty} E(u_n -u).
\end{equation}
up to a subsequence
\end{enumerate}
\end{definition}

\begin{theorem}\label{thm:main1}
Let $p\in [2,\infty)$, $c>0$ and let $\mathcal M$, $X=X(\mathcal M)$ be as in Assumption \ref{as:assumption1}.  Let $E\in C(X(\mathcal M), \mathbb R)$ be a weak limit superadditive functional in $X$. Let
\begin{equation}
t\mapsto E_t = \inf_{\substack{u\in X(\mathcal M)\\ \|u\|_{p}^p=t}} E(u)
\end{equation} 
be a strictly subadditive, continuous function of $t\in [0,c]$. Let $u_n$ be a minimizing sequence of $E_c$, and assume there exists $u\in X$ such that up to a subsequence $u_n \rightharpoonup u$ weakly in $X$. Then either $u\equiv 0$, or $u_n \to u$ strongly in $L^p(\mathcal M)$ and $u\not\equiv 0$ is a minimizer. 
\end{theorem}
\begin{remark}
Theorem \ref{thm:main1} gives rise to a dichotomy. If the requirements of Theorem \ref{thm:main1} are satisfied, then a minimizing sequence satisfies either $u_n \rightharpoonup 0$ in $X$ or there exists a strongly $L^p$ convergent subsequence converging to a minimizer of $E_c$. In case a minimizing sequence does not strongly converge towards a minimizer of $E_c$ from $u_n \rightharpoonup 0$ in $X(\mathcal M)$ we infer $\|u_n\|_{L^p(K)}\to 0$ on any bounded subset $K$ of $\mathcal M$. In particular, since $\|u_n\|_{p}^p=c$ for all $n\in \mathbb N$ the mass needs to move outside any compact set hence. 
\end{remark}

\begin{definition}
In virtue of Theorem \ref{thm:main1} we say a minimizing sequence of $E_c$ is \emph{vanishing} if $u_n  \rightharpoonup 0$ in $X$ and \emph{non-vanishing} otherwise.
\end{definition}
\begin{proof}[Proof of Theorem \ref{thm:main1}]
For $c>0$, suppose $u_n \in X(\mathcal M)$ be a minimizing sequence of $E_c$. Let $u\in X(\mathcal M)$, such that $u_n \rightharpoonup u$ weakly in $X$ with $u\neq 0$. Then since $u_n \to u$ in $L^{p}_{\text{loc}}$ 
we deduce $u\neq 0$ and
\begin{equation}
c\ge \|u\|_{p}^p>0.
\end{equation}
Up to a subsequence $u_n \to u$ pointwise almost everywhere, and from the Br\'ezis--Lieb Lemma (see \cite{brezis1983relation}) 
we conclude
\begin{equation}
\|u\|_p^p + \limsup_{n\to \infty} \|u_n-u\|_p^p =c.
\end{equation}
By weak limit superadditivity, strict subadditivity, and continuity of $t\mapsto E_t$ we deduce that up to a subsequence
\begin{equation}
\begin{aligned}
E_c &\ge E(u) + \limsup_{n\to \infty} E(u-u_n)\\
&\ge E_{\|u\|_p^p} + \limsup_{n\to \infty} E_{\|u-u_n\|_p^p}\\
&\ge E_{\|u\|_p^p} + E_{\limsup_{n\to \infty} \|u_n-u\|_p^p} \ge E_c.
\end{aligned}
\end{equation}
where equality is only attained when $\|u\|_p^p=c$ and $\limsup_{n\to \infty} \|u_n-u\|_p^p=0$. 
Thus $\|u\|_p^p=c$ and we conclude
\begin{equation}
E_c = E(u)
\end{equation}
and $u$ is a minimizer of $E_c$.
\end{proof}

\begin{example}[Subcritical NLS ground states]\label{ex:subcriticalR}
Let $N\in \mathbb N$. Suppose $\Omega\subset \mathbb R^N$ is a connected, unbounded, open set, then with the Euclidean metric $d$ and Lebesgue measure $\mathrm dx$ the triplet $(\Omega, d, \mathrm dx)$ defines a metric measure space. For every precompact open set $K\subset \Omega$ the Rellich--Kondrachov  theorem asserts that $H^1(K)$ compactly imbeds in $L^p(K)$ for $1\le p < p^*$ with
$$p^*:=\frac{np}{n-p}.$$  
If $N=1,2$ we have $p^*=\infty$ and $H^1(K)$ compactly imbeds to $L^p(K)$ for all $1\le p<\infty$. 
In particular, Assumption~\ref{as:assumption1} is satisfied for $1\le p < p^*$.

Consider the NLS energy functional
\begin{equation}
\begin{aligned}
    E_{\text{NLS}}(u)&:= \frac{1}{2} \int_{\Omega} |\nabla u|^2\, \mathrm dx -\frac{\mu}{q} \int_{\Omega} |u|^q\, \mathrm dx \\
    D(E_{\text{NLS}})&:= \{u\in H^1_0(\Omega) |\|u\|_{2}^2=1\}
\end{aligned}
\end{equation}
for $\mu>0$ and $2<q<2+\frac{4}{N}$.  We are going to demonstrate in the following that this functional satisfies continuity, subadditivity and weak superadditivity and that Theorem~\ref{thm:main1} is in fact applicable.

By Gagliardo--Nirenberg inequality we have
\begin{equation}
    \|u\|_q^q\le \|u\|_2^\alpha \|u\|_{H^1}^{1-\alpha}
\end{equation}
for $\alpha=\frac{N(q-2)}{2q}$. Hence for sufficiently small $\varepsilon>0$ there exists $C_{\varepsilon}>0$ such that
\begin{equation}
    E_{\text{NLS}}(u)\ge \left ( \frac{1}{2}-\varepsilon\right ) \int_{\Omega} |u'|^2\, \mathrm dx - C_\varepsilon\ge -C_\varepsilon
\end{equation}
and $E_{\text{NLS}}$ is bounded below.

Define
\begin{equation}
    t\mapsto E_t:= \inf_{\substack{u\in H_0^1(\Omega)\\ \|u\|_{2}^2=t}} E_{\text{NLS}}(u),
\end{equation}
then since $t\mapsto E(t^{1/2}u)$ is concave for each fixed $u\in D(E_{\text{NLS}})$ and $t\in (0,1)$, we deduce
\begin{equation}
    E_{\text{NLS}}(t^{1/2}u)\le t E_{\text{NLS}}(u) 
\end{equation}
and hence $E_t \le t E_1$. For a contradiction, suppose $E_t= t E_1$ for some $t\in(0,1)$, then we have
\begin{equation}\label{eq:helphere}
    E_t= t\inf_{u\in D(E_{\text{NLS}})} \frac{1}{2} \int_{\Omega} |\nabla u|^2\, \mathrm dx - t^{\frac{q-2}{2}} \frac{\mu}{q} \int_{\Omega} |u|^q\, \mathrm dx. 
\end{equation}

Let $u_n\in D(E_{\text{NLS}})$ be such that $E_{\text{NLS}}(u_n)\to E_1$. 
With \eqref{eq:helphere} we deduce
$$
\int_{\Omega} |u_n|^q\, \mathrm dx \to 0
$$
since otherwise 
$$  E_t \le \lim_{n\to \infty} E_{\text{NLS}}(t^{1/2}u_n)< \lim_{n\to \infty} tE_{\text{NLS}}(u_n)=t E_1.$$
In particular, we infer $E_1\ge 0$. Then $E_t$ is strictly subadditive if the ground state energy is negative, i.e.
\begin{equation}
E_1=\min_{u\in D(E_{NLS})} E_{\text{NLS}}(u)<0,
\end{equation}
since $E_t < t E_1$ and we have
$$E_t + E_{1-t} < E_1.$$
In fact, for sufficiently large $\mu>0$ this condition is always satisfied by a test function argument. In fact, suppose $\phi\in C_c^\infty(\Omega)$ with $\|\phi\|_2^2=1$, then for sufficiently big $\mu>0$ we have
\begin{equation}
    E_1 \le \frac{1}{2} \int_{\Omega} |\nabla \phi|^2\, \mathrm dx - \frac{\mu}{q} \int_{\Omega} |\phi|^q\, \mathrm dx<0.
\end{equation}

We remark, that in fact $t\mapsto E_t$ is concave as the infimum of concave functions and therefore in particular continous. The weak superadditivity is then an immediate consequence of the Br\'ezis--Lieb Lemma and weak convergence. In fact, there exists then a subsequence such that
\begin{equation}
\begin{aligned}
    \lim_{n\to \infty} \|u_n\|_{H^1}^2 &= \|u\|^2    +\lim_{n\to \infty} \|u-u_n\|_{H^1}^2\\
    \lim_{n\to \infty} \|u_n\|_{q}^q &= \|u\|_{q}^q + \lim_{n\to \infty} \|u-u_n\|_q^q
\end{aligned}
\end{equation}
and we have
\begin{equation}
    \lim_{n\to \infty} E(u_n) = E(u) + \lim_{n\to \infty} E(u-u_n) 
\end{equation}
for a subsequence of $u_n$.
\end{example}

\begin{example}[NLS with potential] \label{ex:subcriticalRpotential}
Suppose $V\in L^{\frac{2^*}{2^*-2}}+ L^\infty(\Omega)$, then the NLS energy functional with external potential is defined via
\begin{equation}
    \begin{aligned}
        E_{\text{NLS}}^V(u)&:= \frac{1}{2} \int_{\Omega} |\nabla u|^2 + V|u|^2\, \mathrm dx -\frac{\mu}{q} \int_{\Omega} |u|^q\, \mathrm dx \\
    D(E_{\text{NLS}}^V)&:= \{u\in H^1_0(\Omega) |\|u\|_{2}^2=1\}
    \end{aligned}
\end{equation}
for $\mu>0$ and $2<q<2+\frac{4}{N}$.
Suppose $V=V_{1}+V_2$ with $V_1\in L^{\frac{2^*}{2^*-2}}(\Omega)$ and $V_2\in L^\infty(\Omega)$ such that $\|V_1\|_{\frac{2^*}{2^*-2}}<\varepsilon$ with $\varepsilon>0$ sufficiently small. With the Hölder inequality we compute
\begin{equation}
\begin{aligned}
    \left | \int_\Omega V |u|^2\, \mathrm dx \right | &\le \|V_1\|_{L^{\frac{2^*}{2^*-2}}} \|u\|^2_{L^{2^*}} + \|V_2\|_{L^\infty} \|u\|^2_{L^2} \\
    &\le C(\Omega)\varepsilon \|u\|_{H^1}^2 + \|V_2\|_{L^\infty} \|u\|^2_{L^2}.
\end{aligned}
\end{equation}
Then as in (1) we infer that $E_{\text{NLS}}^V$ is bounded below and as in (1) we infer that
\begin{equation}
    t\mapsto E_t^V:= \inf_{\substack{u\in H_0^1(\Omega)\\ \|u\|_{2}^2=t}} E_{\text{NLS}}(u)
\end{equation}
is strictly subadditive, if 
$$E_1^V<\inf_{u\in D(E_{\text{NLS}})} \frac{1}{2}\int_\Omega |\nabla u|^2+ V|u|^2\,\mathrm dx,$$
which is satisfied for sufficiently large $\mu>0$. In fact, by the same arguments as in Example~\ref{ex:subcriticalR} we also can infer continuity and weak limit superadditivity and Theorem~\ref{thm:main1} is applicable.
\end{example}

\subsection{Vanishing sequences and Ionization Energies}
As in the previous section we consider $\mathcal M$ to be a metric measure space and $X(\mathcal M)\subset L^p(\mathcal M)$ to be a Banach space which is locally compactly imbedded in $L^p(\mathcal M)$. In the following we want to introduce partitions of unity and therefore assume the following:
\begin{assumption}\label{as:assumption2}
Let $(\mathcal M, d)$ be a metric space with locally finite Borel measure $\mu$ on $\mathcal M$ and $X(\mathcal M$ as in Assumption~\ref{as:assumption1}. Then we assume $Y(\mathcal M)$ to be a set of $\mu$ measurable functions on $\mathcal M$ such that $X(\mathcal M)$ is invariant with respect to multiplication of elements in $Y(\mathcal M)$
\end{assumption}

In this section we show that the existence  of vanishing sequences gives a bound from below on the ground state energy $E_c$, 
which allows us, under stronger assumptions, to deduce an existence result from Theorem~\ref{thm:main1}. Let us first introduce partitions of unity on metric spaces.

\begin{definition}
Let $Y(\mathcal M)$ be as in Assumption \ref{as:assumption2}. Assume $\cup_{O\in \mathcal O} O=\mathcal G$ is a locally finite open covering $\mathcal O$ of $\mathcal M$. Then we say a family of nonnegative functions $\psi_O\in Y(\mathcal G)$ is a partition of unity subordinate to $\mathcal O$ if
\begin{equation}
\operatorname{supp} \psi_O \subset O, \quad \forall O\in \mathcal O \quad\land \quad \bigcup_{O\in \mathcal O} \operatorname{supp} \psi_O = \mathcal G \quad \land \quad 0 \le \psi_O \le 1
\end{equation}
and $\sum_{O\in \mathcal O} \psi_O(x)\neq 0$ for all $x\in \mathcal M$ and
$$
\Psi_O(x)= 1, \qquad \forall x\in \operatorname{supp} \Psi_O \setminus \bigcup_{\widehat O\in \mathcal O \setminus \{O\}} \operatorname{supp} \Psi_{\widehat{O}}.
$$
\end{definition}

In the following we define for $R> 0$ the open and closed $R$-neighborhoods of a subset $K\subset \mathcal M$ by
\begin{equation}\label{eq:core}
\begin{aligned}
K_{R} &:= \{x\in \mathcal M| \operatorname{dist} (x, K)< R\}\\
\overline{K_{R}} &:= \{x\in \mathcal M|\operatorname{dist} (x, K) \le R\}.
\end{aligned}
\end{equation}

Given a vanishing sequence, the following property of a functional characterizes decomposability with regards to sequences of partition of unity with increasing core:

\begin{definition}
Let $k \in \mathbb N$ and let $\mathcal M$ be a metric space. Let $K$ be a bounded subset of $\mathcal M$ and $K_n$ be defined by \eqref{eq:core} for $n\in \mathbb N$. We say a sequence of open coverings $\mathcal O_n=\{O_n^{(1)}, \ldots, O_n^{(k)}\}$ consisting of $k$ open subsets (not necessarily connected) is \emph{vanishing-compatible}, if
\begin{equation}
K_n\cap O_n^{(i)}= \emptyset, \qquad \forall i\in \{2,\ldots, k\}
\end{equation}
and $O_n^{(1)}$ is bounded for all $n$.
\end{definition}

In particular, $K_n \subset O_n^{(1)}$. That is, for a sequence of open coverings $\mathcal O_n=\{O_n^{(1)}, \ldots, O_n^{(k)}\}$ all its members except $O_n^{(1)}$ move away from $K$. Furthermore, this notion does not depend on the choice of $K$, i.e. up to a subsequence any sequence of open coverings is vanishing-compatible for any other $K$.

\begin{definition}\label{df:definitionsuperadditivitypart}
Let $k\in \mathbb N$ and $\mathcal O_n=\{O_n^{(1)}, \ldots, O_n^{(k)}\}$ be a vanishing-compatible sequence of open coverings.  Then we say  $E\in C(X(\mathcal M), \mathbb R)$ is $k$-superadditive with respect to a sequence of partitions of unity $$\{\psi_{O}\}_{O\in \mathcal O_n}=\left \{\psi_{O_n^{(1)}},\ldots, \psi_{O_n^{(k)}}\right \}$$ 
if for any vanishing sequence $(v_n)$,  
\begin{equation}
\limsup_{n\to \infty} E(v_n) \ge \sum_{i=1}^k \limsup_{n\to \infty} E(\psi_{O_{n}^{(i)}} v_n)
\end{equation}
up to a subsequence.
\end{definition}

\begin{example}\label{ex:subcriticalRoh}
Let $n\in \mathbb N$ and $N\in \mathbb N$. Suppose $\Omega\subset \mathbb R^N$ is a connected, unbounded, open set, then with the Euclidean metric $d$ and Lebesgue measure $\mathrm dx$ the triplet $(\Omega, d, \mathrm dx)$ defines a metric measure space. If $X(\Omega)= H^k_0(\Omega)$ and $Y(\Omega)= W^{k,\infty}(\Omega)$, then $X(\Omega)$ is invariant by multiplication of elements in $Y(\Omega)$. In fact, for $f\in H^1(\Omega)$ and $g\in W^{1,\infty}(\Omega)$ by the product rule we have $fg\in H^1(\Omega)$ and
\begin{equation}
\nabla (fg)= (\nabla f) g+ (\nabla g) f.
\end{equation}
Let $\Psi\in C^\infty(\mathbb R)$ with $\operatorname{supp} \Psi\subset [-2,2]$, such that $0\le \Psi \le 1$ and $\Psi \equiv 1$ on $[-1,1]$. Consider the open covering $\mathcal O$ defined by $K_{2n}(0), \Omega\setminus K_n(0)$, then we can define a partition of unity subordinate to $\mathcal O$ given by
\begin{equation}\label{eq:unityclassic}
    \Psi_n(x) := \Psi\left (\frac{\|x\|}{n}\right ),\qquad  \widehat{\Psi}_n := 1- \Psi_n
\end{equation}
and by construction $\Psi_n+\widehat{\Psi}_n\equiv 1$.
\end{example}
This gives rise to our second main result: 
\begin{theorem}\label{thm:main2}
Let $p\in [2,\infty)$, $c>0$ and let $(\mathcal M, \mu)$, $X(\mathcal M)$ and $Y(\mathcal M)$ satisfy Assumption \ref{as:assumption1} and Assumption \ref{as:assumption2}. Let $K$ be a bounded, connected, nonempty set in $\mathcal M$. Let $E\in C(X(\mathcal M), \mathbb R)$, such that
\begin{equation}
t\mapsto E_t = \inf_{\substack{u\in X(\mathcal M)\\ \|u\|_{p}^p=t}} E(u)
\end{equation} 
is continuous and assume $E$ to be $2$-superadditive with respect to a sequence of partitions of unity $\{\psi_{O}\}_{O\in \mathcal O_n}$ in $Y(\mathcal M)$ subordinate to a vanishing-compatible sequence of open coverings $\mathcal O_n=(O_1^{(n)}, O_2^{(n)})$. If there exists a minimizing sequence which is vanishing, then 
\begin{equation}
E_c = \lim_{R\to \infty} \inf_{\substack{u\in X(\mathcal M), \|u\|_{p}^p=c\\ \operatorname{supp} u \subset \mathcal M \setminus K_R}} E(u) =: \widetilde{E_c}.
\end{equation}
\end{theorem}
\begin{proof}
Let $u_n$ be a vanishing sequence. Assume $(O_n^{(1)}, O_n^{(2)})$ to be such that
\begin{equation}
K \subset O_n^{(1)}
\end{equation} 
and $O_n^{(1)}$ is bounded.

For each fixed $m\in \mathbb N$ we have
\[
\int_{O_m^{(1)}} |u_n|^p\, \mathrm d\mu \to 0 \qquad (n\to \infty).
\]
Then for any $m\in \mathbb N$ we find an $n_m$, such that for $n>n_m$
\[
\int_{O_m^{(1)}} |u_n|^p \, \mathrm d\mu \le \frac{1}{m}. 
\]
Using a diagonal argument we deduce the existence of a subsequence of $u_n$, still denoted by $u_n$, such that
\[
\int_{O_n^{(1)}} |u_n|^p \, \mathrm d\mu \to 0 \qquad (n\to \infty).
\]
In particular,
\begin{equation}
\begin{aligned}
0\le \int_{O_n^{(1)}} |\psi_{O_n^{(1)}}u_n|^p\, \mathrm d\mu &\le \int_{O_n^{(1)}} |u_n|^p\, \mathrm d\mu\\
c -\int_{O_n^{(1)}} |u_n|^p\, \mathrm d\mu \le    \int_{O_n^{(2)}} |\psi_{O_n^{(2)}}u_n|^p \, \mathrm d\mu&\le \int_{\mathcal M} |u_n|^p\, \mathrm d\mu =c
\end{aligned}
\end{equation}
and we obtain
\begin{equation}
\begin{aligned}
\int_{O_n^{(1)}} |\psi_{O_n^{(1)}}u_n|^p \, \mathrm d\mu &\to 0 \qquad (n\to \infty)\\
\int_{O_n^{(2)}} |\psi_{O_n^{(2)}}u_n|^p \, \mathrm d\mu &\to c \qquad (n\to \infty).
\end{aligned}
\end{equation}
Then by superadditivity we have
\begin{equation}
\begin{aligned}
E_c &= \lim_{n\to \infty} E(u_n) \\
&\ge \limsup_{n\to \infty} E\left (\psi_{O_{n}^{(2)}} u_n\right )\ge  \widetilde {E_c}.
\end{aligned}
\end{equation} 
This concludes the inequality $E_c \ge \widetilde{E_c}$. The reverse inequality is trivial since
\begin{equation}
E_c \le \inf_{\substack{u\in X(\mathcal M), \|u\|_{p}^p=c\\ \operatorname{supp} u \subset \mathcal M \setminus K_R}} E(u)
\end{equation}
for all $R>0$.
\end{proof}
\begin{corollary}\label{cor:existence}
Suppose the assumptions in Theorem \ref{thm:main1} and Theorem \ref{thm:main2} are satisfied and
\begin{equation}
E_c <\widetilde{E_c},
\end{equation}
then a minimizer of $E_c$ exists, and any minimizing sequence for $E_c$ admits a subsequence converging in $L^p$ towards a minimizer of $E_c$.
\end{corollary}

Throughout the rest of the paper, given a functional $E\in C(X(\mathcal G), \mathbb R)$, we define the corresponding threshold energy
\begin{equation}\label{eq:defion}
\widetilde{E_c} := \lim_{R\to \infty} \inf_{\substack{u\in X(\mathcal G), \, \|u\|_{p}^p=c\\ \operatorname{supp} u \subset \mathcal G\setminus K_R}}E(u).
\end{equation}
In the case of many-body quantum particle systems, a quantity similar to \eqref{eq:defion} refers to the ionization energy (see \cite{griesemer2004exponential} and \cite{simon1983semiclassical}). For this reason, throughout of the rest of the paper we are going to refer to the quantity in \eqref{eq:defion} also as the \emph{ionization threshold} or \emph{ionization energy}.

\begin{example}\label{ex:NLSclassicR}
Let $N\in \mathbb N$. Suppose $\Omega\subset \mathbb R^N$ is an open, unbounded domain.  Recall the stationary NLS energy functional in Example~\ref{ex:subcriticalRpotential}
\begin{equation}
    \begin{aligned}
        E_{\text{NLS}}^V(u)&:= \frac{1}{2} \int_{\Omega} |\nabla u|^2 + V|u|^2\, \mathrm dx -\frac{\mu}{q} \int_{\Omega} |u|^q\, \mathrm dx \\
    D(E_{\text{NLS}}^V)&:= \{u\in H^1_0(\Omega) |\|u\|_{2}^2=1\}
    \end{aligned}
\end{equation}
with $V\in L^{\frac{2^*}{2^*-2}}+L^\infty$ and $2<q<4+\frac{2}{N}$ as in Example~\ref{ex:subcriticalRpotential}. We consider the ground state problem
\begin{equation}
E_{\text{NLS}}^V= \inf_{u\in D(E_{\text{NLS}}^V)} E_{\text{NLS}}^V(u).
\end{equation}
We already showed that $E_{\text{NLS}}^V$ satisfies the preqrequisites of Theorem~\ref{thm:main1} in Example~\ref{ex:subcriticalRpotential}. Now to show that the prerequites of Theorem~\ref{thm:main2} are satisfied it suffices to show superadditivity with respect to a vanishing-compatible sequence of partitions of unity.

Suppose $u_n\in C_c^\infty(\Omega)$ is a vanishing sequence, then there exists a subsequence, still denoted by $u_n$ with abuse of notation, such that 
\begin{equation}\label{eq:NLSclassicRhelp}
    \int_{K_2n} |u_n|^q\to 0\qquad (n\to \infty).
\end{equation}
Recall $\Psi_n, \overline\Psi_n$ as in Example~\ref{ex:subcriticalRoh}, then the IMS formula\footnote{According to \cite{simon1983semiclassical} due to  Israel Michael Sigal.} states that
\begin{multline}
    \left (-\Delta +V\right ) u= \frac{\Psi_k}{\sqrt{\Psi_k^2+\widehat{\Psi}_k^2}} \left (-\Delta +V\right ) \frac{\Psi_k}{\sqrt{\Psi_k^2+\widehat{\Psi}_k^2}} u \\
    + \frac{\overline \Psi_k}{\sqrt{\Psi_k^2+\widehat{\Psi}_k^2}} \left (-\Delta +V\right ) \frac{\overline \Psi_k}{\sqrt{\Psi_k^2+\widehat{\Psi}_k^2}} \\
    +  \left |\frac{\mathrm d}{\mathrm dx}\frac{\Psi_k}{\sqrt{\Psi_k^2+\widehat{\Psi}_k^2}}\right |^2 u + \left |\frac{\mathrm d}{\mathrm dx}\frac{\overline \Psi_k}{\sqrt{\Psi_k^2+\widehat{\Psi}_k^2}}\right |^2 u
\end{multline}
For all $u\in C_c^\infty(\Omega)$ we compute with integration by parts
\begin{equation}\label{eq:NLSclassicRhelp2}
    E_{\text{NLS}}^V(u) = \frac{1}{2}\langle \left (-\Delta+ V\right ) u, u\rangle_{L^2}- \frac{\mu}{q} \|u\|_q^q.
\end{equation}
In particular, with \eqref{eq:NLSclassicRhelp} and \eqref{eq:NLSclassicRhelp2} we have
\begin{equation}
    \lim_{n\to \infty} E_{\text{NLS}}^V(u_n) = \lim_{n\to \infty} E_{\text{NLS}}^V\left ( \frac{\Psi_n}{\sqrt{\Psi_n^2+ \overline \Psi_n^2}} u_n \right ) + \lim_{n\to \infty} E_{\text{NLS}}^V \left ( \frac{\overline \Psi_n}{\sqrt{\Psi_n^2+ \overline \Psi_n^2}} u_n \right )
\end{equation}
Suppose $V\in L^{2^*}+L^\infty(\Omega)$ is a decaying potential, i.e. $\sup_{|x|\ge R}|V_\infty|\to 0$ as $R\to \infty$, then  
\begin{equation}
    \widetilde{E_{\text{NLS}}^V} = \widetilde{E_{\text{NLS}}^0}\ge E_{\text{NLS}}
\end{equation}
and with Corollary~\ref{cor:existence} we deduce existence of minimizers if
\begin{equation}
    E_{\text{NLS}}^V< E_{\text{NLS}}.
\end{equation}
Suppose $\Omega=\mathbb R^N$. Then by a translation argument we can further characterize the ionization threshold and
\begin{equation}
    \widetilde{E_{\text{NLS}}^0} =E_{\text{NLS}}(\mathbb R^N).
\end{equation}
\end{example}

\subsection{An existence result for translation-invariant functionals on strip type spaces}
In this section we are going to study general translation-invariant functionals defined on function spaces $X(\mathcal M)$ satisfying Assumption~\ref{as:assumption1} on strip type spaces, which we define as follows

\begin{definition}
We say $(\mathcal M, d, \mu)$ is a strip type metric measure space, if there exists a measure space $(\mathcal M', \mathrm dy)$ such that 
\begin{equation}
\begin{aligned}
    \mathcal M&= \mathbb R \times \mathcal M'\\
    \mu &= \mathrm dx \otimes \mathrm dy
\end{aligned}
\end{equation}
such that $I\times \mathcal M'\subset \mathcal M$ is precompact for each finite interval $I\subset \mathbb R$.
\end{definition}

\begin{theorem}\label{thm:main3}
Let $p\in [2,\infty)$, $c>0$, and $\mathcal M= \mathbb R\times \mathcal M'$ be a strip type metric measure space and satisfy Assumption \ref{as:assumption1} and Assumption \ref{as:assumption2}. Suppose $E\in C(X(\mathcal M), \mathbb R)$ is translation invariant, i.e. if $T_{\lambda} u(x,y)= u(x-\lambda,y)$ then  
$$E(u)= E(T_\lambda u)$$
for all $\lambda \in \mathbb R$. Let
\begin{equation}
t\mapsto E_t = \inf_{\substack{u\in X(\mathcal M)\\ \|u\|_p^p =t}} E(u)
\end{equation}
be a strictly subadditive functional in $X(\mathcal M)$. Assume $E$ to be superadditive with respect to a sequence of $3$-partitions of unity 
$\{\psi_O\}_{O\in \mathcal O_n}$ subordinate to the vanishing-compatible sequence of open coverings $$ \mathcal O_n = \{(-2n,2n)\times \mathcal M', (n, \infty)\times \mathcal M', (-\infty, -n)\times \mathcal M'\},$$ then it admits a minimizer.
\end{theorem}
\begin{proof}
 By Theorem \ref{thm:main2} we only need to construct non-vanishing minimizing sequences. Assume $u_n$ to be a minimizing sequence of the functional $E_c$. Then we may construct such a sequence by using the translation invariance of the functional. Indeed, we may assume up to translation invariance
\begin{equation}
\begin{aligned}
\int_{\mathcal M'}\int_{0}^\infty |u_n|^p\, \mathrm dx\, \mathrm dy &= \frac{c}{2}\\
\int_{\mathcal M'}\int_{-\infty}^0 |u_n|^p\, \mathrm dx\, \mathrm dy &= \frac{c}{2}.
\end{aligned}
\end{equation}
For a contradiction, assume $u_n$ is vanishing. Then since $u_n \to 0$ in $L^\infty_{\text{loc}}$ (up to a subsequence) due to a diagonal argument, we have that
\begin{equation}
\begin{aligned}
\int_{\mathcal M'}\int_{\mathbb R} |\psi_{(-2n,2n)} u_n|^p \, \mathrm dx\, \mathrm dy &\to 0 \qquad (n\to \infty) \\
\int_{\mathcal M'}\int_{\mathbb R} |\psi_{(n, \infty)} u_n|^p \, \mathrm dx\, \mathrm dy &\to \frac{c}{2}\qquad (n\to \infty) \\
\int_{\mathcal M'}\int_{\mathbb R} |\psi_{(-n, -\infty)} u_n|^p \, \mathrm dx\, \mathrm dy &\to \frac{c}{2} \qquad (n\to \infty)
\end{aligned}
\end{equation}
Then using the subadditivity of the functional and the strict subadditivity of $t\mapsto E_t$ we conclude
\begin{equation}
\begin{aligned}
E_c &= \lim_{n\to \infty} E(u_n) \\
&\ge \limsup_{n\to \infty} E\left ( \psi_{(-\infty, -n)} u_n\right ) + \limsup_{n\to \infty} E\left (\psi_{(n, \infty)} u_n\right ) \\
&\ge E_{c/2} + E_{c/2} > E_c.
\end{aligned}
\end{equation}
By contradiction after translating the $u_n$ if necessary we can find a non-vanishing subsequence. Passing to a further subsequence there exists a weakly convergent subsequence in $H^1(\mathcal G)$ that converges up to a further subsequence to a minimizer by Theorem \ref{thm:main1}.
\end{proof}

\begin{example}\label{ex:NLSclassicstripe}
Suppose $\Omega=\mathbb R\times [0,1]^N$ and $V(\cdot, y)\equiv V(y)$.  By Example~\ref{ex:subcriticalRpotential} and Example~\ref{ex:NLSclassicR} the NLS energy functional
\begin{equation}\label{eq:NLSclassicstripeeq}
    \begin{aligned}
        E_{\text{NLS}}^V(u)&:= \frac{1}{2} \int_{\Omega} |\nabla u|^2 + V|u|^2\, \mathrm dx\, \mathrm dy -\frac{\mu}{q} \int_{\Omega} |u|^q\, \mathrm dx\, \mathrm dy \\
    D(E_{\text{NLS}}^V)&:= \{u\in H^1_0(\Omega) |\|u\|_{2}^2=1\},
    \end{aligned}
\end{equation}
with $V\in L^{\frac{2^*}{2^*-2}}+L^\infty(\mathcal G)$ and $2<q<2+ \frac{4}{N}$, satisfies the prerequisites of Theorem~\ref{thm:main3} and we have existence of ground states of \eqref{eq:NLSclassicstripeeq} for all $\mu>0$.

To see this, suppose $\phi \in C_c^\infty(\Omega)$ such that $\|\phi\|_2^2=1$, then we rescale to obtain test functions
\begin{equation}
\phi_{\lambda}(x,y):= \lambda^{1/2} \phi(\lambda x, y)
\end{equation}
and we compute
\begin{equation}\label{eq:prething}
    E_{NLS}^V(\phi_{\lambda})= \frac{1}{2} \int_{\Omega} |\nabla_y \phi|^2+V|\phi|^2\,\mathrm dx\, \mathrm dy + \lambda^2\int_{\Omega} |\nabla_x \phi|^2 - \lambda^{\frac{q-2}{2}} \int_{\Omega} |\phi|^q\, \mathrm dx\, \mathrm dy
\end{equation}

Suppose $u_2$ is the minimizer of
\begin{equation}
    \lambda([0,1]^N):=\min_{\substack{u\in H^1([0,1]^N)\\\|u\|_2^2=1}} \int_{[0,1]^N} |\nabla u|^2+V |u|^2\, \mathrm dy
\end{equation}
whose existence can be shown for instance by the direct method in the calculus of variations, then suppose $\phi(x,y):= u_1(x) u_2(y)$ with $x\in \mathbb R$ and $y\in [0,1]^N$, then with \eqref{eq:prething} we compute
\begin{equation}\label{eq:postthing}
\begin{aligned}
    E_{NLS}^V(\phi_{\lambda})&= \frac{1}{2} \int_{[0,1]^N} |\nabla u_2|^2+V|u_2|^2\, \mathrm dy + \lambda^2\int_{\mathbb R} |\nabla u_1|^2\, \mathrm dx - \lambda^{\frac{q-2}{2}} \int_{\Omega} |\phi|^q\, \mathrm dx\, \mathrm dy\\
    &< \frac{\lambda([0,1]^N)}{2} = \frac{1}{2} \inf_{\substack{u\in H^1([0,1]^N\setminus \{0\}}} \frac{\int_{[0,1]^N} |\nabla u|^2+V |u|^2\, \mathrm dy}{\int_{[0,1]^N} |u|^2\, \mathrm dy}  \\
    &\le \frac{1}{2} \inf_{\substack{u\in H^1(\Omega)\\ \|u\|_2^2=1}} \int_{\mathbb R}\int_{[0,1]^N} |\nabla u|^2+ V|u|^2\, \mathrm dx\, \mathrm dy.
\end{aligned}
\end{equation}
for sufficiently small $\lambda>0$. In particular, $E_{\text{NLS}}^V$ is strictly subadditive for all $\mu>0$ due to the arguments in Example~\ref{ex:subcriticalR}.
\end{example}


\section{Sobolev spaces on graphs} \label{sec:Sobolev spaces on graphs}
In this section we define metric graphs and define Sobolev spaces on metric graphs and prove Sobolev inequalities on these spaces.
\subsection{Metric graphs}
Let $\mathcal G=(\mathcal V,\mathcal E)$ be a  metric graph, where each edge $e\in \mathcal E$ is associated with an interval $I_e$ of length $l_e\in (0,\infty]$, where $I_e=[0,\infty)$ if $l_e=\infty$ and $I_e=[0, l_e]$ otherwise. We assume every vertex to be at least of degree one.  For every $e\in \E$ joining two vertices we arbitrarily associate one of the vertices to $0$ and the other to $l_e<\infty$ respectively on the interval $I_e$; in principle this imposes an orientation on the graph, but all the spaces and operators we consider will be unaffected by this. However, we always assume that the half-line $I_e=[0,\infty)$, which we also call a ray, is attached to the remaining part of the graph at $x_e=0$, and the vertex of the graph corresponding to $x_e=\infty$ is called a vertex at infinity. In particular, there are no edges between vertices at infinity. We denote by $\mathcal E_\infty\subset \mathcal E$ the set of all rays.  

A connected metric graph $\mathcal G$ admits a natural structure of a metric space.  Namely, the distance between two points on a graph is given by the length of the shortest path among all paths connecting the two points. A more detailed introduction to metric graphs can be found for instance in \cite{berkolaiko2013introduction}.

We consider two classes of noncompact graphs:
\begin{definition} \label{df:localfin}
Let $\mathcal G$ be a connected metric graph. Then we say
\begin{enumerate}[(1)]
\item 
$\mathcal G$  is a \emph{finite metric graph} if there are at most finitely many edges; 
\item $\mathcal G$ is a \emph{locally finite metric graph} if $\E$ is a countable set such that any bounded subset of the graph intersects at most finitely many edges.
\end{enumerate}
\end{definition}
We will always take our graphs to be connected, locally finite noncompact metric graphs. Note that a finite metric graph $\mathcal G$ is compact if and only if $\mathcal G$ does not admit any rays, that is, there are no edges of $\mathcal G$ that are half-lines. Note that a locally finite metric graph is compact, if and only if the graph is bounded. We see immediately that all compact locally finite metric graphs are finite, and all finite metric graphs are locally finite. In particular for compact graphs the notions in Definition \ref{df:localfin} coincide. 
\subsection{First-order Sobolev spaces}
Let $\mathcal G$ be any locally finite metric graph. Here and in the rest of the article denote 
\begin{equation}\label{eq:short}
u_e:=u\big |_e.
\end{equation} 
We denote by $C(\mathcal G)$ the set of continuous complex-valued functions in $\mathcal G$ and define for $1\le p <\infty$
\begin{equation}
\begin{aligned}
L^p(\mathcal G) &= \left \{u\in \bigoplus_{e\in \E} L^p(I_e)\bigg | \|u\|_p^p:=\sum_{e\in \E} \left \|u_e \right \|_{p}^p<\infty\right \}, \\
W^{1,p}(\mathcal G)&=\left \{u\in C(\mathcal G) \bigg| u_e \in W^{1,p}(I_e)\; \land\; \|u\|_{W^{1,p}}^p:=\sum_{e\in \E} \|u_e \|_{W^{1,p}(I_e)}^p <\infty \right \}.
\end{aligned}
\end{equation} 
Then we set $H^1(\mathcal G) = W^{1,2}(\mathcal G)$ as usual. 

For $q=\infty$ we need to adapt the definition above slightly:
\begin{equation}
\begin{aligned}
L^\infty(\mathcal G) &= \left \{u\in \bigoplus_{e\in \E} L^\infty(I_e)\bigg | \|u\|_\infty:=\sup_{e\in \E} \left \|u_e \right \|_{\infty}<\infty\right \}, \\
W^{1,\infty}(\mathcal G)&=\bigg \{u\in C(\mathcal G) | u_e \in W^{1,\infty}(I_e)\; \\
&\qquad\qquad \land\; \|u\|_{W^{1,\infty}}:=\sup_{e\in \E} \|u_e \|_{W^{1,\infty}(I_e)} <\infty \bigg \}.
\end{aligned}
\end{equation}

\subsection{Gagliardo--Nirenberg inequality with magnetic potential}
The goal of this section is to show a Gagliardo--Nirenberg inequality for locally finite, connected metric graphs. With future applications in mind we actually consider a modified Gagliardo--Nirenberg inequality:
\begin{proposition}\label{cor:Gagliardo-nirenberg}
Let $\mathcal G$ be a locally finite, connected metric graph and $M\in C(\mathcal G)$. For $p\in [2,\infty)$ there exists a constant $C>0$  independent of $M$ such that
\begin{equation}\label{eq:Gagliardo-nirenberg}
\|u\|_{p}^p \le C \left \|\left ( i \frac{\mathrm d}{\mathrm dx}+M\right )u\right \|_{2}^{\frac{p-2}{2}} \| u\|_{2}^{\frac{p+2}{2}},
\end{equation}
for all $u\in H^1(\mathcal G)$.
\end{proposition}
\begin{remark}
The inequality reduces to the usual Gagliardo--Nirenberg inequality when $M\equiv 0$.
\end{remark}
\begin{proof}[Proof of Proposition \ref{cor:Gagliardo-nirenberg}]
Suppose $\mathcal G$ is a tree graph at first. Then using the unitary gauge transform $G: H^1(\mathcal G) \to H^1(\mathcal G)$ (see also §5.7 for details) we deduce that \eqref{eq:Gagliardo-nirenberg} is equivalent to
\begin{equation}\label{eq:Gagliardo-nirenberg2}
\|u\|_{p}^p \le C \left \|u'\right \|_{2}^{\frac{p-2}{2}} \| u\|_{2}^{\frac{p+2}{2}},
\end{equation}
which can be shown via symmetrization methods as considered in \cite{adami2015nls}; although this was shown there for finite metric graphs, the proof can be simply adapted to locally finite ones. In particular, the constant $C>0$ can be chosen independent of $M$. 
Cutting the graph at a discrete set of points on the metric graph, i.e.  we can find a tree graph $\widetilde{\mathcal G}$ such that identifying a discrete set of points on the graph results in a graph isometric isomorph to $\mathcal G$. Hence, there exists lifts of the norms on $H^1(\mathcal G)$ to $H^1(\widetilde{\mathcal G})$ preserving the norms and \eqref{eq:Gagliardo-nirenberg} also holds for $H^1(\widetilde{\mathcal G})$ and the constant $C>0$ can be chosen independent of $M\in C(\mathcal G)$.
\end{proof}
Similarly, one can also show the following Sobolev inequality.
\begin{proposition}
Let $\mathcal G$ be a locally finite, connected metric graph and $M\in C(\mathcal G)$. Let $p\in [2,\infty]$ then there exists a constant $C>0$  independent of $M$ such that
\begin{equation}\label{eq:Sobolevinequality}
\|u\|_{p} \le C \left (\int_{\mathcal G} \left |\left ( i \frac{\mathrm d}{\mathrm dx} + M\right ) u\right |^2\, \mathrm dx+\int_{\mathcal G} |u|^2\, \mathrm dx\right )^{1/2}, 
\end{equation} 
for all $u\in H^1(\mathcal G)$.
\end{proposition}
\begin{proof}

The aproach is similar as before, indeeed we can use the known result in absence of $M$ and use a gauge transform to show that \eqref{eq:Sobolevinequality} holds with a constant $C>0$ independent of the potential $M\in C(\mathcal G)$.
\end{proof}

\subsection{Higher-order Sobolev spaces}
In this section we introduce the notion of higher-order Sobolev spaces on graphs for $p\in [1,\infty)$. Let $\mathcal G$ be a locally finite metric graph.  One naive way of doing so is simply defining it analogously as in $W^{1,p}(\mathcal G)$
\begin{equation}
\widetilde {W^{k,p}}(\mathcal G) := \left \{u\in C(\mathcal G) \bigg | u_e \in W^{k,p} (I_e)\quad \forall e\in E\; \land\; \|u \|_{W^{k,p}}^p:=\sum_{e\in E} \|u_e \|_{W^{k,p}(I_e)}^p <\infty \right \}
\end{equation}
Then for $u\in \widetilde{W^{k,p}}(\mathcal G)$ we always have $u_e \in C^{k-1}(I_e)$ for all $e\in \mathcal E$. However, we also want to specify a condition on the higher-order derivatives. We define 
\begin{multline}
W^{k,p}(\mathcal G) = \{u\in \widetilde{W^{k,p}}(\mathcal G)| u^{(j)} \in C(\mathcal G)\quad \forall j\le k-1 \text{ even}\\
\land \quad\sum_{e:e\succ \mathsf v} \frac{\partial^j}{\partial\nu^j}u_e(\mathsf v) =0 \quad \forall j\le k-1 \text{ odd }\, \forall \mathsf v\in V\},
\end{multline}
where $e:e\succ \mathsf v$ denotes the set of edges $e$ adjacent to a vertex $\mathsf v$.
This definition is natural in the sense that if we consider a dummy vertex $\hat v$ of degree $2$, i.e. subdividing an edge $e\in \mathcal E$ connecting two vertices $v_1, v_2$ into two edges $e_1, e_2$ connecting $v_1, \hat v$ and  $\hat v, v_2$ respectively such that the total length of the graph is preserved, then the Kirchhoff condition simply reduces to a continuity statement of the derivatives. As usual we define $\widetilde{H^k}(\mathcal G)= \widetilde{W^{k,2}}(\mathcal G)$ and $H^k(\mathcal G)= W^{k,2}(\mathcal G)$.   

\begin{remark}
While the Sobolev spaces as defined here are domains of self-adjoint realizations of differential operators on $L^2(\mathcal G)$, the definitions are not necessarily canonical. We refer to \cite{gregorio2017bi} for a discussion on self-adjoint extension of the Bilaplacian, and a discussion for $W^{2,p}$ spaces on graphs. 
\end{remark}

In this context, we are going to define some useful related spaces:
\begin{equation}\label{eq:sobolev0}
\begin{gathered}
\widetilde{W^{k,p}_{0}}(\mathcal G) := \{u\in \widetilde{W^{k,p}}(\mathcal G)| u^{(l)}(\mathsf v)=0, \quad \forall 1\le l\le k-1, \quad \forall \mathsf v\in \mathcal V\}\\
\widetilde{W^{k,p}_{c}} (\mathcal G):= \left \{u\in \widetilde{W^{k,p}_0}(\mathcal G)| \operatorname{supp}(u) \text{ compact}\right \}.
\end{gathered}
\end{equation}
Of special importantance will be the following test function spaces:
\begin{equation}
\begin{split}
\widetilde{C^\infty(\mathcal G)} &:= \bigcap_{k\in \mathbb N}\widetilde{W^{k, \infty}}(\mathcal G)\\
\widetilde{C^\infty_b(\mathcal G)} &:= \bigcap_{k\in \mathbb N}\widetilde{W_0^{k, \infty}}(\mathcal G)\\
\widetilde{C^\infty_c(\mathcal G)} &:= \bigcap_{k\in \mathbb N}\widetilde{W_c^{k, \infty}}(\mathcal G).
\end{split}
\end{equation}

\subsection{Higher-order Gagliardo--Nirenberg inequality for finite metric graphs}
Let $\mathcal G$ be a finite metric graph. Consider the norm on $H^k(\mathcal G)$ defined as
\begin{equation}
|u|_{H^k} := \left (\int_{\mathcal G} |u^{(k)}|^2+ |u|^2\, \mathrm dx \right )^{1/2}. 
\end{equation}
Then  due to the Gagliardo--Nirenberg interpolation inequality on intervals (see e.g. \cite[Theorem 7.41]{leoni2017first}) applied edgewise
\begin{equation}\label{eq:ineedthis}
|u|_{H^k}^2 \le \|u\|_{H^k}^2\le C |u|_{H^k}^2
\end{equation}
and we conclude that $\|\cdot\|_{H^k}$ and $|\cdot|_{H^k}$ are equivalent norms in $H^k(\mathcal G)$.

\begin{proposition}\label{prop:higherordergag}
Let $k\in \mathbb N$ and $\mathcal G$ be a finite metric graph. Then
\begin{equation}\label{eq:propgagk}
\|u\|_{p}^p \le C \|u\|_2^{\frac{(2k-1)p+2}{2k}} |u|_{H^k}^{\frac{p-2}{2k}}
\end{equation}
for all $u\in H^k(\mathcal G)$.
\end{proposition}  
\begin{proof}
From the Gagliardo--Nirenberg inequality on metric graphs and Gagliardo--Nirenberg interpolation inequality on intervals we compute
\begin{equation}
\begin{aligned}
\|u\|_{p}^p &\le  C_1 \|u\|_{2}^{\frac{p}{2}+1}\|u'\|_{2}^{\frac{p}{2}-1}\\
&\le C_k \left \|u\right \|_{2}^{\frac{(2k-1)p+2}{2k}} \big |u\big |_{H^{k}}^{\frac{p-2}{2k}}.
\end{aligned}
\end{equation}
\end{proof}

\subsection{On the density of Sobolev spaces}
\begin{proposition}\label{prop:density}
Let $\mathcal G$ be a finite, connected metric graph and $p\in [1,\infty)$, then $W^{m,p}(\mathcal G)$ is dense in $W^{n,p}(\mathcal G)$ for $m\ge n\ge 1$.
\end{proposition}
\begin{remark}
For $n=0$ this corresponds to the fact that $W^{m,p}$ is dense in $L^p$ for $m\ge 1$.
\end{remark}
\begin{proof}[Proof of Proposition~\ref{prop:density}]
It suffices to prove that $W^{k+1,p}(\mathcal G)$ is dense in $W^{k,p}(\mathcal G)$.  To this end, let $u\in W^{k}(\mathcal G)$ arbitrary and $u_n$ be an edgewise approximating sequence in $\oplus_{e\in \mathcal E} C^\infty(I_e)\cap W^{k+1,p}(I_e)$ such that
\begin{equation}\label{eq:ineedthisnow}
\sum_{e\in \mathcal E} \left \|u_n -u\big |_e\right \|_{W^{k,p}} \le \frac{1}{2^n}
\end{equation}
for all $n\in \mathbb N$. The general idea is to construct sequences $v_n \in \oplus_{e\in \mathcal E} W^{k+1,p}(I_e)$ such that $u_n + v_n \in W^{k+1,p}(\mathcal G)$ and 
$$u_n + v_n \stackrel{W^{k,p}}\longrightarrow u\qquad (n\to \infty).$$
For fixed $\mathsf v\in \mathcal V$ and for $n$ satisfying 
\begin{equation}
\frac{2}{n} \le \min_{e\in \mathcal E} |I_e|
\end{equation}
for all $e\succ \mathsf v$ and $\hat k \in \{0, \ldots, k\}$ we define 
\begin{equation}\label{eq:testfunctions}
v_{n,\hat k, \mathsf v}(x)\big |_e=\begin{cases}
\frac{c_{n, \hat k, \mathsf v}}{\hat k!} x_{\mathsf v}^{\hat k} \left (1-nx_{\mathsf v}\right )^{k+1}, &\quad \text{for }x\in e \text{ with } x_{\mathsf v}:=\operatorname{dist}(x,\mathsf v)\le \frac{1}{n} \\
0, &\quad \text{otherwise}.
\end{cases}
\end{equation}
where $c_{n, \hat k, \mathsf v}$ is given by
\begin{equation}
\begin{aligned}
\text{for }\hat k=0:\qquad c_{n, 0, \mathsf v}&= u-u_n \big |_e (0_{\mathsf v})\\
\text{for } 1\le \hat k \le k-1:\qquad c_{n, \hat k, \mathsf v} &= u^{(\hat k)}-u_n^{(\hat k)} \big |_e(0_{\mathsf v}) - \sum_{\ell=0}^{\hat k-1}  (k+1)_{\ell} (-n)^{\ell}  c_{n,\ell, \mathsf v}.
\end{aligned}
\end{equation}
We can extend the functions $v_{n,\hat k, \mathsf v}$ by zero on the rest of the graph. With the Leibniz rule for $1\le \ell \le k+1$ we compute
\begin{equation}\label{eq:leibnizcomp}
v_{n,\hat k, v}^{(\ell)} (x)\big |_e= \chi_{\{x_{\mathsf v}\le \frac{1}{n}\}}\sum_{m=0}^\ell \frac 1{\hat k!}{\binom{\ell}{m}} c_{n, \hat k, \mathsf v} (-n)^{\ell-m}(\hat k)_m  (k)_{\ell-m}  x_{\mathsf v}^{\hat k - m}(1-n x_{\mathsf v})^{k+m+1-\ell}\
\end{equation}
Then
\begin{equation}
\tilde v_n := \sum_{\ell=0}^{k-1}\sum_{\mathsf v\in \mathcal V} v_{n,\ell, \mathsf v}
\end{equation}
satisfies $\tilde v_n \in \oplus_{e\in \mathcal E} W^{k+1,p}(I_e)$ and $u_n+\tilde v_n \in W^{k,p}(\mathcal G)$ since  $u_n+\tilde v_n$ coincides in all $k-1$ derivatives with $u$ by construction. Indeed, the restrictions of the $k$\textsuperscript{th} derivatives at the vertices are of rank $\le 2|\mathcal E|$. Then we can find $c_{n,k,\mathsf v}$ for all $\mathsf v\in \mathcal V$
\begin{equation}
v_{n,k,\mathsf v}\big |_e (x) =\begin{cases}
\frac{c_{n, k, \mathsf v}}{k!} x_{\mathsf v}^{k} \left (1-\max\{n, c_{n,k, \mathsf v}^{2}\}x_{\mathsf v}\right )^{k+1}, &\quad \text{for }x\in e \text{ with } \\
&x_{\mathsf v}\le \min\{n^{-1}, c_{n,k, \mathsf v}^{-2}\}  \\
0, &\quad \text{otherwise}.
\end{cases}
\end{equation}
such that 
$$u_n + \tilde v_n + \sum_{\mathsf v\in \mathcal V} v_{n,k,\mathsf v}\in W^{k+1,p}(\mathcal G).$$
By assumption \eqref{eq:ineedthisnow} we deduce by applying the Sobolev imbedding edgewise
\begin{equation}\label{eq:additionalinequ}
\sum_{e\in \mathcal E} \left \|u_n-u\big |_e\right \|_{C^{k-1}} \le \frac{C}{2^n}
\end{equation}
for all $n\in \mathbb N$ and some $C>0$ and satisfies by construction
\begin{equation}\label{eq:trickher}
(-n)^\ell c_{n,\hat k, \mathsf v} \to 0 \qquad (n\to \infty)
\end{equation}  
for all $1\le \ell\le k$ and $\mathsf v\in \mathcal V$.  
By a change of variables we then compute for $0\le m \le \ell \le k$
\begin{gather}
 n^{\ell-m} \int_{I_e} x_{\mathsf v}^{\hat k-m} (1-n x_{\mathsf v})^{k+m+1-\ell}\, \mathrm dx_{\mathsf v}= n^{\ell-1 -\hat k} \int_0^1 t^{\hat k-m} (1-t)^{k+m+1-\ell}\, \mathrm dt\\
\begin{multlined}
c_{n,k, \mathsf v}\max\{n, c_{n,k, \mathsf v}^2\}^{\ell-m} \int_{I_e} x_{\mathsf v}^{k-m} (1-\max\{n, c_{n,k, \mathsf v}^2\} x_{\mathsf v})^{k+m+1-\ell}\, \mathrm dx_{\mathsf v}\\
= c_{n,k, \mathsf v}\min\{n^{-1}, c_{n,k, \mathsf v}^{-2}\}^{k+1-\ell} \int_0^1 t^{k-m} (1-t)^{k+m+1-\ell}\, \mathrm dt \longrightarrow 0 \\
(n\to\infty)
\end{multlined}
\end{gather}
and with \eqref{eq:leibnizcomp}  and \eqref{eq:trickher} we conclude
\begin{equation}
\begin{aligned}
&\left \|u- \left [u_n + \tilde v_n + \sum_{\mathsf v\in \mathcal V} v_{n,k,\mathsf v} \right ]\right \|_{W^{k,p}}\\
&\qquad \qquad\le \left \|u- u_n\right \|_{W^{k,p}} + \left \|\tilde v_n + \sum_{\mathsf v\in \mathcal V} v_{n,k,\mathsf v}\right \|_{W^{k,p}}\longrightarrow 0 \qquad (n\to \infty).
\end{aligned}
\end{equation}

\end{proof}
\begin{lemma}\label{prop:densitycompact}
Let $\mathcal G$ be a locally finite, connected metric graph and $p\in [1,\infty)$. Then
$$W^{1,p}_c(\mathcal G)= \{u\in W^{1,p}(\mathcal G)| \operatorname{supp} u \text{ is bounded}\} $$ 
is dense in $W^{1,p}(\mathcal G)$.
\end{lemma}
\begin{proof}
Let $K$ be a bounded, connected subgraph of $\mathcal G$. For $R>0$ set
\begin{equation}
K_R :=\{x\in \mathcal G| \operatorname{dist}(x, K)<R\}.
\end{equation}
We define the cut-off functions $\psi_n$ via
\begin{equation}
\widetilde{\psi_n} := \frac{1}{n}\max \{n, \operatorname{dist}(x,K_n)\},\quad \psi_n:= 1- \widetilde{\psi_n} 
\end{equation}
For all $u\in W^{1,p}(\mathcal G)$ one then computes
\begin{equation}
\begin{aligned}
&\limsup_{n\to \infty}\|u- \psi_n u\|_{W^{1,p}}^p =  \limsup_{n\to \infty}\left [\int_{\mathcal G} \left |\frac{\mathrm d}{\mathrm dx} \widetilde{\psi_n} u\right |^p\, \mathrm dx+ \int_{\mathcal G} \left |\widetilde {\psi_n} u\right |^p\, \mathrm dx\right ]\\
&\qquad\le  \limsup_{n\to \infty} \left [\frac{2^p}{n^p} \int_{\mathcal G\setminus K_n} |u|^p\, \mathrm dx +2^p \int_{\mathcal G\setminus K_n} \left |\widetilde \psi_n u \right |^p\, \mathrm dx + \int_{\mathcal G \setminus K_n}   \left |\widetilde \psi_n u\right |^p\, \mathrm dx\right]=0,
\end{aligned}
\end{equation}
where in the equation we used
\begin{equation}
\int_{\mathcal G\setminus K_n} |\widetilde{\psi_n} u|^p \, \mathrm dx \le \int_{\mathcal G\setminus K_n} |u|^p \, \mathrm dx \to 0 \qquad (n\to \infty).
\end{equation}
As such 
$\psi_n u \to u$
in $W^{1,p}(\mathcal G)$ as $n\to \infty$.
\end{proof}

\begin{proposition}\label{prop:densitylocfinite}
Let $\mathcal G$ be a locally finite, connected metric graph and $p\in [1,\infty)$. Then
$$W^{2,p}_c(\mathcal G)= \{u\in W^{2,p}(\mathcal G)| \operatorname{supp} u \text{ is bounded}\} $$ 
is dense in $W^{1,p}$.
\end{proposition}
\begin{proof}
Let $u\in W^{1,p}(\mathcal G)$.  
By Lemma~\ref{prop:densitycompact} we can find a sequence $u_n\in W^{1,p}_c(\mathcal G)$ with $u_n \to u$ in $W^{1,p}$. Then by Proposition \ref{prop:density} for each $n$ we find  a sequence $u_{n,m} \in W^{2,p}(\mathcal G)$, after extending by zero on the whole graph, converging towards $u_n$ in $W^{1,p}(\mathcal G)$ as $m\to \infty$. Then one can construct a sequence in $W^{2,p}_c(\mathcal G)$  converging to $u$ in $W^{1,p}$ by a diagonal argument.
\end{proof}

\begin{remark}\label{rmk:densitylocfinite}
Proposition~\ref{prop:densitylocfinite} does not depend on the particular choice of vertex conditions. For instance, if $M\in H^1+ W^{1,\infty}(\mathcal G)$ then we may equally show
\begin{equation}
D_c(A^M)=\left \{u\in \widetilde{W^{2,p}}(\mathcal G) | \sum_{e\succ \mathsf v} \left ( \frac{\partial}{\partial\nu} - M\right ) u_e(\mathsf v) =0\, \text{ and } \operatorname{supp} u \text{ is bounded}\right \}
\end{equation}
is dense in $W^{1,p}(\mathcal G)$.
\end{remark}

\subsection{Characterization of $W^{1,\infty}$}
We give a characterization of $W^{1, \infty}$ on locally finite, connected metric graphs in the following:
\begin{proposition}\label{prop:lipschitz}
Let $\mathcal G$ be a locally finite, connected metric graph. Then $W^{1,\infty}(\mathcal G)= C^{0,1}_b(\mathcal G)$ is the set of uniformly bounded, Lipschitz continuous functions.
\end{proposition}
\begin{proof}
Assume $u\in W^{1,\infty}(\mathcal G)$. Let $x,y\in \mathcal G$ and $\gamma$ be a path of length $L(\gamma)$ connecting $x,y$, parametrized by arc length. In the first step let us assume $u\in C^1$ edgewise, then using the continuity of $u$ we have
\begin{equation}
|u(x)-u(y)| \le \int_0^{L(\gamma)} |u'(\gamma)|\, \mathrm d|\gamma| \le \max_{t} |u'(\gamma(t)| L(\gamma). 
\end{equation}
Due to density this holds also for $W^{1,\infty}(\mathcal G)$. Taking the infimum over all paths connecting $x,y$ we conclude
\begin{equation}
|u(x)-u(y)|\le \|u'\|_{\infty} \operatorname{dist}(x,y)
\end{equation}
and thus $u\in C^{0,1}_b(\mathcal G)$. 
On the other hand, let $u\in C^{0,1}_b(\mathcal G)$, then
\begin{equation}
\frac{|u(x)-u(y)|}{\operatorname{dist}(x,y)} \le  L
\end{equation}
for some constant $L>0$. On each edge $e\in \mathcal E$ then $u\in W^{1,\infty}(I_e)$ and $u'$ exists a.e. and
\begin{equation}
\|u'\|_{\infty} \le L.
\end{equation}
We conclude $u\in W^{1,\infty}(\mathcal G)$ since $u$ is also uniformly bounded by assumption. 
\end{proof}

\section[Application: Nonlinear Equations on Finite Graphs]{Application: Existence of Ground states of a class of Nonlinear Equations involving the Polylaplacian on Finite Graphs}
In this section, we give a first application of the results derived in §2 on finite metric graphs. In this context we show a decomposition formula for the Polylaplacian.
\subsection{Formulation of the problem}\label{sec:existence}
Let $\mathcal G$ be a connected, finite metric graph and let $K$ be a connected subgraph of $\mathcal G$.  For $k\in \mathbb N$ consider the energy functional 
\begin{equation}
E^{(k)}(u) = \frac{1}{2}\int_{\mathcal G} |u^{(k)}|^2 + V |u|^2 \, \mathrm dx - \frac{\mu}{p} \int_{\mathcal G} |u|^p\, \mathrm dx,
\end{equation}
with $2<p< 4k+2$ and $V\in L^2+ L^\infty(\mathcal G)$ and for $c>0$ consider the minimization problem
\begin{equation}\label{eq:minimizationeq}
E_{c}^{(k)} := \inf_{\substack{u\in H^k(\mathcal G)\\ \|u\|_{2}^2=c}} E^{(k)}(u).
\end{equation}

\begin{lemma}\label{lem:energyestimateneu}
Let $\mathcal G$ be a finite connected metric graph. The functional $E^{(k)}$ under the $L^2$-constraint $\|\cdot\|_{2}^2=c$ is bounded below for $2<p<4k+2$ and $c>0$. Moreover, for each $0<\varepsilon <1$ there exists a $C_\varepsilon>0$, such that
\begin{equation}
E^{(k)}(u) \ge \frac{1-\varepsilon}{2} \int_{\mathcal G} |u^{(k)}|^2 + V |u|^2 \, \mathrm dx - C_\varepsilon.
\end{equation}
\end{lemma}
\begin{proof}

Let $\varepsilon>0$. Consider a decomposition of $V\in L^2+ L^\infty$ such that
\begin{equation}
V = V_2+ V_\infty, \quad \|V_2\|_{2}\le\varepsilon.
\end{equation}
Then 
\begin{multline}
\int_{\mathcal G} \left |u^{(k)}\right |^2- \left \|V_\infty\right \|_\infty |u|^2\, \mathrm dx- \varepsilon \|u\|_{4}^2\\
\le \int_{\mathcal G} \left |u^{(k)}\right |^2+ V |u|^2\, \mathrm dx\\
\le \int_{\mathcal G} \left |u^{(k)}\right |^2+ \left \|V_\infty\right \|_\infty |u|^2\, \mathrm dx+ \varepsilon \|u\|_{4}^2.
\end{multline}
By the Sobolev inequality we infer
\begin{equation}
\|u\|_4^2 \le C_1  \|u\|_{H^k}^2 \le C_2 \left (\left |u^{(k)}\right |_2^2+ \left |u\right |_2^2 \right ). 
\end{equation}
Adding a constant to the potential if necessary we have that
\begin{equation}
\|u\| :=\left (\int_{\mathcal G} \left |u^{(k)}\right |^2+ V |u|^2\, \mathrm dx\right )^{1/2}
\end{equation}
defines an equivalent norm on $H^k(\mathcal G)$.

From Proposition~\ref{prop:higherordergag} we have
\begin{equation}
\|u\|_{L^p}^p\le C \left \|u\right \|_{L^2(\mathcal G)}^{\frac{(2k-1)p+2}{2k}} \big \|u\big \|^{\frac{p-2}{2k}}
\end{equation}
for some $C>0$. 
Let $0<\varepsilon <1$, then with Young's inequality we infer for all $u\in H^k(\mathcal G)$ with $\|u\|_2^2=c$
\begin{equation}
\frac{\mu}{p}\|u\|_{L^p}^p \le \frac{\varepsilon}{2} \|u\|^2 + C_{\varepsilon,c} 
\end{equation}
for some $C_{\varepsilon, c} >0$ and we obtain
\begin{equation}
E^{(k)} \ge \frac{1-\varepsilon}{2} \int_{\mathcal G} |u^{(k)}|^2 + V |u|^2 \, \mathrm dx - C_{\varepsilon,c}
\end{equation}
for $2< p<4k+2$. 
\end{proof}

\begin{proposition}\label{prop:NLSmult}
Let $\mathcal G$ be a finite, connected metric graph. Assume $u\in H^k(\mathcal G)$ is a minimizer of $E^{(k)}_c$, then $u\in H^{2k}(\mathcal G)$ and there exists $\lambda\in \mathbb R$ such that
\begin{equation}\label{eq:NLSELG}
(-1)^k u_e^{(2k)} + \left ( V+\lambda\right ) u_e = \mu |u_e|^{p-1} u_e
\end{equation}
for all $e\in \E$.
\end{proposition}
\begin{proof}
Since $E^{(k)}\in C^1(H^k(\mathcal G), \mathbb R)$ and the $L^2$-constraint is also $C^1$, and $u$ is a constrained critical point we can compute the G\^ateaux derivative
\begin{equation}
\int_{\mathcal G} \left ( u^{(k)} \eta^{(k)}- u |u|^{p-2} \eta\right )\, \mathrm dx+  \int_{\mathcal G} \left ( V+\lambda\right ) u\eta\, \mathrm dx=0, \qquad \forall \eta\in H^k(\mathcal G)
\end{equation} 
where $\lambda$ is a Lagrange multiplier. Fixing an edge $e$, then with $\eta \in C_c^\infty(I_e)$ and integration by parts we deduce \eqref{eq:NLSELG} for each $e\in \E$ and by elliptic regularity $u\in \widetilde{H^{2k}}(\mathcal G)$. Fixing now $\mathsf v\in V$ and taking $\eta\in H^k(\mathcal G)$ to be locally supported near $\mathsf v$ and not supported at any other vertex, then by integration by parts we deduce
\begin{equation}
\sum_{j=1}^{k}(-1)^{j} \sum_{e\succ \mathsf v}\tfrac{\partial^{(k+j-1)}}{\partial^{(k+j-1)}\nu}u_e \tfrac{\partial^{(k-j)}}{\partial^{(k-j)}\nu}\eta_e(\mathsf v)=0.
\end{equation}  
Since the choice $\eta\in H^{k}$ is arbitrary we deduce
\begin{equation}
\begin{cases}
\sum_{e\succ \mathsf v} \frac{\partial^\ell}{\partial \nu^\ell}u_e(\mathsf v)=0, \qquad \forall k\le \ell \le 2k-1 \text{ odd},\\
u_{e_1}^{(\ell)}(\mathsf v) = u_{e_2}^{(\ell)}(\mathsf v), \qquad \forall k \le \ell \le 2k-1 \text{ even and } \forall e_1, e_2 \text{ adjacent at } \mathsf v
\end{cases}
\end{equation}
for all $\mathsf v\in \mathcal V$.
\end{proof}

\subsection{Partitions of unity in $\widetilde{C_b^\infty}$}
Let $\mathcal G$ be any locally finite, connected graph and $\mathcal O$ be a finite covering of $\mathcal O$. We construct a partition of unity in $\widetilde{C_b^\infty}(\mathcal G)$ by choosing appropriate partitions of unities subordinate to the covering. 
One rather different ``normalization'' of the usual partition of unity will be especially useful in applications:
\begin{lemma}\label{lem:unity2}
Let $\mathcal G$ be a metric graph. Consider any finite open covering $\mathcal O$ of $\mathcal G$. There exists a partition of unity subordinate to $\mathcal O$ in $\widetilde {C_b^\infty}$ satisfying
\begin{equation}\label{eq:normalizationt}
\sum_{O\in \mathcal O} \Psi_O^2 \equiv 1.
\end{equation}
\end{lemma}
\begin{proof}
Consider any smooth partition of unity $\{\psi_O\}_{O\in \mathcal O}$ on the graph subordinate to the open covering $\mathcal O$ satisfying
\begin{equation}
\sum_{O\in \mathcal O} \Psi_O \equiv 1. 
\end{equation}
Then we may define
\begin{equation}
\Psi_O:= \frac{\psi_O}{\sqrt{\sum_{O\in \mathcal O} \psi_O^2}}
\end{equation}
for all $O\in \mathcal O$, which is smooth restricted as functions on all edges since $\sum_{O\in \mathcal O} \psi_O^2(y)\neq 0$ for all $y\in\mathcal G$. Furthermore, it is constant in a neighborhood of any vertex and we infer $\Psi_O\in \widetilde{C_b^\infty}$. By construction we conclude 
\begin{equation}
\sum_{O\in \mathcal O} \Psi_O^2 \equiv 1.
\end{equation}
\end{proof}
\begin{remark}
The normalization in \eqref{eq:normalizationt} replaces in this context the typical  normalization, where one assumes
\begin{equation}
\sum_{O\in \mathcal O } \psi_O \equiv 1.
\end{equation}
Throughout the rest of the paper we will only work with partitions of unity  that satisfy the normalization \eqref{eq:normalizationt}.
\end{remark}

\begin{example}\label{ex:unity2}
Let $\mathcal G$ be a finite, connected metric graph with core $K= \mathcal G\setminus \mathcal E_\infty$. Consider on $\mathcal G$ the open covering $ \mathcal O$ given by $K_2$ and $\mathcal G\setminus K_1$, where $K_1$ and $K_2$ are the neighborhoods of $K$ given as in \eqref{eq:core}, such that $\mathcal G\setminus K_1$ only consists of disjoint rays. Consider the partition of unity subordinate to $\mathcal O$ from Lemma \ref{lem:unity2} given by $\psi_K, \{\psi_e\}_{e\in \mathcal E_\infty}$ respective to $K_2$ and $\mathcal G\setminus K_1$, then we define slight modifications
\begin{align}
\psi_{K,R}(x) &=\begin{cases}
1, &\qquad \text{on K}\\
\psi_{K}(x/R) &\qquad \text{on all rays } e\in \mathcal E_\infty;
\end{cases}\\
\psi_{e,R}(x) &=\begin{cases}
0, &\qquad \text{on }\mathcal G\setminus \{e\}\\
\psi_{e}(x/R) &\qquad \text{on } e\in \mathcal E_\infty.
\end{cases}
\end{align}
Then by definition, $\{\psi_{O,n}\}_{O\in \mathcal O}$ is a vanishing-compatible sequence of partitions of unity subordinate to the open coverings given by $K_{2n}$ and $\mathcal G\setminus K_{n}$. 
By Lemma \ref{lem:unity2} there exists a sequence of partitions of unity
$$\Psi_n:=\Psi_{K,n},\qquad  \widetilde{\Psi_n}:=\sum_{e\in \mathcal E_\infty}\Psi_{e,n}$$ 
in $\widetilde{C_b^\infty}$ satisfying
\begin{equation}
\Psi_{n}^2+ \widetilde{\Psi_n}^2\equiv 1.
\end{equation} 
\end{example}

\subsection{A decomposition formula}\label{sec:decompositionpoly}
In the following we identify a given function $f\in \widetilde{C_b^\infty}(\mathcal G)$ with its corresponding multiplication operator $\mathcal M_f \phi:= f\phi$. Let $A$ be an operator such that $fD(A) \subset D(A)$, then we can define the commutator $[A,f]=Af-fA$ and 
\begin{align}
fAf&= f^2 A + f [A, f]\\
fAf&= Af^2 + [A,f]f.
\end{align} 
Averaging the two preceding equations we conclude
\begin{equation}\label{eq:decompositionn}
fAf = \frac{1}{2} ( f^2 A + A f^2) + \frac{1}{2} ( f [A,f] - [A,f]f).
\end{equation}

\begin{lemma}\label{lem:IMSunity}
Let $\mathcal G$ be a locally finite, connected metric graph. Assume $\{\psi_k\}_{k=1}^N$  is  a family of function in $\widetilde{C_b^\infty}(\mathcal G)$ with $0 \le \psi_k\le 1$ for all $k\in \{1,\ldots, k\}$ and 
\begin{equation}
\sum_{k=1}^N \psi_k^2 \equiv 1.
\end{equation}
Assume $D(A)$ is invariant under multiplication by elements in $\widetilde{C_b^\infty}(\mathcal G)$, then
\begin{equation}\label{eq:decompositionadv}
A= \sum_{k=1}^N \psi_k A \psi_k -\frac{1}{2} ( \psi_k [A,\psi_k] - [A,\psi_k]\psi_k).
\end{equation}
\end{lemma}
\begin{proof}
Follows immediately with \eqref{eq:decompositionn}.
\end{proof}

We refer to \eqref{eq:decompositionadv} as a decomposition formula for $A$. In the following, we develop a decomposition formula for the Polylaplacian $A=(-\Delta)^k$.

Let $\mathcal G$ be a finite metric graph and let $k\in \mathbb N$. In the following, we define the Polylaplacian $A=(-\Delta)^k$ on $\mathcal G$ as an operator $A: D(A) \subset L^2(\mathcal G) \to L^2(\mathcal G)$ given by
\begin{gather}
\left ((-\Delta)^k u\right )_e :=  (-\Delta)^k u_e:= (-1)^k u_e^{(2k)}\\
D(A) = H^{2k}(\mathcal G).
\end{gather}

\begin{lemma}\label{lem:decomposition2}
Let $\mathcal G$ be a locally finite connected graph. Let $A=(-\Delta)^k$ with $D(A)= H^{2k}$, then 
\begin{enumerate}[(i)]
\item $fD(A)\subset D(A)$ for all $f\in \widetilde{C_b^\infty}(\mathcal G)$.
\item Let $f\in \widetilde{C_b^\infty}(\mathcal G)$, then the operator $fAf$ is given by 
\begin{equation}\label{eq:decomposition2}
\begin{aligned}
\left (fAf\right ) \phi &= \frac{1}{2} ( f^2 A + A f^2)\phi \\
&\qquad+ \frac{(-1)^{k+1}}{2} \sum_{m=1}^{2k-1} \sum_{n=1}^{2k-m} \frac{(2k)_{m+n}}{m! n!} f^{(m)} f^{(n)} \phi^{(2k-m-n)}
\end{aligned}
\end{equation}
for all $\phi \in D(A)$.
\end{enumerate}
\end{lemma}
\begin{proof}
We apply Leibniz' formula and compute
\begin{equation}
\begin{aligned}
[A,f]\phi &= (-\Delta)^k f \phi - f(-\Delta)^k \phi\\
&= (-1)^k \sum_{m=1}^{2k} \binom{2k}{l} f^{(m)} \phi^{(2k-m)}.
\end{aligned}
\end{equation}
Then we apply Leibniz' formula again and compute
\begin{equation}
\begin{aligned}
\left ([A,f]f\right )\phi &= (-1)^k \sum_{m=1}^{2k}\sum_{n=0}^{2k-m} \binom{2k}{m}\binom{2k-m}{n} f^{(m)} f^{(n)} \phi^{2k-m-n}\\
&= (-1)^k \sum_{m=1}^{2k} \sum_{n=0}^{2k-m} \frac{(2k)_{m+n}}{m! n!} f^{(m)} f^{(n)} \phi^{(2k-m-n)}
\end{aligned}
\end{equation}
and we conclude
\begin{equation}
\frac{1}{2} ( f [A,f] - [A,f]f)\phi=\frac{(-1)^{k+1}}{2} \sum_{m=1}^{2k-1} \sum_{n=1}^{2k-m} \frac{(2k)_{m+n}}{m! n!} f^{(m)} f^{(n)} \phi^{(2k-m-n)}.
\end{equation}
The statement follows upon combining this with \eqref{eq:decompositionadv}.
\end{proof}

Given the core $K= \mathcal G\setminus \E_\infty$ of $\mathcal G$ and $R>0$ we define 
\begin{equation}\label{eq:numerateR}
\begin{aligned}
D_R&:= \{\phi \in D(A) |\operatorname{supp}(\phi) \subset \mathcal G \setminus K_R\}\\
\Sigma_R&:= \inf\{\langle \phi, A\phi\rangle|\phi \in D_R, \|\phi\|_{2}^2=1\},
\end{aligned}
\end{equation} 
where $K_R$ was defined in \eqref{eq:core}.

For $R=0$ we set
\begin{equation}\label{eq:numerate0}
\begin{aligned}
D_0&:= D(A)\\
\Sigma_0&:= \inf\{\langle \phi, A\phi \rangle |\phi \in D(A), \|\phi\|_{2}^2=1\}.
\end{aligned}
\end{equation}
The last relevant quantity, which we will discuss later in §\ref{sec:onthreshold} is 
\begin{equation}\label{eq:numerate1}
\Sigma := \lim_{R\to \infty} \Sigma_R=\sup_{R>0} \Sigma_R.
\end{equation}

\begin{lemma}\label{lem:preconditions}
Let $\mathcal G$ be a finite, connected metric graph and let $V\in L^2+ L^\infty(\mathcal G)$. $E^{(k)}$ is weak limit superadditive, superadditive with respect to the partition of unity in Example \ref{ex:unity2}. Assume $A= (-\Delta)^k + V$ admits a ground state, then 
$t\mapsto E_t$ as defined in \eqref{eq:minimizationeq} is strictly subadditive. 
\end{lemma}
\begin{proof}
\emph{Weak limit superadditivity.} We showed in the proof of Lemma~\ref{lem:energyestimateneu} that 
\begin{equation}
\|u\| =\left (\int_{\mathcal G} \left |u^{(k)}\right |^2+ V |u|^2\, \mathrm dx\right )^{1/2}
\end{equation}
defines an equivalent norm on $H^k(\mathcal G)$ upon adding a constant to $V$. 

Assume $u_n \rightharpoonup u$ in $H^k(\mathcal G)$ weakly, then up to a subsequence, which we still denote by $u_n$, by the Brezis--Lieb Lemma and weak convergence
\begin{equation}
\begin{gathered}
\limsup_{n\to \infty} \|u_n\|^2  = \|u\|^2 + \limsup_{n\to \infty} \|u-u_n\|,\\
\limsup_{n\to \infty} \int_{\mathcal G} |u_n|^p \, \mathrm dx= \int_{\mathcal G} |u|^p\, \mathrm dx + \limsup_{n\to \infty}\int_{\mathcal G} |u-u_n|^p \, \mathrm dx.
\end{gathered}
\end{equation}
Then
\begin{equation}
\limsup_{n\to \infty} E^{(k)} (u_n) =E^{(k)} (u)+ \limsup_{n\to \infty} E^{(k)} (u-u_n)
\end{equation}
and $E^{(k)}$ is weak limit superadditive.

\emph{Superadditivity with respect to a sequence of partitions of unity.}
Finally we need to show superadditivity with respect to the partition of unity $\{\Psi_n, \widetilde {\Psi_n}\}$  in Example \ref{ex:unity2}. Let $u_n\rightharpoonup 0$ be a vanishing sequence with $\|u_n\|_{L^2}^2=c$, then 
\begin{equation}
\|\Psi_n u_n\|_{L^2}^2 +\|\widetilde{\Psi_n} u_n\|_{L^2}^2= c
\end{equation}
 and up to a subsequence, still denoted by $u_n$,
\begin{gather}
\lim_{n\to \infty} \int_{K_{2n}} |u_n|^p\, \mathrm dx =0, \qquad \lim_{n\to \infty} \int_{K_{2n}} |u_n|^2\, \mathrm dx = 0\\
\lim_{n\to \infty} \int_{\mathcal G} |u_n|^p=\liminf_{n\to \infty}  \int_{\mathcal G} |u_n|^p = \limsup_{n\to \infty} \int_{\mathcal G} |u_n|^p \, \mathrm dx.
\end{gather}
Then from
\begin{equation}
\begin{aligned}
0&\le \liminf_{n\to \infty} \int_{K_{2n}} |\Psi_n u_n|^p\, \mathrm dx \le \limsup_{n\to \infty}  \int_{K_{2n}} |\Psi_n u_n|^p\, \mathrm dx \le \lim_{n\to \infty} \int_{K_{2n}} |u_n|^p\, \mathrm dx\\
0&\le \liminf_{n\to \infty} \int_{K_{2n}} |\widetilde{\Psi_n} u_n|^p\, \mathrm dx \le \limsup_{n\to \infty}  \int_{K_{2n}} |\widetilde{\Psi_n} u_n|^p\, \mathrm dx \le \lim_{n\to \infty} \int_{K_{2n}} |u_n|^p\, \mathrm dx
\end{aligned}
\end{equation}
we deduce 
\begin{gather}
\lim_{n\to \infty} \int_{\mathcal G} |\Psi_n u_n|^p\, \mathrm dx=\lim_{n\to \infty} \int_{\mathcal K_{2n}} |\Psi_n u_n|^p =0\\
\lim_{n\to \infty} \int_{K_{2n}} |\widetilde{\Psi_n} u_n|^p\, \mathrm dx=0.
\end{gather}
and in particular
\begin{equation}
\lim_{n\to \infty} \int_{\mathcal G} |u_n|^p{\mathrm dx} = \lim_{n \to \infty} \int_{\mathcal G\setminus K_{2n}} |\widetilde{\Psi_n} u_n|^p \, \mathrm dx  = \lim_{n\to \infty}  \int_{\mathcal G} |\widetilde{\Psi_n} u_n|^p.
\end{equation}

$H^{2k}(\mathcal G)$ is dense in $H^k(\mathcal G)$ by Proposition \ref{prop:density} and we may assume that there exists a minimizing sequence in $H^{2k}(\mathcal G)$. Let $u_n$ be a minimizing sequence in $H^{2k}(\mathcal G)$, then by Lemma~\ref{lem:decomposition2} we conclude passing to a subsequence, still denoted by $u_n$, using integration by parts and Young's inequality
\begin{equation}
\begin{aligned}
&\sum_{i=1}^{2k-1} \sum_{j=1}^{2k-m} \frac{(2k)_{m+n}}{i! j!} \left |\left \langle \Psi_n^{(i)} \Psi_n^{(j)} u_n^{(2k-i-j)}, \phi\right \rangle_{L^2} \right | \\
&\qquad=  \sum_{i=1}^{2k-1} \sum_{j=1}^{2k-m}  \frac{(2k)_{m+n}}{i! j!}  \left |\left\langle \frac{\mathrm d^{k-j}}{\mathrm dx^{k-j}} \Psi_n^{(i)} \Psi_n^{(j)}u_n, u_n^{(k-i)}\right \rangle_{L^2}\right |\\
&\qquad\le    \sum_{i=1}^{2k-1} \sum_{j=1}^{2k-m}  \frac{(2k)_{m+n}}{i! j!} \left \| \frac{\mathrm d^{k-j}}{\mathrm dx^{k-j}} \Psi_n^{(i)} \Psi_n^{(j)}u_n \right \|_{L^2} \left \|u_n^{(k-i)}\right \|_{L^2} \\ 
&\qquad \le \frac{C }{n^2}\sum_{i=1}^{2k-1} \sum_{j=1}^{2k-m} \frac{(2k)_{m+n}}{i! j!} \|u_n\|_{H^k}\to 0 \qquad (n\to \infty),
\end{aligned}
\end{equation}
and we infer
\begin{equation}
\begin{aligned}
E^{(k)}&= \lim_{n\to \infty} \frac{1}{2}\langle u_n, Au_n\rangle + \frac{\mu}{p}\|u_n\|_{p}^p\\
&= \limsup_{n\to \infty}  \frac{1}{2}\langle \Psi_n u_n, A \Psi_nu_n \rangle + \frac{\mu}{p}\|\Psi_n u_n \|_p^p  \\
&\qquad + \limsup_{n\to \infty}\frac{1}{2}\langle \widetilde{\Psi_n} u_n, A \widetilde{\Psi_n} u_n \rangle+ \frac{\mu}{p}\|\widetilde{\Psi_n} u_n \|_p^p\\
&= \limsup_{n\to \infty} E^{(k)} (\Psi_n u_n) + \limsup_{n\to \infty} E^{(k)} (\widetilde{\Psi_n} u_n)
\end{aligned}
\end{equation}
and $E^{(k)}$ is (super-)additive with respect to the partition of unity $\{\Psi_n, \widetilde{\Psi_n}\}$ in Example \ref{ex:unity2}. 

\emph{Subadditivity.} 
To show the subadditivity, note that
\begin{equation}\label{eq:scalingarg2}
E_t^{(k)} = t \inf_{\substack{u\in H^1\\ \|u\|_{L^2}^2=1}} \left \{\frac{1}{2}\int_{\mathcal G} \left |u^{(k)}\right |^2+ V |u|^2\, \mathrm dx - t^{\frac{p-2}{2}} \frac{\mu}{p} \int_{\mathcal G} |u|^p\, \mathrm dx \right \}.
\end{equation} 
We deduce the property by showing that $t\mapsto E_t^{(k)}$ is a concave function. Indeed, the scaling defines a concave function and passing to the limit we deduce concavity of the functional. Hence,
\begin{equation}\label{eq:concavityarg2}
E_t^{(k)} \ge t E_1^{(k)}, \qquad t\in [0,1],
\end{equation}
so that 
\begin{equation}
E_t^{(k)} + E_{1-t}^{(k)} \ge E_1^{(k)}, \qquad t\in [0,1].
\end{equation}
For the strictness in the inequality it suffices to show strictness in the inequality \eqref{eq:concavityarg2}. Assume
\begin{equation}
E_t^{(k)} = t E_1^{(k)}
\end{equation}
for some $t\in (0,1)$ and let $u_n$ be a minimizing sequence for $E_t^{(k)}$, then in particular due to \eqref{eq:scalingarg2} 
\begin{equation}\label{eq:recipe1}
\int_{\mathcal G} |u_n|^p\, \mathrm dx \to 0 \qquad (n\to \infty).
\end{equation}
By density we may assume $u_n \in D(A)$ and we deduce
\begin{equation}
\begin{aligned}
E_t^{(k)} &= \lim_{n\to \infty} \frac{1}{2} \langle A u_n, u_n\rangle - \frac{\mu}{p} \int_{\mathcal G} |u|^p\, \mathrm dx \\
&= \lim_{n\to \infty} \frac{1}{2} \langle A u_n, u_n \rangle \ge \frac{\Sigma_0 t}{2}.
\end{aligned}
\end{equation}
Now suppose that a ground state to $E^{(k)}$ exists, i.e. there exists $u\in D(A)$ with $\|u\|_2^2=t$ and $\langle Au, u\rangle = \Sigma_0 t$, then
\begin{equation}\label{eq:recipe2}
E^{(k)}_t \le \frac{1}{2}\langle Au, u\rangle - \frac{\mu}{p}\int_{\mathcal G} |u|^p\, \mathrm dx < \frac{\Sigma_0 t}{2}.
\end{equation} 
This is a contradiction,  and we conclude strict subadditivity in this case. 

\end{proof}

\begin{proposition}\label{prop:second}
Assume $\mathcal G$ is a finite, connected metric graph and $\Sigma_0 < \Sigma$ as defined in \eqref{eq:numerate0} and \eqref{eq:numerate1}. Then there exists $\hat \mu >0$, such that for all $\mu \in (0, \hat \mu)$
\begin{equation}
\widetilde \Sigma_0^{(\mu,k)} := \inf_{\substack{u \in D(A)\\ \|\phi\|_2^2=1}} E^{(k)}(u) < \lim_{n\to \infty} \inf_{\substack{u \in D_R(A)\\\|\phi\|_2^2=1}} E^{(k)}(u) =: \widetilde \Sigma^{(\mu,k)}.
\end{equation}
\end{proposition}
\begin{proof}
Without loss of generality we may assume  $\Sigma_0>0$ as we otherwise can simply add a constant to the functional. In particular,
\begin{equation}
\|u\|= \left ( \int_{\mathcal G} |u^{(k)}|^2 + V |u|^2\, \mathrm dx \right )^{1/2}
\end{equation}
defines an equivalent norm on $H^k$ as shown in Lemma~\ref{lem:energyestimateneu}. 
Let $\varepsilon>0$. By Proposition~\ref{prop:higherordergag}
\begin{equation}
\|u\|_{L^p}^p \le C \|u\|_{L^2(\mathcal G)}^{\frac{(2k-1)p+2}{2k}} \|u\|^{\frac{p-2}{2k}}
\end{equation}
for some $C>0$. In particular, for sufficiently small $\mu$
with Young's inequality we have
\begin{equation}
\frac{\mu}{p} \|u\|_{L^p}^p\le \frac{\varepsilon}{2} \int_{\mathcal G} |u^{(k)}|^2+ V|u|^2\, \mathrm dx + \frac{\widetilde{C} \varepsilon}{2}
\end{equation}
for sufficiently small $\mu>0$ we deduce 
\begin{equation}
E^{(k)}(u) \ge \frac{1- \varepsilon}{2} \left ( \int_{\mathcal G} |u^{(k)}|^2+ V|u|^2\, \mathrm dx\right )- \frac{\widetilde{C} \varepsilon}{2}.
\end{equation}
Hence,
\begin{equation}
\widetilde{\Sigma}^{(\mu,k)} -\widetilde{\Sigma}_0^{(\mu,k)} \ge  \frac{1-\varepsilon}{2} \Sigma-\frac{\varepsilon}{2} - \frac{1}{2}\Sigma_0 = \frac{1}{2} \left (\Sigma - \Sigma_0\right ) -\frac{\varepsilon}{2} \left (\widetilde{C}+ \Sigma\right ). 
\end{equation}
Since $\varepsilon$ can be chosen arbitrarily small and $\Sigma > \Sigma_0>0$, we have for sufficiently small $\mu$
\begin{equation}
\widetilde \Sigma^{(\mu,k)} > \widetilde \Sigma^{(\mu,k)}_0.
\end{equation}
\end{proof}

\begin{theorem}\label{thm:bigresult2}
Let $\mathcal G$ be a finite, connected graph and let $c>0$. 
Assume $\Sigma_0 < \Sigma$ as defined in \eqref{eq:numerate0} and \eqref{eq:numerate1}, then $E_c^{(k)}(\mathcal G)$ admits a minimizer for $\mu \in (0, \hat \mu)$ as in Proposition \ref{prop:second} and let $\hat \mu$ be as in Proposition \ref{prop:second}. Then $E_c^{(k)}$ admits a minimizer for any $\mu \in (0,\hat \mu)$.
\end{theorem}

\begin{proof}
By Lemma \ref{lem:energyestimateneu} any minimizing sequence admits a weakly convergent subsequence. $E^{(k)}$ satisfies the prerequisites of Theorem \ref{thm:main1} and Theorem \ref{thm:main2} with $X(\mathcal G)= H^k(\mathcal G)$ and $Y(\mathcal G) = \widetilde{C_b^\infty}(\mathcal G)$ by Lemma \ref{lem:preconditions}. Then due to Proposition \ref{prop:second} the energy inequality in Corollary \ref{cor:existence} is satisfied. In particular we deduce existence of a minimizer of $E_c^{(k)}(\mathcal G)$ under the assumptions of the statement. 
\end{proof}

\subsection{The case $V\equiv 0$ on the real line}
Consider the infimization problem on the real line
\begin{equation}
E^{(k)}(\mathbb R) = \inf_{\substack{u\in H^1(\mathbb R)\\ \|u\|_2^2=1}} \frac{1}{2} \int_{\mathbb R} |u^{(k)}|^2\, \mathrm dx - \frac{\mu}{p} \int_{\mathbb R} |u|^p \, \mathrm dx.
\end{equation}
This is the special case when $V\equiv 0$ and $\mathcal G=\mathbb R$ in \eqref{eq:minimizationeq}.
In this case $E_c$ admits a minimizer due to Theorem \ref{thm:main3}:
\begin{theorem}\label{thm:bigresult3}
Let $V\equiv 0$ and $\mathcal G=\mathbb R$. The minimization problem
\begin{equation}
E^{(k)}(\mathbb R) = \inf_{\substack{u\in H^1(\mathbb R)\\ \left \|u\right \|_2^2=1}} E^{(k)}(u)
\end{equation}
admits a minimizer for all $\mu >0$. Furthermore, any minimizer of $E_c(\mathbb R)$ is $C^\infty_b$ and satisfies the Euler--Lagrange equation
\begin{equation}
(-1)^k u^{(2k)} + \lambda u = \mu |u|^{p-2}u, 
\end{equation}
where $\lambda \in \mathbb R$ is a Lagrange multiplier.
\end{theorem}
\begin{proof}
Due to Lemma \ref{lem:preconditions} it suffices to show that
$$t\mapsto E_{t}^{(k)}= \inf_{\substack{u\in H^1(\mathbb R)\\ \|u\|_2^2=t}}\frac{1}{2} \int_{\mathbb R} |u^{(k)}|^2\, \mathrm dx - \frac{\mu}{p} \int_{\mathbb R} |u|^p \, \mathrm dx.$$
is strictly subadditive, then the prerequisites of Theorem \ref{thm:main3} are satisfied. The Euler--Lagrange equation is satisfied because of Proposition \ref{prop:NLSmult}. The regularity is due to elliptic regularity and a bootstrap argument.

For a contradiction, assume as in the proof of Lemma \ref{lem:preconditions}
\begin{equation}
E_t = t E_1
\end{equation}
for some $t\in [0,c]$. Assume $u_n$ is a minimizing sequence for $E_t$, then as in the proof of Lemma \ref{lem:preconditions} we have
\begin{equation}\label{eq:rcase}
\lim_{n\to \infty} \int_{\mathbb R} |u_n|^p\, \mathrm dx = 0.
\end{equation}
Given a test function $\phi \in C_c^\infty(\mathbb R)$ with $\|\phi\|_{L^2}^2=1$, we define the rescaling for $\lambda >0$
\begin{equation}
\phi_{\lambda} := \lambda^{1/2} \phi( \lambda x).
\end{equation}
Then $\|\phi_{\lambda}\|_{L^2}^2=1$  for all $\lambda >0$. Then
\begin{equation}
E^{(k)}(\phi_{\lambda}) = \lambda^{2k} \int_{\mathbb R} |\phi^{(k)}|^2\, \mathrm dx - \lambda^{\frac{p}{2}-1} \int_{\mathbb R} |\phi|^p \, \mathrm dx. 
\end{equation}
In particular for sufficiently small $\lambda>0$
\begin{equation}
E^{(k)} \le E^{(k)}(\phi_{\lambda}) <0.
\end{equation}
On the other hand with \eqref{eq:rcase}
\begin{equation}
E^{(k)} = \lim_{n\to \infty} E^{(k)}(u_n) = \lim_{n\to \infty} \int_{\mathcal G} |u_n^{(k)}|^2\, \mathrm dx\ge 0 
\end{equation}
and we conclude subadditivity due to the contradiction.  This concludes the proof.
\end{proof}

\begin{corollary}\label{cor:mainresult3} Let $V\equiv 0$. Then
$$E^{(k)}(\mathbb R)<0.$$
\end{corollary}

\subsection{Decaying potentials}
In the following we study the minimization problem on finite metric graphs $\mathcal G$ under the assumption that $V=V_2+V_\infty$ with $V_2\in L^2(\mathcal G), V_\infty\in L^\infty(\mathcal G)$ such that
\begin{equation}\label{eq:decayingpotential}
V_\infty(x) \to 0 \qquad (x\to \infty)
\end{equation}
on all of the rays.  Consider the quantitities
\begin{equation}
\begin{aligned}
\widetilde{\Sigma}_0^{(\mu,k)} &= \inf_{\substack{\phi \in D(A)\\ \|\phi\|_{L^2}^2 =1}} E^{(k)}(u)\\
\widetilde{\Sigma}^{(\mu,k)} &= \lim_{R\to \infty} \inf_{\substack{\phi \in D_R(A)\\ \|\phi\|_{L^2}^2 =1}} E^{(k)}(u)\\
E^{(k)}(\mathbb R) &= \inf_{\substack{u\in H^1(\mathbb R)\\ \|u\|_{L^2(\mathbb R)}^2=1}} \frac{1}{2} \int_{\mathbb R} |u^{(k)}|^2\, \mathrm dx - \frac{\mu}{p} \int_{\mathbb R} |u|^p \, \mathrm dx.
\end{aligned}
\end{equation}
\begin{lemma}\label{lem:decayingpotential}
Let $\mathcal G$ be a finite metric graph and assume that $V\in L^2+L^\infty(\mathcal G)$ satisfies \eqref{eq:decayingpotential}. Then
\begin{equation}
\widetilde{\Sigma}^{(\mu,k)} = E^{(k)}(\mathbb R).
\end{equation}
\end{lemma}
\begin{proof}
Assume $\phi$ is a minimizer of $E^{(k)}(\mathbb R)$, which exists due to Theorem~\ref{thm:bigresult3}. Due to density, we can  consider a minimizing sequence $u_n$ for $E^{(k)}(\mathbb R)$ in $C_c^\infty(\mathbb R)$ satisfying $\|u_n\|_{2}^2=1$, such that $u_n \to \phi$ strongly in $H^1$ as $n\to \infty$. Now by translation invariance we may assume that $u_n$ is supported in $[n,\infty)$ for $n\in \mathbb N$.  identifying the half-line with one of the rays of $\mathcal G$, we may consider $u_n$ as a function in $H^1(\mathcal G)$. Then
\begin{equation}
\begin{aligned}
\left |\int_{\mathcal G} V |u_n|^2\, \mathrm dx \right | &=  \left |\int_{\mathcal G\setminus K_n} V |u_n|^2\, \mathrm dx \right |\\
&\le C \left (\sup_{x\in \mathcal G\setminus K_n} |V_\infty(x)|^2 +\int_{\mathcal G\setminus K_n} |V_2|^2\, \mathrm dx\right )\to 0 \qquad (n\to \infty)  
\end{aligned}
\end{equation}
and we compute
\begin{equation}
\begin{aligned}
E^{(k)}(\mathbb R)&=\lim_{n\to \infty}E^{(k)}(u_n)\\
&= \lim_{n\to \infty} \frac{1}{2}\int_{\mathcal G} |u^{(k)}_n|^2 \, \mathrm dx - \frac{\mu}{p} \int_{\mathcal G} |u|^p\, \mathrm dx\\
&= \lim_{n\to \infty} E^{(k)}(u_n)\ge \widetilde{\Sigma}^{(\mu,k)}.
\end{aligned}
\end{equation}
On the other hand given a minimizing sequence $u_n$ for $\widetilde{\Sigma}^{(\mu,k)}$, such that $\operatorname{supp} u_n \subset \mathcal G \setminus K_n$ then the functions in the sequence are supported on each of the rays and
\begin{equation}\label{eq:goesto0}
\begin{aligned}
&\left |\int_{\mathcal G} V |u_n|^2\, \mathrm dx \right | =  \left |\int_{\mathcal G\setminus K_n} V |u_n|^2\, \mathrm dx \right |\\
&\qquad\le C \left (\sup_{x\in \mathcal G\setminus K_n} |V_\infty(x)|+\left (\int_{\mathcal G\setminus K_n} |V_2|^2\, \mathrm dx\right )^{1/2}\right )\to 0 \qquad (n\to \infty).
\end{aligned}
\end{equation}
By density we can consider a collection of sequences $u_n^{(1)}, \ldots, u_n^{(|\mathcal E_\infty|)}$ in $C_c^\infty(\mathbb R)$ and choose them to have disjoint supports. Then if we define
\begin{equation}
\widetilde{u}_n := \sum_{i=1}^{|\mathcal E_\infty|} u_n^{(i)}.
\end{equation}
Then with \eqref{eq:goesto0} we compute
\begin{equation}
\begin{aligned}
\widetilde{\Sigma}^{(\mu,k)} &= \lim_{n\to \infty} \sum_{i=1}^{\mathcal E_\infty} E^{(k)}(u_n^{(i)})\\
&= \lim_{n\to \infty} E^{(k)}\left ( \widetilde{u}_n \right ) \ge E^{(k)}(\mathbb R). \\
\end{aligned}
\end{equation}
\end{proof}

\begin{remark}\label{rmk:decayingpotential}
Suppose $\mathcal G$ is a locally finite metric graph with at least one ray. Then the inequality
\begin{equation}
\widetilde{\Sigma}^{(\mu,k)}\le  E^{(k)}(\mathbb R)
\end{equation}
can still be shown as in the proof of Lemma \ref{lem:decayingpotential} using the test function argument on the half-line.
\end{remark}

\begin{theorem}\label{thm:decayingpotential}
Let $\mathcal G$ be a finite metric graph. Assume $V\in L^2+L^\infty(\mathcal G)$ satisfies \eqref{eq:decayingpotential}, then $E^{(k)}$ is strictly subadditive and if
\begin{equation}
\widetilde{\Sigma}_0^{(\mu,k)} < E^{(k)}(\mathbb R)
\end{equation}
then there exists a minimizer to the minimization problem.
\end{theorem}

\begin{proof}
By Lemma \ref{lem:energyestimateneu} any minimizing sequence admits a weakly convergent subsequence. By Lemma \ref{lem:preconditions} and Lemma~\ref{lem:decayingpotential}  it suffices to prove the strict subadditivity of $E^{(k)}$.  As in Lemma \ref{lem:preconditions} we can argue by contradiction. Assume namely that
\begin{equation}
E_t^{(k)}= t E_1^{(k)}
\end{equation}
for some $t\in (0,1)$ and let $u_n$ be a minimizing sequence for $E_t^{(k)}$, then in particular
\begin{equation}
\int_{\mathcal G} |u_n|^p \, \mathrm dx \to 0 \qquad (n\to \infty).
\end{equation}
But then $u_n$ is a vanishing sequence and passing to a subsequence still denoted by~$u_n$, we deduce with superaddditivity with respect to a sequence of partitions of unity as defined in Example~\ref{ex:unity2}
\begin{equation}\label{eq:reipse2}
\begin{aligned}
E_t^{(k)} &= \limsup_{n\to \infty} E^{(k)}(\Psi_n u_n) +\limsup_{n\to \infty} E^{(k)}(\widetilde{\Psi_n} u_n) \\
&\ge \frac{1}{2}\lim_{n\to \infty}\inf_{\substack{\phi\in H^1\\ \|u\|_{2}^2=t,\, \operatorname{supp} u \subset \mathcal G\setminus K_n}}\langle A u, u\rangle\\
&\ge -\frac{C}{2} \lim_{n\to \infty} \left (\sup_{x\in \mathcal G\setminus K_n} |V_\infty(x)|+ \left (\int_{\mathcal G\setminus K_n} |V_2|^2\, \mathrm dx\right )^{1/2}\right ) =0,
\end{aligned}
\end{equation}
since $\|u_n\|_{H^1}\le C$ for some $C>0$ by Lemma~\ref{lem:energyestimateneu}. On the other hand, by Lemma~\ref{lem:decayingpotential} and Corollary~\ref{cor:mainresult3}
\begin{equation}
\widetilde{\Sigma}^{(\mu,k)} = E^{(k)}(\mathbb R)<0
\end{equation}
and by contradiction we deduce strict subadditivity.

Hence, the prerequisites of Theorem~\ref{thm:main1} and Theorem~\ref{thm:main2} are satisfied with $X(\mathcal G)= H^k(\mathcal G)$ and $Y(\mathcal G) = \widetilde{C_b^\infty}(\mathcal G)$. Then due to Proposition \ref{prop:second} the energy inequality in Corollary \ref{cor:existence} is satisfied. In particular we deduce existence of a minimizer of $E^{(k)}(\mathcal G)$ under the stated assumptions. 
\end{proof}

\begin{example}\label{ex:shouldbeputinintroduction}
Let $\mathcal G$ be a finite metric graph and let $V\in L^2+L^\infty(\mathcal G)$ satisfy \eqref{eq:decayingpotential}. Similarly as in Lemma~\ref{lem:decayingpotential} we can show
\begin{equation}
\Sigma=\lim_{R\to \infty} \inf_{\substack{\phi \in D_R(A)\\ \|\phi\|_{L^2}^2 =1}} \langle \phi, A\phi\rangle_{L^2}=0.
\end{equation} 
In particular if $\Sigma_0<0$, then by Theorem~\ref{thm:bigresult2} there exists $\hat \mu>0$, such that for $\mu \in (0,\hat\mu]$ there exists a minimizer to $E^{(1)}$. 
As in \cite{cacciapuoti2018existence} one can show due to scaling properties that
\begin{equation}
\Sigma^{(\mu,1)}_0 < \Sigma_0  \le \gamma_p \mu^{\frac{4}{6-p}}= E^{(1)}(\mathbb R)
\end{equation}
for some $\gamma_p<0$ and $0<\mu\le (\Sigma_0/\gamma_p)^{\frac{3}{2}-\frac{p}{4}}$. In particular, we can deduce existence of minimizers for $E^{(1)}$ and
$0<\mu\le(\Sigma_0/\gamma_p)^{\frac{6-p}{4}}$ by Theorem~\ref{thm:decayingpotential}.  
\end{example}

\subsection{On the Threshold condition}\label{sec:onthreshold}
In this section we study the quantities $\Sigma_0$ and $\Sigma$ that appeared in the applications of the previous sections. For details on the definitions and characterizations of the essential spectrum we refer to \cite[§VII]{reedmethods}.

\subsubsection{A Persson type theory for the Polylaplacian on metric graphs} \label{subsubsec:perssontypetheorypoly}
Let $\mathcal G= (\mathcal V, \mathcal E)$ be a connected metric finite metric graph in this first part of the section. In particular $\mathcal G$ consists of a compact core $K$ with rays $\mathcal E_\infty\subset \mathcal E$ attached to $K$. 
Consider the Polylaplacian on $\mathcal G$ defined edgewise on $H^{2k}(\mathcal G)$:
\begin{equation}
A= (-\Delta)^k,\qquad D(A) =H^{2k}(\mathcal G)
\end{equation}
Combining Lemma \ref{lem:IMSunity} and the abstract decomposition formula in Lemma \ref{lem:decomposition2} we have the decomposition formula for the Polylaplacian:
\begin{lemma}\label{lem:decomposition3}
Let $\mathcal G= (V, \mathcal E)$ be a connected finite metric graph and assume $\{\Psi_1, \ldots, \Psi_N\}$ to be a partition of unity subordinate to an open covering $\mathcal O=\{O_1, \ldots, O_n\}$ satisfying
\begin{equation}
\sum_{k=1}^N \Psi_k^2 \equiv 1.
\end{equation}
Then 
\begin{equation}\label{eq:decomposition3}
\begin{aligned}
A\phi&= \sum_{j=1}^k \Psi_j A \Psi_j\phi+ \frac{(-1)^k}{2} \sum_{m=1}^{2k} \sum_{n=1}^{2k-m} \frac{(2k)_{m+n}}{m! n!} \Psi_j^{(m)} \Psi_j^{(n)} \phi^{(2k-m-n)}
\end{aligned}
\end{equation}
for all $\phi \in D(A)$.
\end{lemma}

We recall that $K=\mathcal G\setminus \mathcal E_\infty$ is the core of the graph and that for  $R>0$ 
\begin{align}
D_R&= \{\phi \in D(A) |\operatorname{supp}(\phi) \subset \mathcal G \setminus K_R\}\\
\Sigma_R&= \inf\{\langle \phi, A\phi\rangle|\phi \in D_R, \|\phi\|_{2}^2=1\}.
\end{align} 

Since $D(A)$ is nontrivial and invariant under multiplication by test functions in $\widetilde{C_c^\infty}$ the set $D_R$ is nonempty. 

For $R=0$ we set
\begin{align}
D_0&= D(A)\\
\Sigma_0&= \inf\{\langle \phi, A\phi \rangle |\phi \in D(A), \|\phi\|_{2}^2=1\}
\end{align}
and recall that
\begin{equation}
\Sigma = \lim_{R\to \infty} \Sigma_R=\sup_{R>0} \Sigma_R.
\end{equation}
In the following we characterize the quantities that were central to the existence theorems in the existence results before.
Since $A$ is self-adjoint one can show 
\begin{equation}
\Sigma_0 = \inf \sigma(A).
\end{equation}

\begin{theorem}\label{thm:persson}
Assume $\mathcal G$ is a finite metric graph. Let $A$ be a self-adjoint, nonnegative operator on $L^2(\mathcal G)$ that satisfies the decomposition formula \eqref{eq:decomposition3}. Additionally let $f(A+i)^{-1}$ be compact for all $f\in\widetilde{C_c^\infty}(\mathcal G)$. Then
\begin{equation}
\Sigma= \inf \sigma_{\text{ess}} (A).
\end{equation}
\end{theorem}
\begin{proof}
\textit{$\inf \sigma_{\text{ess}}(A)\ge \Sigma$.} Let $\lambda \in \sigma_{\text{ess}}(A)$ and let $(\phi_n)$ be an associated Weyl sequence satisfying $\|\phi_n\|_2^2=1$. 
Consider the vanishing-compatible sequence of partitions of unity $\Psi_{n}, \widetilde{\Psi_n}$ from Example \ref{ex:unity2}.

Since $\Psi_R^2 (A+i)^{-1}$ is compact for all $R>0$, and since $(A+i)\phi_n\rightharpoonup 0$ as $n\to \infty$ we deduce that
$$\|\Psi_R \phi_n\|_2= \|\Psi_R (A+i)^{-1} (A+i) \phi_n\|_2\to 0\qquad (n\to \infty)$$ 
and passing to a subsequence, still denoted by $\phi_n$, we may assume 
$$\|\Psi_n \phi_n\|_2= \|\Psi_n (A+i)^{-1} (A+i) \phi_n\|_2\to 0.$$
Furthermore, with \eqref{eq:ineedthis} we deduce that
\begin{equation}
\|\phi_n\|_{H^{2k}} \le C |\phi_n|_{H^{2k}}= C \left ( \|A \phi_n\|_2^2+ \|\phi_n\|_2^2\right )^{1/2}
\end{equation}
is uniformly bounded.  Since $\phi_n$ is a Weyl sequence for $\lambda \in \sigma_{\text{ess}}(A)$ with the decomposition formula in Lemma~\ref{lem:decomposition3} we then compute
\begin{equation}
\begin{aligned}
\lambda &= \lim_{n\to \infty} \langle \phi_n, A \phi_n\rangle_{L^2} \\ 
&= \lim_{n\to \infty} \langle \Psi_n\phi_n, A \Psi_n \phi_n\rangle_{L^2} + \langle \widetilde{\Psi_n} \phi_n, A \widetilde{\Psi_n} \phi_n\rangle_{L^2} + O\left (\frac{1}{n^2}\right )  \\
&\ge \lim_{n\to \infty} \sum_{e\in \mathcal E_\infty} \langle \widetilde{\Psi_n} \phi_n, A\widetilde{\Psi_n} \phi_n\rangle_{L^2}\ge  \lim_{n\to \infty} \Sigma_n = \Sigma.
\end{aligned}
\end{equation}
Since $\lambda \in \sigma_{\text{ess}}(A)$ was arbitrary, we conclude $\inf \sigma_{\text{ess}}(A)\ge \Sigma$.

\textit{$\inf \sigma_{\text{ess}}(A)\le \Sigma$.} Assume for a contradiction that $\inf \sigma_{\text{ess}}(A)\ge \Sigma+ 3\varepsilon$ with $\varepsilon >0$. Then $\sigma(A) \cap (-\infty, \Sigma+ 2\varepsilon]$ is discrete and since $A$ is bounded from below, the spectral projector $P_\Sigma:= P_{(-\infty, \Sigma+2\varepsilon]}$ is of finite rank. 
Assume $\phi_n \in D_n(A)$ is a sequence such that
\begin{equation}
\langle \phi_n, A\phi_n\rangle \le \Sigma + \varepsilon
\end{equation}
and $\phi_n \rightharpoonup 0$ in $L^2$. Then since $\left (A +\Sigma + 2\varepsilon\right )P_\Sigma $ is a compact operator and
\begin{equation}
\left (A +\Sigma + 2\varepsilon\right )P_\Sigma \phi_n\to 0 \qquad (n\to \infty).
\end{equation}
Hence
\begin{equation}
\begin{aligned}
\langle \phi_n, A\phi_n\rangle_{L^2} &= \langle \phi_n, A(1-P_{\Sigma})\phi_n\rangle + \langle \phi_n,  A P_\Sigma \phi_n\rangle_{L^2}\\
&\ge (\Sigma + 2\varepsilon) \langle \phi_n, (1-P_{\Sigma})\phi_n\rangle_{L^2} + \langle \phi_n, AP_\Sigma \phi_n\rangle_{L^2} \\
&\ge \Sigma+2\varepsilon + \left \langle \phi_n, \left (A +\Sigma + 2\varepsilon\right )P_\Sigma \phi_n \right \rangle_{L^2}.
\end{aligned}
\end{equation}
Passing to the limit we conclude
\begin{equation}
\liminf_{n\to \infty} \langle \phi_n, A\phi_n\rangle_{L^2} \ge \Sigma+2\varepsilon
\end{equation}
and we infer the statement by contradiction.
\end{proof}

\subsubsection{Sufficient conditions for the threshold condition for the Polylaplacian}

\begin{proposition}\label{thm:lastresultt2}
Let $\mathcal G$ be a finite, connected metric graph and $k\ge 1$. 
Assume $\sigma_\text{ess}((-\Delta)^k+V) \subset [0,\infty)$ and assume additionally either 
\begin{enumerate}[(i)]
\item $V\in L^1(\mathcal G)\cap L^2(\mathcal G)$ and
\begin{equation}
\int_{\mathcal G}  V \, \mathrm dx <0
\end{equation}
 \item or $V<0$ on $\mathcal G$. 
\end{enumerate}
Then $\Sigma_0 < \Sigma$ (as defined in \eqref{eq:numerate02} and \eqref{eq:numerate12}) and there exists $\hat \mu>0$, such that the minimization problem
\begin{equation}\label{eq:minimizerklast}
E^{(k)}=\inf_{\substack{u\in H^1(\mathcal G)\\\|u\|_{2}^2=1}}  E^{(k)}(u)
\end{equation}
admits a minimizer for $\mu \in (0,\hat \mu)$. 
\end{proposition}
\begin{proof}
The proof is analagous to the proof in Proposition \ref{thm:lastresultt1}. We only need to replace the test functions $\Psi_n$ with the ones in Example \ref{ex:unity2}. Then 
$$\|\Psi_n^{(k)}\|_\infty \le \frac{1}{n^{2k}}C$$ 
with $C$ independent of $n$. We infer the result as in the proof of Proposition \ref{thm:lastresultt1}.

Consider  as test functions $\Psi_n$ as defined as in Example~\ref{ex:unity2}, then we only need to show that for $n$ sufficiently high, the Rayleigh quotient
\begin{equation}
    \mathcal R[\Psi_n]:= \frac{\int_{\mathcal G} |\Psi_n^{(k)}|^2 + V|\Psi_n|^2\, \mathrm dx}{\int_{\mathcal G}|\Psi_n|^2\, \mathrm dx} <0.
\end{equation}
Since $\|\Psi_n'\|_\infty^2 \le \frac{C}{n^2}$ for some $C>0$ we have
\begin{equation}
    \|\Psi'\|_2^2 \le \frac{C |\mathcal E_\infty|}{n}\to 0 \qquad (n\to \infty). 
\end{equation}
If $V<0$ then for sufficiently large $n$ and $\varepsilon>0$ sufficiently small
\begin{equation}
    \int_{\mathcal G} V|\Psi_n|^2\, \mathrm dx \le - |\{x\in \mathcal G: V(x) \le -\varepsilon\}| \varepsilon <0.
\end{equation}
If $\int_{\mathcal G} V\, \mathrm dx <0$, then
\begin{equation}
    \liminf_{n\to \infty}\int_{\mathcal G} V |\Psi_n|^2\, \mathrm dx - \int_{\mathcal G} V\, \mathrm dx<0
\end{equation}
We deduce $R[\Psi_n]<0$ for sufficiently large $n$ and thus $\inf \sigma((-\Delta)^k+V) <0$. Then $\Sigma_0<\Sigma$ and we conclude the existence of a minimizer of \eqref{eq:minimizerklast} by Theorem~\ref{thm:bigresult2}. 
\end{proof}

\begin{remark}
If $V\in L^2+L^\infty(\mathcal G)$ is a relativly compact perturbation of $(-\Delta)^k$, i.e. 
\begin{equation}
V\left ((-\Delta)^k+i\right )^{-1}
\end{equation}
is compact, then $\inf \sigma_\text{ess}\left ((-\Delta)^k+V\right ) =0$  and we deduce
\begin{equation}
\inf \sigma_\text{ess}\left ((-\Delta)^k+V\right )\subset [0,\infty).
\end{equation}
\end{remark}

We finish the section by giving a criterion for the potential $V$ such that
\begin{equation}
\Sigma = \lim_{n\to \infty} \inf_{\substack{u\in D(A) \\ \|u\|_2^2=1, \; \operatorname{supp} u\subset \mathcal G\setminus K_n}} \langle u, Au\rangle \ge 0.
\end{equation}
This in particular implies 
$$\sigma_\text{ess}((-\Delta)^k+V) \subset [0,\infty).$$
Consider decaying potentials $V=V_2+ V_\infty$ with $V_2\in L^2(\mathcal G)$ and $V_\infty\in L^\infty(\mathcal G)$ such that
\begin{equation}\label{eq:newdecayingpotentialsf}
\sup_{x\in \mathcal G\setminus K_n} |V_\infty(x)|\to 0 \qquad (n\to \infty).
\end{equation}
\begin{proposition}\label{prop:Iwantthistoendf}
Let $\mathcal G$ be a finite metric graph.  Assume $V\in L^2+ L^\infty(\mathcal G)$ satisfying \eqref{eq:newdecayingpotentialsf}. Let $A=(-\Delta)^k+V$, then
\begin{equation}
\widetilde{\Sigma} = \lim_{n\to \infty} \inf_{\substack{u\in H^{2k}(\mathcal G)\\ \|u\|_2^2=1, \; \operatorname{supp} u\subset \mathcal G\setminus K_n}} \langle u, Au\rangle_{L^2} = 0.
\end{equation}
\end{proposition}

\begin{proof}
Assume $u_n$ is a minimizing sequence, such that $\|u_n\|_{L^2}^2$, $\operatorname{supp} u\subset \mathcal G\setminus K_n$ and 
\begin{equation}
\langle u_n, Au_n\rangle_{L^2}\to \Sigma.
\end{equation}
With \eqref{eq:ineedthis} we deduce that
\begin{equation}\label{eq:ineedthis2f}
\|u_n\|_{H^k} \le C \left (\langle u_n, Au_n\rangle_{L^2}^2 + \|u_n\|_2^2\right )
\end{equation}
is uniformly bounded.  Integrating by parts and using \eqref{eq:ineedthis2f} we infer
\begin{equation}
\begin{aligned}
\int_{\mathcal G} \left |u^{(k)}_n\right |^2 + V|u_n|^2\, \mathrm dx&\ge \int_{\mathcal G} \left |u^{(k)}_n\right |^2\, \mathrm dx \\
&\qquad - \widetilde{C} \left ( \left (\int_{\mathcal G\setminus K_n} |V|^2\, \mathrm dx\right )^{1/2}  + \sup_{x\in \mathcal  G\setminus K_n} |V_\infty(x)|\right ).
\end{aligned}
\end{equation}
We have
\begin{equation}
\left ( \left (\int_{\mathcal G\setminus K_n} |V|^2\, \mathrm dx\right )^{1/2}  + \sup_{x\in \mathcal  G\setminus K_n} |V_\infty(x)|\right )\to 0 \qquad (n\to \infty).
\end{equation}
Thus,
\begin{equation}
\Sigma = \lim_{n\to \infty} \langle u_n, Au_n\rangle_{L^2} \ge 0.\qedhere
\end{equation}
In fact, suppose $\phi \in C_c^\infty(\mathbb R)$, then we can define $$\phi_n:= \frac{1}{\sqrt{n}} \phi\left (\frac{x}{n}\right ).$$
We imbed $\phi_n$ to a function $\phi \in \widetilde{C_b^\infty}(\mathcal G)$ on the graph by defining it on one ray and then extending it to the rest of the graph by zero and w.l.o.g. we may assume $\operatorname{supp} \phi_n \subset \mathcal G\setminus K_n$, then we have
\begin{equation}
\begin{aligned}
    \widetilde{\Sigma} &\le \lim_{n\to \infty} \int_{\mathcal G} |\phi_n^{(k)}|^k +  V|\phi_n|^2 \, \mathrm dx \\
    &\le \int_{\mathcal G} \frac{1}{n^k}|\phi^{(k}|\, \mathrm dx - \overline  C \left ( \left (\int_{\mathcal G\setminus K_n} |V|^2\, \mathrm dx\right )^{1/2}  + \sup_{x\in \mathcal  G\setminus K_n} |V_\infty(x)|\right )=0.
\end{aligned}
\end{equation}
\end{proof}

\section{Locally Finite Graphs}
In this section, we study the NLS energy functional with potentials on more general graphs. We show a decomposition formula for the form associated with the magnetic Schrödinger operator and adapt previous arguments by introducing a suitable sequence of partitions of unity in the case of locally finite metric graphs.
\subsection{Formulation of the problem}
Consider the Schödinger operator with potentials $M\in H^1+W^{1,\infty}(\mathcal G)$ and $V\in L^2+L^\infty(\mathcal G)$ satisfying natural vertex conditions on $\mathcal G$:
\begin{equation}\label{eq:schroedinger2}
\begin{gathered}
A=\left (i \frac{\mathrm d}{\mathrm dx}+M\right )^2+V \\
D(A)=\left \{ u\in \widetilde{H^2}(\mathcal G) \bigg | \sum_{e \succ \mathsf v} \left ( i \frac{\partial}{\partial\nu} +  M \right ) u_e (\mathsf v)=0, \qquad \forall \mathsf v \in \mathcal V \right \}
\end{gathered}
\end{equation}
associated to the form $a: H^1(\mathcal G) \times H^1(\mathcal G) \to \mathbb C$ defined via
\begin{equation}\label{eq:schroedingerform2}
a(u,v) := \int_{\mathcal G} \overline{ \left ( i \frac{\mathrm d}{\mathrm dx} + M \right )u}  \left ( i \frac{\mathrm d}{\mathrm dx} + M \right )v+ V \overline u v\, \mathrm dx.
\end{equation}

\begin{theorem}
Let $\mathcal G$ be a metric graph, $M\in H^1+ W^{1,\infty}(\mathcal G)$ and $V\in L^2+L^\infty(\mathcal G)$. Then $A$ is a self-adjoint operator on $L^2(\mathcal G)$ and
\begin{equation}
    \langle Au, v\rangle = a(u,v)
\end{equation}
for all $u,v\in D(A)$.
\end{theorem}
\begin{proof}
For $u,v\in D(A)$ with integration by parts one easily computes
\begin{equation}
\left \langle Au, v\right \rangle_{L^2} = a(u,v) = \left  \langle Au, v \right  \rangle_{L^2}
\end{equation}
and $A$ is symmetric, hence $A\subset A^*$. Let $e\in \mathcal E$ be fix but arbitrary, then suppose $u\in D(A^*)$ and $v\in D(A)$ to be supported on $e$ and $v_e\in C_c^\infty(I_e)$, then
\begin{equation}
\left \langle u_e, (Av)_e \right \rangle_{L^2} = \left \langle  u, Av \right \rangle_{L^2} = \left \langle A^*u, v\right \rangle_{L^2} = \left \langle (A^*u)_e, v_e\right \rangle_{L^2}
\end{equation}
Since $e$ and $v$ were otherwise arbitrary we deduce
\begin{equation}
\left [\left ( i \frac{\mathrm d}{\mathrm dx} + M\right )^2 +V \right ]u_e\in L^2(I_e)
\end{equation}
in the distribution sense. Thus, $u_e\in H^2(I_e)$ for each $e\in \E$. Now, suppose $u\in D(A^*)$ and $v\in D(A)$ arbitrary. Then by partial integration we compute
\begin{equation}
\left \langle u, Av \right \rangle_{L^2} = \left \langle Au, v\right \rangle_{L^2} + \sum_{\mathsf v\in \mathcal V} \sum_{e\succ \mathsf v}   \left (\overline{u_e}  \left ( i \frac{\partial}{\partial \nu} + M \right ) v_e-v_e \overline{\left ( i \frac{\partial}{\partial \nu} +M \right )} u_e \right )(\mathsf v)
\end{equation}
Since the choice of $u,v$ is otherwise arbitrary we deduce $u\in C(\mathcal G)$ and we compute
\begin{equation}
\begin{aligned}
\left \langle A^* u, v\right \rangle &=\left \langle u, Av \right \rangle_{L^2} \\
&=\left \langle Au, v\right \rangle_{L^2}+\sum_{\mathsf v\in \mathcal V} \left (\overline{u(\mathsf v)} \left ( \sum_{e\succ \mathsf v} \left ( i \frac{\partial}{\partial \nu} +M \right ) v_e(\mathsf v)\right )\right .\\
&\qquad \qquad \qquad \qquad \qquad \qquad\left . -  v(\mathsf v) \left (\sum_{e\succ \mathsf v} \overline{\left ( i \frac{\partial}{\partial \nu} + M \right ) u_e(\mathsf v)}\right ) \right )\\
&= \left \langle Au, v\right \rangle_{L^2} + \sum_{v\in \mathcal V} v(\mathsf v) \left (\overline{\sum_{e\succ \mathsf v} \left ( i \frac{\partial}{\partial \nu} + M \right ) u_e(\mathsf v)}\right ).
\end{aligned}
\end{equation}
Hence $u\in D(A)$ and we infer $A=A^*$.
\end{proof}

Consider the NLS functional
\begin{equation}
E_\text{NLS}^{(\mathcal K)} (u) := \frac{1}{2}\int_{\mathcal G} \left |\left (i \frac{\mathrm d}{\mathrm dx}+M\right ) u\right |^2 + V |u|^2\, \mathrm dx- \frac{\mu}{p} \int_{\mathcal K}|u|^p\, \mathrm dx
\end{equation} 
where $V\in L^2+ L^\infty$ and $\mathcal K$ is a not necessarily bounded subgraph of $\mathcal G$. Define the corresponding minimization problem
\begin{equation}
E_{\text{NLS}}^{(\mathcal K)}:=\inf_{\substack{u\in H^1(\mathcal G)\\ \|u\|_2^2=1}} E_{\text{NLS}}^{(\mathcal K)} (u) 
\end{equation} 
similarly as in Section \ref{sec:existence}. We consider two cases:
\begin{itemize}
\item \emph{The localized case,} when $\mathcal K$ is a bounded subgraph of $\mathcal G$;
\item \emph{The global case,} when $\mathcal K= \mathcal G$ is the whole graph. In this case, we drop the argument and simply define
\begin{equation}
E_\text{NLS} (u) := \frac{1}{2}\int_{\mathcal G} \left |\left (i \frac{\mathrm d}{\mathrm dx}+M\right ) u\right |^2 + V |u|^2\, \mathrm dx- \frac{\mu}{p} \int_{\mathcal G}|u|^p\, \mathrm dx
\end{equation} 
and
\begin{equation}
E_{\text{NLS}}:=\inf_{\substack{u\in H^1(\mathcal G)\\ \|u\|_2^2=1}} E_{\text{NLS}} (u).
\end{equation} 
\end{itemize}
We define quantities analogous to \eqref{eq:numerateR}, \eqref{eq:numerate0} and \eqref{eq:numerate1}. 
Given a bounded subgraph of $\mathcal G$ and $R>0$ we define 
\begin{equation}
\begin{aligned}\label{eq:numerateR2}
D_R&:= \{\phi \in D(A) |\operatorname{supp}(\phi) \subset \mathcal G \setminus K_R\}\\
\Sigma_R&:= \inf\{\langle \phi, A\phi\rangle|\phi \in D_R, \|\phi\|_{2}^2=1\},
\end{aligned}
\end{equation} 
where $K_R$ was defined in \eqref{eq:core},
\begin{equation}\label{eq:numerate02}
\begin{aligned}
D_0&:= D(A)\\
\Sigma_0&:= \inf\{\langle \phi, A\phi \rangle |\phi \in D(A), \|\phi\|_{2}^2=1\}
\end{aligned}
\end{equation}
and 
\begin{equation}\label{eq:numerate12}
\Sigma := \lim_{R\to \infty} \Sigma_R=\sup_{R>0} \Sigma_R.
\end{equation}
Most results can be extended simply to this case following the previous proofs. The principal difficulty lies in establishing superadditivity with respect to a suitable sequence of partitions of unity. We give a construction of such a sequence of partitions of unity in the following. 

\subsection{Partitions of unity in $W^{1,\infty}(\mathcal G)$}
Here we give an important example for a partition of unity in $W^{1,\infty}(\mathcal G)= C^{0,1}(\mathcal G)$. Given any partition of unity in $W^{1,\infty}(\mathcal G)$ one can always find a renormalization as in Lemma~\ref{lem:unity2}:

\begin{lemma}\label{lem:unity}
Let $\mathcal G$ be a connected, locally finite metric graph. Consider any finite open covering $\mathcal O$ of $\mathcal G$. Then there exists a partition of unity in $W^{1,\infty}(\mathcal G)$ subordinate to $\mathcal O$ satisfying
\begin{equation}
\sum_{O\in \mathcal O} \Psi^2_O \equiv 1.
\end{equation}
\end{lemma}
\begin{proof}
Consider a partition of unity $\{\psi_O\}_{O\in \mathcal O}$ on the graph subordinate to the open covering $\mathcal O$. Then we define
\begin{equation}
\Psi_O := \frac{\psi_O}{\sqrt{\sum_{O\in \mathcal O}  \psi_O^2}}
\end{equation}
for all $O\in \mathcal O$. As a product of uniformly bounded Lipschitz continuous functions, $\Psi_O$ is also one; and by Proposition \ref{prop:lipschitz}  we conclude $\psi_O\in W^{1,\infty}(\mathcal G)$. Moreover, $\sum_{O\in \mathcal O} \psi_O^2 \equiv 1$ by construction.
\end{proof}

\begin{example}\label{ex:first2}
Let $\mathcal G$ be a locally finite metric graph and let $K$ be some bounded, connected subgraph. Let $X(\mathcal G)= H^1(\mathcal G)$ and $Y(\mathcal G)= W^{1,\infty}(\mathcal G)$. Then $X(\mathcal G),Y(\mathcal G)$ satisfy Assumption \ref{as:assumption1} and Assumption \ref{as:assumption2}.  Consider the partition of unity in $W^{1,\infty}(\mathcal G)$
\begin{equation}
\psi(x) = \max\{\operatorname{dist}(\mathcal G\setminus K_{2}, x), 1\}, \qquad \widetilde {\psi}(x) = 1-\psi.
\end{equation}
We construct a sequence of partitions of unity via
\begin{equation}
\psi_n(x)= \frac{1}{n} \max\{\operatorname{dist}(\mathcal G\setminus K_{2n}, x), n\}, \qquad \widetilde {\psi_n}(x) = 1-\psi_n.
\end{equation}
By Lemma \ref{lem:unity} we can  rescale them in such a way that
\begin{equation}
\Psi_n^2 + \widetilde{\Psi_n}^2 \equiv 1.
\end{equation}
By definition the partitions $K_{2n}, \mathcal K\setminus K_n$ are vanishing-compatible and $\Psi_n, \widetilde{\Psi_n}$ is a vanishing-compatible sequence of partitions of unity.
\end{example}

\begin{definition}\label{df:kirchhoff}
Let $f\in C^{0,1}(\mathcal G)$. We call a point $x\in \mathcal G$ a Kirchhoff point of $f$ if one of the following holds:
\begin{enumerate}[(1)]
\item $x\in \mathcal V$ is a vertex of degree $d_x\neq 2$, the derivatives $f_e'(x)$ exist for all $e\succ x$, and $f$ satisfies the Kirchhoff condition
\begin{equation}
\sum_{e\succ x} \tfrac{\partial}{\partial\nu}f_e(x) =0,
\end{equation}
\item $x\in\mathcal G$ is an interior point of an edge (equivalently, a dummy vertex of degree~$2$), and $f$ is differentiable at $x$.
\end{enumerate}
We call the set
\begin{equation}
\mathcal N_f = \mathcal G\setminus \{x\in \mathcal G: x \text{ is a Kirchhoff point of }f\}
\end{equation}
the non-Kirchoff set of $f$.
\end{definition}

\begin{remark}\label{rmk:important}
The approach in Example~\ref{ex:unity2} to construct vanishing-compatible sequences of partitions of unity is not applicable due to the absense of a core graph here. Instead, We are going to consider the sequence of partitions of unity in Example \ref{ex:first2} instead. This concrete sequence has some interesting properties, such that
for all $n\in \mathbb N$
\begin{equation}\label{eq:interestingprop1}
\|\psi_n'\|_{L^\infty} = \frac{1}{n} \qquad \|\widetilde{\psi_n}'\|_{L^\infty} = \frac{1}{n}
\end{equation}
and in particular
\begin{equation}\label{eq:interestingprop2}
\|\Psi_n'\|_{L^\infty} \le \frac{C}{n} \qquad \|\widetilde{\Psi_n}'\|_{L^\infty} \le \frac{C}{n}
\end{equation}
for a $C=C(\mathcal G)$ only dependent on the graph.
\end{remark}

\subsection{A decomposition formula}
For the Schrödinger operator with magnetic potential 
\begin{equation}\label{eq:test}
\begin{gathered}
\widetilde{A}=\left (i \frac{\mathrm d}{\mathrm dx}+M\right )^{2}\\ 
D(\widetilde A)= \widetilde{H^2}(\mathcal G)
\end{gathered}
\end{equation}
one can show as in Section \ref{sec:decompositionpoly}, see Lemma~\ref{lem:decomposition2}:
\begin{lemma}\label{lem:decompositionmagnetic}
Let $\mathcal G$ be a locally finite connected metric graph. Let 
\begin{equation}
\begin{gathered}
\widetilde A:=\left (i \frac{\mathrm d}{\mathrm dx}+M\right )^{2}\\
D(\widetilde A):= \widetilde{H^2}(\mathcal G)
\end{gathered}
\end{equation}
edgewise defined, i.e. 
$$\left (\widetilde A \phi\right )_e = \widetilde A \phi_e.$$ Then $\widetilde A$ defines a closed operator on $L^2(\mathcal G)$ and satisfies
\begin{enumerate}[(i)]
\item $fD(\widetilde A)\subset D(\widetilde A)$ for all $f\in \widetilde{C^\infty}(\mathcal G)$.
\item Let $f\in \widetilde{C^\infty}(\mathcal G)$, then the operator $f\widetilde Af$ is given by
\begin{equation}\label{eq:formulareduced}
f\widetilde A f= \frac{1}{2} \left (f^2\widetilde A + \widetilde A f^2\right ) + |f'|^2
\end{equation}
\end{enumerate}
\end{lemma}
\begin{proof}
The proof is analogous to the one in Lemma \ref{lem:decomposition2}.
\end{proof}
\begin{remark}
\eqref{eq:formulareduced} does not uniquely determine an operator. Indeed  \eqref{eq:formulareduced} is the special case of \eqref{eq:decomposition2} when $k=1$. In particular, formula \eqref{eq:decomposition2} in the case $k=1$ holds for all self-adjoint realizations of the magnetic Schrödinger operators (e.g. \eqref{eq:schroedinger}) and independent of the choice of $M\in H^1+ W^{1,\infty}(\mathcal G)$.
\end{remark}
We will be interested in a decomposition lemma on the form associated to $A$ as given in \eqref{eq:schroedinger2}.
\begin{lemma}\label{lem:lemma5.7}
Let $\mathcal G$ be a locally finite, connected metric graph and $a(\cdot, \cdot)$ be the symmetric bilinearform given by 
\begin{equation}
a(u,v):= \int_{\mathcal G} \overline{\left ( i \frac{\mathrm d}{\mathrm dx}+M\right )u} \left ( i \frac{\mathrm d}{\mathrm dx}+M\right ) v \, \mathrm dx
\end{equation}
for $u,v\in H^1(\mathcal G)$ Then for $f\in W^{1,\infty}(\mathcal G)\cap \widetilde{C^\infty}(\mathcal G)$ we have
\begin{equation}\label{eq:IMSformulapre}
a(fu, fv) = \frac{1}{2} \left (a_{\mathcal G}(u, f^2 v)+ a(f^2 u, v)\right )+\langle |f'|^2 u, v\rangle_{L^2(\mathcal G)}.
\end{equation}
\end{lemma}
\begin{proof}
By Proposition~\ref{prop:densitylocfinite} we may assume $u,v\in H^2_{c}(\mathcal G)$ and $fu, fv\in \widetilde{H^2}(\mathcal G)\cap H^1_c(\mathcal G)$. Integrating by parts on an arbitrary bounded subgraph $K$ containing $\operatorname{supp}u$ and $\operatorname{supp}{v}$ we compute 
\begin{equation}
\begin{aligned}
a(fu,fv)&= -\int_{K} \overline {\left (fAf\right ) u} v \, \mathrm dx + \sum_{\mathsf v\in \mathcal N_f\cap K} \sum_{e\succ \mathsf v} \left [\overline{\left ( i \frac{\mathrm d}{\mathrm dx} + M\right ) fu}\right ]_e fv(\mathsf v)\\
&=-\int_{K}  \left (\overline{\frac{1}{2}\left (f^2\widetilde A + \widetilde A f^2\right )u  + |f'|^2 u}\right ) v\, \mathrm dx\\
&\qquad \qquad\quad \qquad\qquad + \sum_{\mathsf v\in \mathcal N_f\cap K} \sum_{e\succ \mathsf v} \left [\overline{\left ( i \frac{\mathrm d}{\mathrm dx} + M\right ) fu}\right ]_e fv(\mathsf v)\\
&= \frac{1}{2}  \left (a(u, f^2 v)+ a(f^2 u, v)\right )+ \int_{\mathcal G} |f'|^2 \overline u v\, \mathrm dx\\
&\qquad\qquad -  \sum_{\mathsf v\in \mathcal N_f\cap K} \sum_{e\succ \mathsf v} \frac{1}{2}\left [\overline{\left ( i \frac{\mathrm d}{\mathrm dx} + M\right ) f^2  u}\right ]_e v(\mathsf v)\\
&\qquad \qquad \qquad-\sum_{\mathsf v\in \mathcal N_f\cap K} \sum_{e\succ \mathsf v} \frac{1}{2}\left [\overline{\left ( i \frac{\mathrm d}{\mathrm dx} + M\right ) u}\right ]_e f^2v(\mathsf v)\\
&\qquad \qquad \qquad \qquad  + \sum_{\mathsf v\in \mathcal N_f\cap K} \sum_{e\succ \mathsf v} \left [\overline{\left ( i \frac{\mathrm d}{\mathrm dx} + M\right ) fu}\right ]_e fv(\mathsf v)\\
&= \frac{1}{2}  \left (a(u, f^2 v)+ a(f^2 u, v)\right )+ \int_{\mathcal G} |f'|^2 \overline u v\, \mathrm dx
\end{aligned}
\end{equation}
and the statement follows by density.
\end{proof}

\subsection{Existence of NLS ground state for a class of Schrödinger operators}
\subsubsection{The localized setting}

In the following we study the localized case. We also remark that some of the lemmas will also apply to the global case. For $t>0$ we define
\begin{equation}\label{eq:minimizationdefloc}
E_t^{(\mathcal K)}:= \inf_{\substack{u\in H^1(\mathcal G)\\\|u\|_{L^2}^2=t}} E_{\text{NLS}}^{(\mathcal K)}.
\end{equation}
\begin{lemma}\label{eq:NLSboundedbelow}
Let $\mathcal G$ be a connected locally finite metric graph. Let $\mathcal K$ be a not necessarily bounded subset of $\mathcal G$. The functional $E_{\text{NLS}}^{(\mathcal K)}$ under $L^2$-constraint $\|\cdot\|_{L^2}^2=1$ is bounded below for $2<p<6$.
\end{lemma}
\begin{proof}
From the Gagliardo--Nirenberg inequality \eqref{eq:Gagliardo-nirenberg} we have
\begin{equation}
\begin{aligned}
\int_{\mathcal K} |u|^p\, \mathrm dx &\le \int_{\mathcal G} |u|^p\, \mathrm dx \\
&\le \varepsilon \int_{\mathcal G} \left |\left ( i \frac{\mathrm d}{\mathrm dx} + M \right )u\right |^2+ V |u|^2\, \mathrm dx + C_{\varepsilon} \int_{\mathcal G} |u|^2\, \mathrm dx
\end{aligned}
\end{equation}
and therefore 
\begin{equation}
E_{\text{NLS}}^{(\mathcal K)}(u) \ge(1-\varepsilon) \int_{\mathcal G} \left |\left ( i \frac{\mathrm d}{\mathrm dx} + M \right )u\right |^2+ V |u|^2\, \mathrm dx - C_{\varepsilon} \ge   -C_{\varepsilon}.
\end{equation}
for all $u\in H^1(\mathcal G)$ satisfying $\|u\|_2^2=1$.
\end{proof}

\begin{lemma}\label{lem:energyinequality}
Let $\mathcal G$ be a locally finite, connected  metric graph. Assume $A= \left ( i \frac{\mathrm d}{\mathrm dx} + M\right )^2 + V$ admits a ground state, 
then
$$E_t^{(\mathcal K)}=\inf_{\substack{u\in H^1(\mathcal G)\\\|u\|_{L^2}^2=t}} E_{\text{NLS}}^{(\mathcal K)} \le \frac{\Sigma_0t}{2}.$$
The inequality is strict if the ground state does not vanish identically on $\mathcal K$
\end{lemma}
\begin{proof}
Assume $u$ is a ground state of $A= \left ( i \frac{\mathrm d}{\mathrm dx} + M\right )^2 + V$ with $\|u\|_{L^2}^2=t$, then 
\begin{equation}
E_{\text{NLS}}^{(\mathcal K)}(u) = \frac{\Sigma_0t}{2} - \frac{\mu}{p} \int_{\mathcal K} |u|^{p}\, \mathrm dx\le \frac{\Sigma_0}{2} t
\end{equation}
and the inequality is strict if $u$ is not identically vanishing on $\mathcal K$. In particular
$$\inf_{\substack{u\in H^1\\\|u\|_{L^2}^2=t}} E_{\text{NLS}}^{(\mathcal K)} \le \frac{\Sigma_0t}{2}$$
with strictness in the inequality if there exists a ground state, which is not identically vanishing on $\mathcal K$.
\end{proof}

\begin{lemma}\label{lem:prework}
Let $\mathcal G$ be a locally finite, connected metric graph and let $\mathcal K$ be any subgraph. Assume $A= \left ( i \frac{\mathrm d}{\mathrm dx} + M\right )^2 + V$ admits a ground state that is not identically vanishing on $\mathcal K$, then the functional $E_{\text{NLS}}^{(\mathcal K)}$ is  weak limit superadditive, superadditive with respect to the partition of unity in Example \ref{ex:first2} and $t\mapsto E_t$ as defined in \eqref{eq:minimizationdefloc}   is strictly subadditive. 
\end{lemma}
\begin{proof}
With the Minkowski inequality we have
\begin{equation}
\begin{aligned}
&\left ( \int_{\mathcal G}|u'|^2\mathrm dx\right )^{1/2}- \left (\int_{\mathcal G} |M|^2 |u|^2\, \mathrm dx\right )^{1/2}\\
&\qquad \qquad \le \left (\int_{\mathcal G} \left |\left ( i \frac{\mathrm d}{\mathrm dx} + M \right )u\right |^2\, \mathrm dx\right )^{1/2}\\
&\qquad \qquad \qquad \qquad\le  \left ( \int_{\mathcal G}|u'|^2\mathrm dx\right )^{1/2}+ \left (\int_{\mathcal G} |M|^2 |u|^2\, \mathrm dx\right )^{1/2}.
\end{aligned}
\end{equation}
Adding a constant to the potential similarly as in the proof of Lemma \ref{lem:preconditions} we may assume
\begin{equation}
\|u\|_{2, M,V} =\left (\int_{\mathcal G} \left |\left ( i \frac{\mathrm d}{\mathrm dx} + M \right )u\right |^2+ V |u|^2\, \mathrm dx\right )^{1/2}
\end{equation}
to define an equivalent norm on $H^1$.

\emph{Weak limit superadditivity.}
Assume $u_n \rightharpoonup u$ weakly in $H^1$, then up to a subsequence by the Brezis--Lieb Lemma and weak limit superadditivity (similarly as in the proof of Lemma \ref{lem:preconditions})
\begin{equation}
\limsup_{n\to \infty} E_{\text{NLS}}^{(\mathcal K)} (u_n) =E_{\text{NLS}}^{(\mathcal K)}(u)+ \limsup_{n\to \infty} E_{\text{NLS}}^{(\mathcal K)} (u-u_n)
\end{equation}
and $E_{\text{NLS}}$ is weak limit superadditive.

\emph{Superaddivity with respect to a sequence of partitions of unity.} 
For the superadditivity, since $u_n$ is vanishing, up to a subsequence 
$$\|\widetilde{\Psi_n} u_n\|_p^p- \|u_n\|_p^p\to 0 \qquad (n\to \infty).$$ 
Then using the decomposition formula \eqref{eq:IMSformulapre} we compute, similarly as in the proof of Lemma \ref{lem:preconditions}:
\begin{equation}
\begin{aligned}
\limsup_{n\to \infty}E_{\text{NLS}}^{(\mathcal K)}(u_n) &= \limsup_{n\to \infty} \frac{1}{2} a(u_n, u_n) - \frac{\mu}{p} \|u_n\|_p^p\\
&\ge \limsup_{n\to \infty} a(\widetilde{\Psi_n} u_n, \widetilde{\Psi_n}u_n) - \frac{\mu}{p} \|\widetilde{\Psi_n} u_n\|_p^p\\
&\qquad \qquad \qquad + a(\Psi_n u_n, u_n) - \frac{\mu}{p} \|\Psi_n u_n\|_p^p\\
&= \limsup_{n\to \infty} E_{\text{NLS}}^{(\mathcal K)}(\widetilde{\Psi_n}u_n)+ E_{\text{NLS}}^{(\mathcal K)}(\Psi_n u_n).
\end{aligned}
\end{equation}

\emph{Subadditivity.} 
To show the subadditivity, note that
\begin{equation}\label{eq:scalingarg}
E_t^{(\mathcal K)} = t \inf_{\substack{u\in H^1\\ \|u\|_{L^2}^2=1}} \left \{\frac{1}{2}\int_{\mathcal G} \left |\left ( i \frac{\mathrm d}{\mathrm dx} + M \right )u\right |^2+ V |u|^2\, \mathrm dx - t^{\frac{p-2}{2}} \frac{\mu}{p} \int_{K} |u|^p\, \mathrm dx \right \}.
\end{equation} 
We deduce the property by showing that $t\mapsto E_t^{(\mathcal K)}$ is a concave function. Indeed, the scaling defines a concave function and in the limit we deduce concavity of the functional. In particular
\begin{equation}\label{eq:concavityarg}
E_t^{(\mathcal K)} \ge t E_1^{(\mathcal K)}, \qquad t\in [0,1].
\end{equation}
Then 
\begin{equation}
E_t^{(\mathcal K)} + E_{1-t}^{(\mathcal K)} \ge E_1^{(\mathcal K)}, \qquad t\in [0,1].
\end{equation}
For the strictness in the inequality it suffices to show strictness in the inequality \eqref{eq:concavityarg}. Suppose for a contradiction
\begin{equation}
E_t^{(\mathcal K)} = t E_1^{(\mathcal K)}
\end{equation}
for some $t\in (0,1)$ and let $u_n$ be a minimizing sequence for $E_t$, then in particular due to \eqref{eq:scalingarg} 
\begin{equation}
\int_{\mathcal K} |u_n|^p\, \mathrm dx \to 0 \qquad (n\to \infty).
\end{equation} 
Then by density we may assume $u_n \in D(A)$ and we infer
\begin{equation}
\begin{aligned}
E_t^{(\mathcal K)} &= \lim_{n\to \infty} E_{NLS}^{(\mathcal K)} (u_n)\\
&\ge \frac{1}{2} \limsup_{n\to \infty} \langle Au_n, u_n\rangle \ge \frac{\Sigma_0 t}{2},
\end{aligned}
\end{equation}
which is a contradiction to the inequality in Lemma \ref{lem:energyinequality}.
\end{proof}

\begin{theorem}\label{thm:bigresulttt}
Let $\mathcal G$ be a connected, locally finite metric graph. Assume $A= \left ( i \frac{\mathrm d}{\mathrm dx} + M\right )^2 + V$ admits a ground state, which is not identically vanishing on $\mathcal K$, then $E_{\text{NLS}}^{(\mathcal K)}$ admits a minimizer for all $\mu >0$.
\end{theorem}
\begin{proof}
For $R>0$ sufficiently large since $\mathcal K$ is considered to be bounded
\begin{equation}
\begin{aligned}
\inf_{\substack{u\in D_R(A) \\ \|u\|_{L^2}^2=1}} E_{\text{NLS}}^{(\mathcal K)}(u)&= \inf_{\substack{u\in D_R(A)\\\|u\|_{L^2}^2=1}} \frac{1}{2}\langle Au, u\rangle\\
&\ge \inf_{\substack{u\in D(A) \\ \|u\|_{L^2}^2=1}} \frac{1}{2}\langle Au, u\rangle= \frac{\Sigma_0}{2}.
\end{aligned}
\end{equation}
In particular with Lemma~\ref{lem:energyinequality} we have
\begin{equation}
E_{\text{NLS}}^{(\mathcal K)} < \lim_{R\to \infty} \inf_{\substack{u\in D_R(A) \\ \|u\|_{L^2}^2=1}} E_{\text{NLS}}^{(\mathcal K)}(u)=: \widetilde{E_{\text{NLS}}^{(\mathcal K)}}.
\end{equation}
Due to Lemma \ref{lem:prework} the requirements of Theorem \ref{thm:main1} and \ref{thm:main2} are satisfied and up to a subsequence any minimizing sequence admits a strong limit in $L^2$ such that the limit achieves the minimum in $E_{\text{NLS}}^{(\mathcal K)}$.
\end{proof}

\begin{remark}\label{rem:importantvanishing}
If $M\equiv 0$ then we can assume that a ground state of $A$ is nonnegative and in fact by Hopf's maximum principle positive everywhere. In particular, any ground state of $A$ is not identically vanishing on any subset of $\mathcal G$. We will see in §5.5 that $\Sigma_0<\Sigma$ implies the existence of ground states of $A$. In particular, this condition becomes obsolete in this case.
\end{remark}

\subsubsection{The global setting $\mathcal K=\mathcal G$}

Consider now the global case, where we consider the functional
\begin{equation}
E_\text{NLS} (\phi) = \frac{1}{2}\int_{\mathcal G} \left |\left ( i \frac{\mathrm d}{\mathrm dx} + M \right )\phi\right |^2+ V |\phi|^2\, \mathrm dx- \frac{\mu}{p} \int_{\mathcal G}|\phi|^p\, \mathrm dx, \qquad \|\phi\|_{L^2}=1.
\end{equation}  
In the global case Lemma~\ref{lem:prework} applies since any ground state of the magnetic Schrödinger operator $A=\left ( i \frac{\mathrm d}{\mathrm dx}+M\right )^2+V$ is not identically zero. 
In the following we give a criterion for existence of ground states with regards to these quantities.
\begin{proposition}\label{prop:first}
Assume $\mathcal G$ is a locally finite, connected metric graph and $\Sigma_0 < \Sigma$. Then there exists $\hat \mu >0$, such that for all $\mu \in (0, \hat \mu)$.
\begin{equation}
\widetilde \Sigma_0^{(\mu)} := \inf_{\phi \in D(A)} E_{\text{NLS}}(\phi) < \lim_{R\to \infty}\inf_{\phi \in D_R(A)} E_{\text{NLS}}(\phi) =: \widetilde \Sigma^{(\mu)}.
\end{equation}
\end{proposition}
\begin{proof}
W.l.o.g. $\Sigma_0>0$; otherwise we simply  add a constant to the potential $V$. Let $0<\varepsilon<1$ arbitrary, which we will only fix later. With Proposition~\ref{cor:Gagliardo-nirenberg} we deduce (similarly as in Proposition~\ref{prop:second}) that for sufficiently small $\mu>0$ 
\begin{equation}
E_\text{NLS}(\phi) \ge \frac{1- \varepsilon}{2} \left ( \int_{\mathcal G} \left |\left ( i \frac{\mathrm d}{\mathrm dx} + M \right )\phi\right |^2+ V|\phi|^2\, \mathrm dx\right )- \frac{C\varepsilon}{2}.
\end{equation}
Then
\begin{equation}
\widetilde{\Sigma}^{(\mu)} -\widetilde{\Sigma}_0^{(\mu)} \ge  \frac{1-\varepsilon}{2} \Sigma-\frac{C\varepsilon}{2} - \frac{1}{2}\Sigma_0 = \frac{1}{2} \left (\Sigma - \Sigma_0\right ) -\frac{\varepsilon}{2} \left (\widetilde{C}+ \Sigma\right ). 
\end{equation}
Since $\varepsilon$ can be chosen arbitrarily small, we have for sufficiently small $\mu$
\begin{equation}
\widetilde \Sigma^{(\mu)} > \widetilde \Sigma^{(\mu)}_0.
\end{equation}
\end{proof}

\begin{lemma}\label{lem:energyinequality2}
Let $\mathcal G$ be a locally finite, connected metric graph. Assume $A= \left ( i \frac{\mathrm d}{\mathrm dx} + M\right )^2 + V$ admits a ground state, 
then
$$E_t=\inf_{\substack{u\in H^1(\mathcal G)\\\|u\|_{L^2}^2=t}} E_{\text{NLS}}(u) < \frac{\Sigma_0t}{2}.$$
\end{lemma}
\begin{proof}
In the nonlocalized case, we can proceed analogously to before. Given a ground state $u\in H^2$ we simply compute analogously as in Lemma \ref{lem:energyinequality}
$$
E_t < \frac{\Sigma_0 t}{2}.
$$
\end{proof}

\begin{lemma}\label{lem:essential}
Assume $\mathcal G$ is a locally finite, connected metric graph and $\Sigma_0 < \Sigma$. Then $ E_{\text{NLS}}$ is weak limit superadditive, superadditive with respect to the sequence of partitions of unity in Example \ref{ex:unity2} and $t\mapsto E_t$ defines a strictly subadditive functional.
\end{lemma}

\begin{proof}
The proof is analogous to the one in Lemma \ref{lem:prework} by simply replacing $\mathcal K$ with the whole graph. $\Sigma_0 < \Sigma$, as we will see later, implies by Theorem \ref{thm:persson2}
\begin{equation}
\inf \sigma\left ( \left ( i \frac{\mathrm d}{\mathrm dx} +M\right )^2 +V\right ) < \inf \sigma_{\text{ess}}\left ( \left ( i \frac{\mathrm d}{\mathrm dx} +M\right )^2 +V\right ).
\end{equation}
In particular, there exist discrete eigenvalues below the essential spectrum and $A$ admits a ground state. 
\end{proof}

\begin{theorem}\label{thm:bigresultt}
Let $\mathcal G$ be a locally finite, connected metric graph.
Assume $\Sigma_0 < \Sigma$, then for $\mu \in (0, \hat \mu)$ as in Proposition \ref{prop:second} 
$$E_\text{NLS} = \inf_{\substack{\phi \in H^1(\mathcal G)\\ \|u\|_2^2=1}} \frac{1}{2}\int_{\mathcal G} \left |\left ( i \frac{\mathrm d}{\mathrm dx} + M \right )\phi\right |^2+ V |\phi|^2\, \mathrm dx- \frac{\mu}{p} \int_{\mathcal G}|\phi|^p\, \mathrm dx$$
admits a minimizer.
\end{theorem}
\begin{proof}
By Lemma~\ref{lem:essential} the requirements of Theorem \ref{thm:main2} are satisfied. Furthermore the energy inequality in Corollary~\ref{cor:existence} is satisfied by Proposition~\ref{prop:first} and we infer the statement.
\end{proof}

\subsection{On the Threshold condition}\label{sec:onthreshold2}
In this section we study the quantities $\Sigma_0$ and $\Sigma$ that appeared in the applications of the previous sections. For details on the definitions and characterizations of the essential spectrum we refer to \cite[§VII]{reedmethods}.

\subsubsection{Schrödinger operators and IMS formula on locally finite metric graphs}
Let $\mathcal G=(\mathcal V, \mathcal E)$ be a locally finite metric graph  throughout the rest of the section. Consider the Schrödinger operators with potentials $M$ and $V$ satisfying natural vertex conditions on $\mathcal G$:
\begin{equation}\label{eq:schroedinger}
\begin{gathered}
A=\left (i \frac{\mathrm d}{\mathrm dx}+M\right )^2+V \\
D(A)=\left \{ u\in \widetilde{H^2}(\mathcal G) \bigg | \sum_{e \succ \mathsf v} \left ( i \frac{\partial}{\partial\nu} +  M \right ) u_e (\mathsf v)=0, \qquad \forall \mathsf v \in \mathcal V \right \}
\end{gathered}
\end{equation} 
In the case of magnetic Schrödinger operators 
$$A= \left ( i \frac{\mathrm d}{\mathrm dx}+M\right )^2 + V$$ 
on domains $\Omega \subset \mathbb R^N$ the decomposition formula can be obtained via the IMS formula (see also Example~\ref{ex:NLSclassicR})
\begin{equation}
A = \sum_{j=1}^k \Psi_k A \Psi_k  + |\Psi_k'|^2 \quad \text{where}\quad \sum_{j=1}^k \Psi_k^2 \equiv 1.
\end{equation}
When considering locally finite metric graphs, we may not use the approach as before because general locally finite metric graphs do not necessarily contain a core and we need to adapt the theory using sequences of partitions of unity as in Example \ref{ex:first2}.

\begin{definition}[IMS formula on locally finite metric graphs]\label{df:IMS}
Let $\mathcal G$ be a locally finite, connected metric graph. Let $A: D(A) \subset L^2(\mathcal G) \to L^2(\mathcal G)$ be a densely defined, self-adjoint operator and assume $a(\cdot, \cdot)$ is the associated symmetric, sesquilinear form, defined on $H^1(\mathcal G)$. We say $A$ satisfies the IMS formula if for all $f\in W^{1,\infty}(\mathcal G)\cap \widetilde{C^\infty}(\mathcal G)$ 
\begin{multline}\label{eq:IMSformula}
a(fu,fv) = \frac{1}{2} \left (a(u, f^2 v)+ a(f^2 u, v)\right )+\langle |f'|^2 u, v\rangle_{L^2}, \qquad \forall u,v\in D(A).
\end{multline}
\end{definition}

We showed in Lemma \ref{lem:lemma5.7} that the magnetic Schrödinger operator in \eqref{eq:schroedinger2} satisfies the IMS formula \eqref{eq:IMSformula}; and in particular the following result applies:

\begin{theorem}\label{thm:persson2}
Assume $\mathcal G$ is a locally finite, connected metric graph. Let $A$ be a self-adjoint, nonnegative operator on $L^2(\mathcal G)$ that satisfies the IMS formula \eqref{eq:IMSformula}. Additionally let $f(A+i)^{-1}$ be compact for all $f\in C_c^{0,1}\cap \widetilde{C^\infty}$ then
\begin{equation}
\Sigma= \inf \sigma_{\text{ess}} (A).
\end{equation}
\end{theorem}
\begin{proof}
$\inf \sigma_{\text{ess}}(A)\le \Sigma$ follows by an abstract argument analogous to the argument in Theorem~\ref{thm:persson}. To establish  
$$\inf \sigma_{\text{ess}}(A)\le \Sigma$$
consider $\lambda \in \sigma_{\text{ess}}(A)$ and let $(\phi_n)$ be an associated Weyl sequence satisfying 
$$\|\phi_n\|_2^2=1.$$ 
Consider the vanishing-compatible sequence of partitions of unity $\Psi_{n}, \widetilde{\Psi_n}$ from Example \ref{ex:first2}. 

Since $\Psi_R^2 (A+i)^{-1}$ is compact for all $R>0$, and since $(A+i)\phi_n\rightharpoonup 0$ as $n\to \infty$ we deduce
$$\|\Psi_R \phi_n\|_2= \|\Psi_R (A+i)^{-1} (A+i) \phi_n\|_2\to 0\qquad (n\to \infty)$$ 
and passing to a subsequence, still denoted by $\phi_n$, we may assume 
$$\|\Psi_n \phi_n\|_2= \|\Psi_n (A+i)^{-1} (A+i) \phi_n\|_2\to 0.$$

 Since $\phi_n$ is a Weyl sequence for $\lambda \in \sigma_{\text{ess}}(A)$, with the decomposition formula in Lemma~\ref{lem:decomposition3} we then compute
\begin{equation}
\begin{aligned}
\lambda &= \lim_{n\to \infty} \langle \phi_n, A \phi_n\rangle_{L^2} \\ 
&= \lim_{n\to \infty} a(\Psi_n \phi_n, \Psi_n \phi_n) + a(\widetilde{\Psi_n} \phi_n, \widetilde{\Psi_n}\phi_n) + O\left (\frac{1}{n^2}\right )  \\
&\ge \lim_{n\to \infty} \sum_{e\in \mathcal E_\infty} a(\widetilde{\Psi_n}\phi_n, \widetilde{\Psi_n}\phi_n)\ge  \lim_{n\to \infty} \Sigma_n = \Sigma.
\end{aligned}
\end{equation}
Since $\lambda \in \sigma_{\text{ess}}(A)$ was arbitrary, we conclude $\inf \sigma_{\text{ess}}(A)\ge \Sigma$.
\end{proof}

\subsubsection{Sufficient conditions for the threshold condition for the Schrödinger operators without magnetic potentials}\label{sec:subsecsuff}
In this section we obtain criteria for the threshold condition for the operator
\begin{equation}
\begin{gathered}
A= -\Delta + V\\
D(A)=H^{2}(\mathcal G).
\end{gathered}
\end{equation}
defined on general locally finite metric graphs satisfying a volume growth assumption.
\begin{proposition}\label{thm:lastresultt1}
Let $\mathcal G$ be a locally finite, connected metric graph and let $K$ be a connected, precompact subgraph. We suppose additionally the volume assumption
\begin{equation}\label{eq:volumegrowthassumpt}
\left |K_{2n} \setminus K_{n}\right |=o(n^2) \qquad (n\to \infty).
\end{equation}
Assume $\sigma_\text{ess}(-\Delta+V) \subset [0, \infty)$ and assume additionally either 
\begin{enumerate}[(i)]
\item $V\in L^1(\mathcal G)\cap L^2(\mathcal G)$ and
\begin{equation}
\int_{\mathcal G}  V \, \mathrm dx <0
\end{equation}
\item or $V<0$ on $\mathcal G$. 
\end{enumerate}
Then $\Sigma_0 < \Sigma$ (as defined in \eqref{eq:numerate02} and \eqref{eq:numerate12}) and there exists $\hat \mu>0$ such that the minimization problem
\begin{equation}\label{eq:minimizerlast}
E_{\text{NLS}}^V=\inf_{\substack{u\in H^1(\mathcal G)\\\|u\|_{2}^2=1}}  E_{\text{NLS}}^V(u)
\end{equation}
admits a minimizer for $\mu \in (0,\hat \mu)$.
\end{proposition}
\begin{proof}
Consider as a test function $\Psi_n$ as defined in Example \ref{ex:first2}, then we only need to show that for $n$ sufficiently high, the Rayleigh quotient
\begin{equation}
\mathcal R[\Psi_n] := \frac{\int_{\mathcal G} |\Psi'_n|^2 + V |\Psi_n|^2\, \mathrm dx }{\int_{\mathcal G} |\Psi_n|^2\, \mathrm dx}<0.
\end{equation}
Indeed, since $\|\Psi'_n\|_{\infty}^2\le O(\frac{1}{n^2})$ as $n\to \infty$ we deduce
$$\|\Psi'_n\|_2^2\le \|\Psi'_n\|_{\infty}^2 |K_{2n}\setminus K_n|\to 0 \qquad (n\to \infty).$$ 
If $V<0$ then for sufficiently large $n$ and $\varepsilon>0$ sufficiently small
\begin{equation}
\int_{\mathcal G} V |\Psi_n|^2\, \mathrm dx\le -\left |\left \{x\in \mathcal G: V(x)\le -\varepsilon\right \}\right | \varepsilon<0.
\end{equation} 
If $\int_{\mathcal G}  V \, \mathrm dx <0$, then
\begin{equation}
\liminf_{n\to \infty}\int_{\mathcal G} V |\Psi_n|^2 \, \mathrm dx= \int_{\mathcal G}  V\, \mathrm dx<0
\end{equation}
by dominated convergence. In particular for $n$ large enough
\begin{equation}
\int_{\mathcal G} V|\Psi_n|^2 \le \frac{1}{2} \int_{\mathcal G} V\, \mathrm dx <0.
\end{equation}
We deduce $R[\Psi_n]<0$ and thus $\inf \sigma(-\Delta+V)<0$. Then $\Sigma_0 < \Sigma$  and we conclude the existence of minimizers of \eqref{eq:minimizerlast} by Theorem \ref{thm:bigresultt}.
\end{proof}

We finish the section by giving a criterion for the potential $V$ such that
\begin{equation}
\Sigma = \lim_{n\to \infty} \inf_{\substack{u\in D(A) \\ \|u\|_2^2=1, \; \operatorname{supp} u\subset \mathcal G\setminus K_n}} \langle u, Au\rangle \ge 0.
\end{equation}
This in particular implies 
$$\sigma_\text{ess}(-\Delta+V) \subset [0,\infty).$$
Consider decaying potentials $V=V_2+ V_\infty$ with $V_2\in L^2(\mathcal G)$ and $V_\infty\in L^\infty(\mathcal G)$ such that
\begin{equation}\label{eq:newdecayingpotentials}
\sup_{x\in \mathcal G\setminus K_n} |V_\infty(x)|\to 0 \qquad (n\to \infty).
\end{equation}
\begin{proposition}\label{prop:Iwantthistoend}
Let $\mathcal G$ be a locally finite metric graph.  Assume $V\in L^2+ L^\infty(\mathcal G)$ satisfying \eqref{eq:newdecayingpotentials}. Let $A=-\Delta+V$, then
\begin{equation}
\widetilde{\Sigma} = \lim_{n\to \infty} \inf_{\substack{u\in H^{2k}(\mathcal G)\\ \|u\|_2^2=1, \; \operatorname{supp} u\subset \mathcal G\setminus K_n}} \langle u, Au\rangle_{L^2} \ge 0.
\end{equation}
\end{proposition}

\begin{proof}
Assume $u_n$ is a minimizing sequence, such that $\|u_n\|_{L^2}^2$, $\operatorname{supp} u\subset \mathcal G\setminus K_n$ and 
\begin{equation}
\langle u_n, Au_n\rangle_{L^2}\to \Sigma.
\end{equation}
With integration by parts we deduce that
\begin{equation}\label{eq:ineedthis2}
\|u_n\|_{H^k} = \langle u_n, Au_n\rangle_{L^2}^2 + \|u_n\|_2^2
\end{equation}
is uniformly bounded.  Integrating by parts we infer
\begin{equation}
\begin{aligned}
\int_{\mathcal G} \left |u_n'\right |^2 + V|u_n|^2\, \mathrm dx&\ge \int_{\mathcal G} \left |u_n'\right |^2\, \mathrm dx \\
&\qquad - \widetilde{C} \left ( \left (\int_{\mathcal G\setminus K_n} |V|^2\, \mathrm dx\right )^{1/2}  + \sup_{x\in \mathcal  G\setminus K_n} |V_\infty(x)|\right ).
\end{aligned}
\end{equation}
We have
\begin{equation}
\left ( \left (\int_{\mathcal G\setminus K_n} |V|^2\, \mathrm dx\right )^{1/2}  + \sup_{x\in \mathcal  G\setminus K_n} |V_\infty(x)|\right )\to 0 \qquad (n\to \infty).
\end{equation}
Thus,
\begin{equation}
\Sigma = \lim_{n\to \infty} \langle u_n, Au_n\rangle_{L^2} \ge 0.\qedhere
\end{equation}
\end{proof}

\begin{theorem}\label{thm:updatefill}
Let $\mu>0$ and $2<p<6$. Let $\mathcal G$ be a locally finite metric graph with at least one ray, and suppose $V\in L^2+ L^\infty(\mathcal G)$, then $E_{\text{NLS}}^V$ is strictly subadditive and
\begin{equation}
    E_{\text{NLS}}^V= \inf_{u\in D(E_{\text{NLS}}^V)}\frac{1}{2} \int_{\mathcal G} |u'|^2+ V|u|^2\, \mathrm dx - \frac{\mu}{p} \int_{\mathcal G}|u|^p\, \mathrm dx 
\end{equation}
admits a minimizer if
\begin{equation}
    E_{\text{NLS}}^V < \widetilde{E_{\text{NLS}}^0}.
\end{equation}
\end{theorem}
\begin{proof}
As in Lemma~\ref{lem:prework} we have weak limit superadditivity and superadditivity with respect to the partition of unity in Example~\ref{ex:first2}. For the strict subadditivity, analagous to the approach in the proof of Lemma~\ref{lem:prework} it is sufficient to show $E_t < t E_1$ for all $t\in(0,1)$ with $E_t$ defined via
\begin{equation}
    t\mapsto E_t:= \inf_{\substack{u\in H^1(\mathcal G)\\ \|u\|_2^2=t}}  E_{\text{NLS}}^V (u).
\end{equation}
Suppose for a contradiction $E_t= t E_1$ for some $t\in(0,1)$, then as in the proof of Lemma~\ref{lem:prework} we infer that a mininizing sequence $u\in D(E_{\text{NLS}}^V)$ needs to satisfy
\begin{equation}
    \int_{\mathcal G} |u_n|^p\, \mathrm dx \to 0 \qquad (n\to \infty).
\end{equation}
Hence, $u_n \rightharpoonup 0$ in $H^1(\mathcal G)$
and by superadditivity with respect to the partition of unity in Example~\ref{ex:first2} and Proposition~\ref{prop:Iwantthistoend} we infer
\begin{equation}
    E_1= \lim_{n\to \infty} E_{\text{NLS}}^V(u_n) \ge \widetilde{\Sigma}\ge 0,
\end{equation}
but by Remark~\ref{rmk:decayingpotential} we have
\begin{equation}
    E_{\text{NLS}}^V\le E_{\text{NLS}}(\mathbb R) <0
\end{equation}
and the statement follows since 
\begin{equation}
    \widetilde{E_{\text{NLS}}^V} = \widetilde{E_{\text{NLS}}^0} 
\end{equation}
since 
\begin{equation}
\left ( \left (\int_{\mathcal G\setminus K_n} |V|^2\, \mathrm dx\right )^{1/2}  + \sup_{x\in \mathcal  G\setminus K_n} |V_\infty(x)|\right )\to 0 \qquad (n\to \infty).
\end{equation}
\end{proof}

\begin{example}\label{ex:NLSclassic}
Let $\mathcal G$ be a locally finite, connected noncompact metric graph with at least one ray. Consider the NLS energy functional as considered in \cite{adami2016threshold}
\begin{equation}
E_{\text{NLS}}(u,\mathcal G) = \int_{\mathcal G} |u'|^2\, \mathrm dx - \frac{\mu}{p} \int_{\mathcal G} |u|^p\, \mathrm dx.
\end{equation}
Consider the minimization problem
\begin{equation}\label{eq:exeq}
E_{\text{NLS}}(\mathcal G) := \inf_{\substack{u\in H^1(\mathcal G) \\\|u\|_2^2=1}} E_{\text{NLS}}(u, \mathcal G).
\end{equation}
Then by Theorem~\ref{thm:updatefill} the functional $E_{\text{NLS}}$ satisfies the prerequisites of Theorem~\ref{thm:main1} and Theorem~\ref{thm:main2}. 
As discussed in Remark~\ref{rmk:decayingpotential} we have
\begin{equation}\label{eq:excomp}
\widetilde{E_{\text{NLS}}}(\mathcal G):=\lim_{n\to \infty}\inf_{\substack{u\in H^1(\mathcal G) \\\|u\|_2^2=1,\; \operatorname{supp} u \subset \mathcal G\setminus K_n}} E_{\text{NLS}}(u, \mathcal G) \le E_{\text{NLS}}(\mathbb R). 
\end{equation}
If $\mathcal G$ is a finite metric graph we have equality in \eqref{eq:excomp} due to Lemma~\ref{lem:decayingpotential}.
In particular Corollary~\ref{cor:existence} gives a generalization of Theorem~\ref{thm:introast2016}. Indeed, in \cite{adami2016threshold} was shown that if $\mathcal G$ is finite then minimizers of \eqref{eq:exeq} exist if
\begin{equation}\label{eq:inequality}
E_{\text{NLS}}(\mathcal G) < E_{\text{NLS}}(\mathbb R).
\end{equation}
Since \eqref{eq:excomp} does not guarantee existence by Corollary~\ref{cor:existence} under the assumption \eqref{eq:inequality} one cannot necessarily extend this result to general locally finite metric graphs. But as we will see in Example~\ref{ex:unrootedtrees}, for a class of infinite tree graphs, one can show the reverse inequality 
\begin{equation}\label{eq:excomp2}
\widetilde{E_{\text{NLS}}}(\mathcal G) \ge E_{\text{NLS}}(\mathbb R).
\end{equation} 
In particular one can derive for such graphs satisfying \eqref{eq:excomp2} existence of minimizers of $E_{\text{NLS}}(\mathcal G)$ under assumption \eqref{eq:inequality}.
\end{example}

\subsection{Application: Schrödinger operators with magnetic potentials on infinite tree graphs}
In certain cases as discussed in \cite[§2.6]{berkolaiko2013introduction} the gauge transform $G$, as defined below, unitarily transforms the Schrödinger operator with magnetic potential into a Schrödinger operator without magnetic potential, and the NLS functional under gauge transform reduces to a problem without magnetic potential. We may use the results from Section \ref{sec:subsecsuff} to show existence of minimizers of the NLS functional with magnetic potential. 

For infinite tree graphs in the context of locally finite, connected metric graphs it is particularly easy to see this. In this context, let $\mathcal G$ be an infinite tree graph. Given a vertex $\mathsf v$ we can define the gauge transform $G$ radially. For any $x\in \mathcal G$, let $\gamma$ be a simple path from $\mathsf v$ to $x$ parametrized by arc length, then
\begin{equation}
G:u(x)\mapsto e^{i \int_{\operatorname{im} \gamma} M\, \mathrm d\gamma} u(x).
\end{equation}

Assume $A^{M}= (i \frac{\mathrm d}{\mathrm dx} + M)^2+V$ admits a ground state. In this particular case since
\begin{equation}
G^{-1} A^{M} G = -\Delta + V= A^{0},
\end{equation}
this is equivalent to the assertion that $A^{0}$ admits a ground state. Indeed, let $u_M$ be a ground state to $A^{M}$, then 
\begin{equation}
A^{0} G^{-1} u_0 =G^{-1} A^{M} u_M = \Sigma G^{-1} u_M
\end{equation}
and $G^{-1} u_M$ is a ground state of $A^{0}$. Then we may assume $u_0>0$ by phase invariance and the maximum principle. Then $u_M$ does not vanish anywhere. 
In particular independent of $M\in H^1+W^{1,\infty}(\mathcal G)$
\begin{equation}\label{eq:unitaryequivalence}
\begin{aligned}
\Sigma_0^{M} &= \inf_{\substack{u\in D(A^{M})\\\|u\|_{2}^2=1}}\left \langle A^{M}u, u\right \rangle= \inf_{\substack{u\in D(A^{0})\\\|u\|_{2}^2=1}}\left \langle A^{0}u, u\right \rangle= \Sigma_0\\
\Sigma_{R}^{M} &= \inf_{\substack{u\in D_R(A^{M})\\\|u\|_{2}^2=1}}\left \langle A^{M}u, u\right \rangle= \inf_{\substack{u\in D_R(A^{0})\\\|u\|_{2}^2=1}}\left \langle A^{0}u, u\right \rangle= \Sigma_R\\
\Sigma^{M}&= \lim_{R\to \infty} \Sigma_R^M = \lim_{R\to \infty} \Sigma_R= \Sigma
\end{aligned}
\end{equation} 
and in Section \ref{sec:subsecsuff} we gave sufficient conditions for $\Sigma_0 < \Sigma$.
\begin{proposition}\label{prop:app1}
Assume $\mathcal G$ is an infinite tree graph, connected and locally finite. Assume $\mathcal K$ is a bounded subgraph of $\mathcal G$ and $-\Delta+V$ admits a ground state, then the infimization problem
\begin{equation}
E_{\text{NLS}}^{(\mathcal K)}=\inf_{\substack{u\in H^1(\mathcal G)\\ \|u\|_2^2=1}} \frac{1}{2}\int_{\mathcal G} \left |\left ( i \frac{\mathrm d}{\mathrm dx} + M \right )u\right |^2+ V |u|^2\, \mathrm dx - \frac{\mu}{p} \int_{\mathcal K} |u|^p\, \mathrm dx
\end{equation}
admits a minimizer for all $\mu  >0$.
\end{proposition}
\begin{proof}
This follows immediately from Theorem \ref{thm:bigresulttt} and the unitary equivalence of the problem in absence of a magnetic potential under the gauge transform.
\end{proof}

\begin{proposition}\label{prop:app2}
Assume $\mathcal G$ is a infinite tree graph, locally finite and connected. Assume $\mathcal K$ is any unbounded subgraph and $\Sigma_0 < \Sigma$ then there exists $\hat \mu>0$, such that the infimization problem
\begin{equation}
E_{\text{NLS}}=\inf_{\substack{\phi\in H^1(\mathcal G)\\ \|\phi \|_2^2=1}} \frac{1}{2}\int_{\mathcal G} \left |\left ( i \frac{\mathrm d}{\mathrm dx} + M \right )\phi\right |^2+ V |\phi|^2\, \mathrm dx- \frac{\mu}{p}\int_{\mathcal K}  |\phi|^p\, \mathrm dx.
\end{equation}
admits a minimizer for all $\mu\in (0,\hat \mu)$.
\end{proposition}
\begin{proof}
This follows immediately from Theorem \ref{thm:bigresultt} and the unitary equivalence of the problem in absence of a magnetic potential under the gauge transform
\end{proof}

For decaying potentials in §5.5 we discussed criteria such that $\Sigma_0< \Sigma$ is satisfied. Indeed, for any given locally finite metric graph, one can construct decaying potentials in the following way:
\begin{example}
Let $\mathcal G$ be a locally finite, connected graph and $K$ a bounded, connected subgraph. Consider the higher-order Schrödinger operator $A=(-\Delta)^k+V$ with potential $V$. We define a potential $V$ a.e. via
\begin{equation}
\begin{aligned}
V\bigg |_{K_1} &\equiv -\frac{1}{2}\\
V\bigg |_{K_{2n}\setminus K_n} &\equiv -\frac{1}{2^n |K_{2n} \setminus K_n|}, \quad n\ge 2
\end{aligned}
\end{equation} 
on each ``annulus'' $K_{2n}\setminus K_n$. Then $V\in L^2 \cap L^1(\mathcal G)$,
\begin{equation}
\int_{\mathcal G} V\, \mathrm d\mu= -\sum_{n=0}^\infty \frac{1}{2^n}  <0
\end{equation} 
and by Proposition~\ref{prop:Iwantthistoend} we infer $\inf \sigma_{\text{ess}}(A)\ge 0$. In particular, if $\mathcal G$ is an infinite tree graph satisfying the volume growth assumption \eqref{eq:volumegrowthassumpt}, then the prerequisites in Proposition~\ref{thm:lastresultt1} are satisfied as well and we have
\begin{equation}
\Sigma_0 < \Sigma.
\end{equation} 
In particular Proposition \ref{prop:app1} and Proposition \ref{prop:app2} are applicable and there exists $\hat \mu >0$ such that
\begin{equation}
E_{\text{NLS}}^{(\mathcal K)}=\inf_{\substack{u\in H^1(\mathcal G)\\ \|u\|_2^2=1}} \frac{1}{2}\int_{\mathcal G} \left |\left ( i \frac{\mathrm d}{\mathrm dx} + M \right )u\right |^2+ V |u|^2\, \mathrm dx - \frac{\mu}{p} \int_{\mathcal K} |u|^p\, \mathrm dx
\end{equation}
admits a minimizer for $\mu \in (0, \hat \mu)$. If $\mathcal K\subset \mathcal G$ is precompact, then minimizers exist for~all~$\mu>0$.
\end{example}

For a certain class of infinite tree graphs we can in a similar way as in Example~\ref{ex:shouldbeputinintroduction} give an explicit $\hat \mu$ such that for $\mu\in (0,\hat \mu]$ the minimization problem $E_{NLS}$ admits a minimizer. 

\begin{example}\label{ex:unrootedtrees}
Consider an \emph{unrooted tree graph} $\mathcal G$ as considered for instance in \cite{dovetta2019nls}, i.e. there are no vertices of degree $1$ apart of vertices at infinity. Such trees in particular satisfy the \textbf{(H)}-condition formulated in \cite{adami2015nls} in the special case of finite metric graphs:
\begin{itemize}
\item[\textbf{(H)}] For every point $x\in \mathcal G$, there exist two injective curves $\gamma_1, \gamma_2:[0,+\infty) \to \mathcal G$ parametrized by arc length, with disjoint images except on a discrete set of points, and such that $\gamma_1(0)=\gamma_2(0)=x$.
\end{itemize}
and by rearrangement methods one can show for decaying potentials $V= V_2+V_\infty$ with $V_2\in L^2(\mathcal G)$ and $V_\infty\in L^\infty(\mathcal G)$  satisfying
\begin{equation}
\sup_{x\in \mathcal G\setminus K_n} |V_\infty(x)|\to 0 \qquad (n\to \infty)
\end{equation}
that
\begin{equation}
\begin{aligned}
\widetilde{\Sigma}^{(\mu)} &= \lim_{n\to \infty} \inf_{\substack{u\in H^1(\mathcal G\\\|u\|_2^2=1, \, \operatorname{supp} u\subset \mathcal G\setminus K_n}} E_{\text{NLS}}^V(u)\\
&\ge \lim_{n\to \infty} \inf_{\substack{u\in H^1(\mathcal G\\\|u\|_2^2=1, \, \operatorname{supp} u\subset \mathcal G\setminus K_n}} E_{\text{NLS}}^0(u) \ge E_{\text{NLS}}(\mathbb R),
\end{aligned}
\end{equation}
where by Remark~\ref{rmk:decayingpotential} one has equality if $\mathcal G$ contains a half line.

When $V\equiv 0$, by strictness in the rearrangement inequality one can prove nonexistence results similarly as in \cite{adami2015nls}. On the other hand, under the assumption
\begin{equation}
\Sigma_0 = \inf \sigma(-\Delta +V) <0,
\end{equation}
as discussed in Example~\ref{ex:shouldbeputinintroduction} we have thus the existence of minimizers of $E_{\text{NLS}}$ for 
\begin{equation}
\mu \in \left [ 0, \left ( \frac{\Sigma_0}{\gamma_p} \right )^{\frac{6-p}{4}} \right ]
\end{equation}
as in Example~\ref{ex:shouldbeputinintroduction}.
\end{example}

\begin{remark}\label{rmk:unrooted}
The arguments in Example~\ref{ex:unrootedtrees} can be applied to all graphs that satisfy the \textbf{(H)}-condition. One can even consider more general graphs as long they satisfy the following weaker version of the \textbf{(H)}-condition:
\begin{itemize}
\item[\textbf{(\=H)}] There exists a precompact set $K\subset \mathcal G$ such that for every point $x\in \mathcal G\setminus K$ each connected component of $\mathcal G\setminus \{x\}$ is either unbounded or contains $K$.
\end{itemize}
In particular, for each $x$ there exist two injective, simple curves $\gamma_1:[0,1]\to \mathcal G, \gamma_2:[0,+\infty) \to \mathcal G$ with disjoint images except on a discrete set of points, such that $\gamma_1(0)=x=\gamma_2(0)$ and $\gamma_1(1)\in K$. In particular, this property is satisfied for any finite noncompact graph.
\end{remark}

\begin{example}
Consider the graph consisting of two half-lines and a pendant edge joined at a single vertex (see also Figure~\ref{fig:donedone}), then the graph satisfies the \textbf{(\=H)}-condition but not the \textbf{(H)}-condition and the existence result from Example~\ref{ex:unrootedtrees} as discussed in Remark~\ref{rmk:unrooted} is still applicable.
\begin{figure}[H]
\centering
\includegraphics{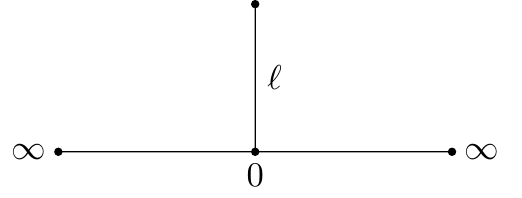}
\caption{The graph consisting of two half-lines and a pendant edge as an example of a graph that satisfies the \textbf{(\=H)}-condition but not the \textbf{(H)}-condition.}
\label{fig:donedone}
\end{figure}
\end{example}

We finish this section by proving Theorem~\ref{thm:introlastlastlast}:
\begin{proof}[Proof of Theorem~\ref{thm:introlastlastlast}]
Let $\mathcal G$ be a locally finite metric tree graph that contains at most finitely many vertices of degree $1$.  Then there exists a connected, precompact set $K\subset \mathcal G$ that contains all vertices of degree $1$ by assumption. 
Consider the set $\overline{\mathcal G}$ of points $x\in \mathcal G$, such that there exist two injective curves $\gamma_1, \gamma_2:[0,+\infty)\to \mathcal G$ parametrized by arc length, with disjoint images except on a discrete set of points, and such that $\gamma_1(0)= \gamma_2(0)=x$. In particular, if $x\in \overline{\mathcal G}$, then 
$$\operatorname{im} \gamma_1,\; \operatorname{im} \gamma_2\subset \overline{\mathcal G}.$$ 

\emph{1st Case: $\overline{\mathcal G}\neq \emptyset$.} Then by assumption 
$\mathcal G\setminus \overline{\mathcal G}$ 
contains at most finitely many connected components. Moreover of the connected components is precompact. Otherwise one could construct an injective curve  $\gamma_1:[0, +\infty)\to \mathcal G\setminus \overline{\mathcal G}$ for all $x\in \mathcal G\setminus \overline{\mathcal G}$ and since we assumed $\overline{\mathcal G}\neq \emptyset$, we can construct $\gamma_2:[0,+\infty)\to \mathcal G\setminus \overline{\mathcal G}$. This would then imply that $\mathcal G\setminus \overline{\mathcal G}$ is necessarily precompact.  Since $\mathcal G$ is a tree graph, this also implies that each connected component of $\mathcal G\setminus \overline{\mathcal G}$ contains necessarily a vertex of degree $1$.  In particular, $\mathcal G\setminus \overline{\mathcal G}$ admits at most finitely many connected components and is precompact. By construction, $\overline{\mathcal G}$ satisfies the \textbf{(H)}-condition and hence $\mathcal G$ satisfies the \textbf{(\=H)}-condition. Then as in Example~\ref{ex:unrootedtrees} we have
\begin{equation}
\begin{aligned}
\widetilde{\Sigma}^{(\mu)} &= \lim_{n\to \infty} \inf_{\substack{u\in H^1(\mathcal G\\\|u\|_2^2=1, \, \operatorname{supp} u\subset \mathcal G\setminus K_n}} E_{\text{NLS}}^V(u)\\
&\ge \lim_{n\to \infty} \inf_{\substack{u\in H^1(\mathcal G\\\|u\|_2^2=1, \, \operatorname{supp} u\subset \mathcal G\setminus K_n}} E_{\text{NLS}}^0(u) \ge E_{\text{NLS}}(\mathbb R)
\end{aligned}
\end{equation}
and we obtain existence of minimizers of $E_{\text{NLS}}$ for
\begin{equation}
\mu \in \left [ 0, \left ( \frac{\Sigma_0}{\gamma_p} \right )^{\frac{6-p}{4}} \right ].
\end{equation}

\emph{2nd Case: $\overline{\mathcal G}=\emptyset$.} In particular for each $x\in \mathcal G$ there exists only one connected component of $\mathcal G$ that contains a vertex at infinity. Assume $K$ is a precompact set that contains all vertices of degree $1$, then by assumption for any $x\in \mathcal G\setminus K$ the connected components of $\mathcal G\setminus \{x\}$ consist of a compact core graph containing all vertices of degree $1$ and a half-line. In particular, $\mathcal G$ is a finite metric graph and Example~\ref{ex:shouldbeputinintroduction} yields the existence of minimizers of $E_{\text{NLS}}$ for
\begin{equation}
\mu \in \left [ 0, \left ( \frac{\Sigma_0}{\gamma_p} \right )^{\frac{6-p}{4}} \right ].
\end{equation}
\end{proof}

\appendix 

\bibliographystyle{alpha}

\end{document}